\documentclass[onefignum,onetabnum]{siamonline250211}

\usepackage{iftex}

\usepackage[english]{babel}
\ifpdftex
    \usepackage[T1]{fontenc}

    \input{glyphtounicode}
\else
    \usepackage{fontspec}
\fi

\usepackage{xcolor}
\usepackage{graphicx}
\usepackage{tikz}

\usepackage{amssymb}
\usepackage{nicefrac}
\usepackage{textcomp}

\usepackage[shortlabels]{enumitem}
\usepackage{booktabs}
\usepackage{caption}
\usepackage{subcaption}
\usepackage[title]{appendix}

\usepackage[activate={true,nocompatibility},final]{microtype}

\newif\ifreview
\reviewfalse

\ifreview
    \definecolor{MarkRedo}{HTML}{A2142F}
    \definecolor{MarkRevise}{HTML}{D95319}
    \definecolor{MarkRead}{HTML}{0072BD}

    \newcounter{afnote}
    \newcounter{eknote}
    \newcounter{lfnote}
    \NewDocumentCommand \circlednumber {m} {%
        \tikz[baseline=(char.base)]{
            \node[shape=circle, draw, fill=MarkRead, text=white,
                  inner sep=1.5pt, font=\scriptsize\sffamily\bfseries] (char) {#1};
        }%
    }

    \NewDocumentCommand \afnote { O{note} +m } {%
        \stepcounter{afnote}%
        \textcolor{gray!60}{[#1]} \textcolor{MarkRead}{\circlednumber{\theafnote} \textbf{Alex}: #2}}
    \NewDocumentCommand \eknote { O{note} +m } {%
        \stepcounter{eknote}%
        \textcolor{gray!60}{[#1]} \textcolor{MarkRead}{\circlednumber{\theeknote} \textbf{Eva}: #2}}
     \NewDocumentCommand \lfnote { O{note} +m } {%
        \stepcounter{lfnote}%
        \textcolor{gray!60}{[#1]} \textcolor{MarkRead}{\circlednumber{\thelfnote} \textbf{Lavinia}: #2}}
    
    \newenvironment{revision}{\color{MarkRevise}}{}
\else
    \definecolor{MarkRedo}{HTML}{000000}
    \definecolor{MarkRevise}{HTML}{000000}
    \definecolor{MarkRead}{HTML}{000000}

    \NewDocumentCommand \afnote { O{note} m } {}
    \NewDocumentCommand \eknote { O{note} m } {}
    \NewDocumentCommand \lfnote { O{note} m } {}

\fi

\setlist[enumerate]{leftmargin=0.25in}
\setlist[itemize]{leftmargin=0.25in}

\NewDocumentCommand \ii {} {\mathrm{i}}
\NewDocumentCommand \coloneq {o} {\mathrel{\mathop:}\mathrel{\mkern-1.2mu}=}
\NewDocumentCommand \dx {O{x}} {\,\mathrm{d} #1}
\NewDocumentCommand \conj { m } {\bar{#1}}
\NewDocumentCommand \iffarrow { } {\Longleftrightarrow}
\NewDocumentCommand \pd {s m m} {%
    \IfBooleanTF#1{\partial_{#3} #2}{\dfrac{\partial #2}{\partial #3}}
}

\NewDocumentCommand \vect { m } { \boldsymbol{#1} }
\NewDocumentCommand \matr { m } { \boldsymbol{#1} }

\DeclareMathOperator{\sign}{sign}

\newsiamremark{remark}{Remark}
\newsiamremark{example}{Example}
\crefname{example}{Example}{Examples}

\title{
Macroscopic Multistability and Bifurcations in Theta-Neuron Networks with Distributed Delays\thanks{
Submitted to the editors on June 17th 2026.
\funding{This work was supported by UEFISCDI under grant no. ROSUA-2024-0002 (Kaslik, Fikl).}
}}
\headers{Multistability and Bifurcations in Theta-Neuron Networks}{L. Bîrdac, A. Fikl, E. Kaslik and R. Muresan}

\author {Lavinia Bîrdac\thanks{Department of Computer Science, West University of Timişoara, 300223 Timișoara, Romania (\email{lavinia.birdac91@e-uvt.ro}, \email{eva.kaslik@e-uvt.ro}, \email{raluca.muresan@e-uvt.ro})},
\and Alexandru Fikl\thanks{Institute for Advanced Environmental Research, West University of Timișoara, 300086 Timișoara, Romania (\email{alexandru.fikl@e-uvt.ro})}
\and Eva Kaslik\footnotemark[2]
\and Raluca Mureșan\footnotemark[2]}

\begin{document}

\maketitle

\begin{abstract}
We study an all-to-all coupled network of identical theta neurons with synaptic
interaction mediated by a distributed time delay. Using the Watanabe--Strogatz reduction
and passing to the thermodynamic limit under the assumption of uniformly distributed
constants of motion, we derive a single delay differential equation for the complex order
parameter. The delay is modeled by a family of delay kernels with prescribed mean delay,
allowing discrete and distributed delays to be treated in a unified framework.
The equilibria of the reduced system can be classified into two geometrically distinct
families: type~1 equilibria on the unit circle and type~2 equilibria on the real
axis. For both families, the local stability problem reduces to scalar
characteristic equations involving the Laplace--Stieltjes transform of the delay
kernel. We obtain stability criteria for admissible kernels and explicit Hopf
bifurcation conditions for the Dirac kernel, with additional comparison to weak
and strong Gamma kernels. The results show that the delay may either preserve
stability, destabilize equilibria, or produce stability switching, depending on
the equilibrium branch, parameter regime, and choice of kernel.
Numerical simulations for the discrete-delay case support the analytical results and
illustrate the corresponding phase portraits, basins of attraction, coexistence of
attractors, and delay-induced periodic dynamics.
\end{abstract}



\section{Introduction}
\label{sec:introduction}

Large populations of interacting oscillatory units provide a basic mathematical
framework for the study of collective dynamics in physics, biology, and neuroscience. In
neuronal systems, such collective dynamics include synchronization, coherent
oscillations, multistability, and transitions between quiescent and firing regimes. At
the microscopic level, these phenomena are often described by high-dimensional nonlinear
systems, while experimentally observable activity is typically macroscopic (mean firing
rates, order parameters, or population-averaged currents). Deriving reductions that
retain the relevant collective dynamics, but are tractable enough for rigorous
stability and bifurcation analysis, remains a key objective. Phase oscillator
models have been particularly effective in this context~\cite{Kuramoto1984,Strogatz2000}.

The theta neuron is a canonical phase model for type-I neuronal excitability. It can be
obtained as the normal form of a saddle-node bifurcation on an invariant circle and
provides a tractable description of spike generation in terms of a phase variable
\cite{Ermentrout1996,ErmentroutTerman2010}. Because of its analytic structure, the theta
neuron model is well-suited for exact reductions of large networks. In particular, for
globally coupled networks of $N > 3$ identical phase oscillators, the Watanabe--Strogatz
theory reduces the to three global variables and $N - 3$ constants of motion
\cite{watanabe1993integrability}. In the thermodynamic limit, and for uniformly
distributed constants of motion, this reduction is closely related to the Ott--Antonsen
ansatz \cite{OttAntonsen2008}, which has become a standard tool for deriving closed
macroscopic equations for large oscillator populations
\cite{luke2013complete,Laing2014,laing2018dynamics}. These reductions make it possible
to pass from a microscopic network description to a small number of collective variables
without relying on moment closures or perturbative approximations \cite{Montbrio2015}.

Time delays arise naturally from finite axonal propagation speeds, synaptic
transmission, dendritic processing, and other physiological mechanisms. Even when the
underlying delay-free system is low-dimensional, delays can create an
infinite-dimensional phase space and may induce oscillatory instabilities, stability
switching, and coexistence of attractors \cite{ErmentroutTerman2010,Campbell2009}. In
a physiological setting, a single discrete delay is overly idealized, since transmission times are
heterogeneous across cells and pathways, and the resulting memory effect is better modeled by a
distributed delay. Mathematically, this leads to convolution terms with delay kernels
and to characteristic equations involving the Laplace--Stieltjes transform of those
delay kernels \cite{MacDonald1978,Kuang1993,HaleVerduynLunel1993}.

Laing~\cite{laing2018dynamics} investigated finite and infinite all-to-all coupled
networks of identical theta neurons with instantaneous and delayed synaptic
interactions. Delay effects have also been analyzed in spatially extended theta
neuronal networks and delay-coupled phase-oscillator arrays~\cite{Laing2016Travelling},
as well as in small delayed theta-neuron systems, including delayed self-feedback and
pairs of delay-coupled excitable theta neurons~\cite{LaingKrauskopf2022ThetaFeedback,
LaingKrauskopf2025PairDelay}. Closely related results are available for quadratic
integrate-and-fire networks, which are equivalent to theta-neuron networks under a
standard change of variables. These include networks with synaptic delay and
distributed-delay coupling~\cite{PazoMontbrio2016,DevalleMontbrioPazo2018SynapticDelay,RatasPyragas2018DistributedDelay}.
However, the stability and bifurcation structure of
an all-to-all coupled identical theta-neuron network with distributed synaptic delay
kernels has not yet been systematically analyzed.

In particular, there are several open questions regarding how the equilibrium geometry
of the reduced macroscopic system interacts with the choice of delay kernel and how
discrete and distributed delays differ in generating delay-induced qualitative changes
in the system's dynamics. In this work, we derive the Watanabe--Strogatz reduced
macroscopic dynamics of the delayed theta-neuron network in the thermodynamic limit,
under the assumption of uniformly distributed constants of motion. The resulting complex
order-parameter equation, which is equivalent to a two-dimensional system in polar
coordinates, provides the basis for a systematic local stability and bifurcation
analysis. The delay kernel enters the characteristic equations through its
Laplace--Stieltjes transform, which unifies the treatment of discrete and distributed
delays. We apply this framework to a general class of admissible kernels and provide
specific results for the Dirac, weak Gamma, and strong Gamma kernels. The theoretical
results are supported by numerical simulations that show the resulting phase portraits,
basins of attraction, and delay-induced changes in representative parameter regimes.

The remainder of the paper is organized as follows. In~\Cref{sec:description} we
introduce the theta-neuron network with distributed delay, define the admissible kernel
families, and present the reduced macroscopic model. In~\Cref{sec:local.stability} we
describe the equilibrium equations and obtain the general characteristic equation.
The local stability and bifurcation analysis of type~1 equilibria is carried out in
\Cref{sec:stability.type1}, while the corresponding analysis for type~2 equilibria is
given in~\Cref{sec:stability.type2}. In~\Cref{sec:classification.figure} we combine
these results into a local bifurcation-based classification for the discrete-delay
case. Finally, \Cref{sec:num} presents numerical simulations illustrating the main
stability regimes and the dependence of the observed attractors on the delay. Conclusions and directions for future research are formulated in \Cref{sec:conclusions}.

\section{Description of the mathematical model}
\label{sec:description}

\subsection{Theta neuron network}
\label{sec:description.theta}

Following~\cite{luke2013complete,so2014networks,laing2018dynamics}, we study a network
of $N$ identical theta neurons all-to-all coupled through a synaptic current $I$. The
synaptic current models the time-varying excitability induced by the collective activity
of the network. In this work, we extend the model to include a distributed delay in $I$
to account for finite transmission delays between neighboring neurons.

The state of a neuron $j$ at time $t$ is given by the phase variable $\theta_j(t) \in
\mathbb{R}/ 2 \pi \mathbb{Z}$ on the unit circle. The network dynamics are given
by the following system of $N$ autonomous differential equations with distributed delay:
\begin{equation} \label{eq:sys.theta}
\dot{\theta}_j(t) =
    1 - \cos (\theta_j(t))
    + \left[1 + \cos (\theta_j(t))\right]
    \left[\eta + \kappa (l_\tau \ast I)(t)\right],
\quad j \in \{1, 2, \dots, N\},
\end{equation}
where the parameters $\eta \in \mathbb{R}$ and $\kappa \in \mathbb{R}$ represent a fixed
input current for each individual neuron and the overall coupling strength of the
network, respectively. The distributed delay enters the dynamics through the convolution
$l_\tau \ast I$ of the synaptic current $I$ with the delay kernel $l_\tau$
(see~\Cref{def:admissible.kernels}), where $\tau \ge 0$ represents the mean delay.
The input synaptic current $I$ is defined as an average of all the pulses emitted
between neurons and is given by
\begin{equation}\label{eq:synaptic.current}
I(t) \coloneq \frac{1}{N} \sum_{j = 1}^N \left[1 - \cos (\theta_j(t))\right]^2.
\end{equation}

\subsection{Admissible delay kernels}
\label{sec:description.admissible.kernels}

In~\eqref{eq:sys.theta}, the delayed synaptic input is described by a family of delay
kernels \(l_\tau\), where \(\tau\ge0\) denotes the mean delay. To include cases such as
the discrete Dirac delay kernel $\delta_\tau$, we consider a general class of functions
and state the convolution in a standard measure form. For a bounded Borel function
\(f\), we define
\[
(l_\tau \ast f)(t)
\coloneq
\int_{[0,\infty)} f(t-s) \dx[l_\tau(s)],
\]
which simplifies to a more common form for absolutely continuous kernels
(see~\Cref{rem:admissible.density.form}). The case \(\tau=0\) represents the delay-free
equation and is given by $l_0 \equiv \delta_0$ and $(l_0\ast f)(t) = f(t)$.

In the remainder of this paper, we work with delay kernels obtained by rescaling a reference
probability measure with unit mean. The exponential moment condition below ensures that
the associated Laplace--Stieltjes transform is well-defined (see, for example, \cite[Chapter~1]{HaleVerduynLunel1993}).

\begin{definition}[Admissible delay kernel families]
\label{def:admissible.kernels}
Let \(\mathcal M([0,\infty))\) denote the set of Borel probability measures on
\([0,\infty)\). A family of delay kernels
\(\{l_\tau\}_{\tau>0}\subset \mathcal M([0,\infty))\) is called
\emph{admissible} if there exist a Borel probability measure
\(\mu\in\mathcal M([0,\infty))\) and a constant \(\sigma>0\) such that
\[
\int_{[0,\infty)} u \dx[\mu(u)] = 1,
\qquad
\int_{[0,\infty)} e^{\sigma u} \dx[\mu(u)] < \infty,
\]
and, for every \(\tau>0\) and every bounded Borel function
\(\varphi:[0,\infty)\to\mathbb C\),
\[
\int_{[0,\infty)} \varphi(s) \dx[l_\tau(s)]
=
\int_{[0,\infty)} \varphi(\tau u) \dx[\mu(u)].
\]
Equivalently, \(l_\tau\) is the push-forward of \(\mu\) under the map
\(u\mapsto \tau u\). In particular, for $\varphi(u) \equiv u$,
\[
\int_{[0,\infty)} s \dx[l_\tau(s)] =
\int_{[0, \infty)} \tau u \dx[\mu(u)] = \tau,
\]
so that \(l_\tau\) has mean delay \(\tau\).
\end{definition}

\begin{remark}
\label{rem:admissible.density.form}
If the reference measure \(\mu\) is absolutely continuous with respect to the
Lebesgue measure, then \(\mathrm{d}\mu(u)=h(u)\,\mathrm{d}u\), where \(h\ge0\), and the
conditions in \Cref{def:admissible.kernels} become
\[
\int_0^\infty h(u) \dx[u]=1,
\qquad
\int_0^\infty u h(u) \dx[u]=1,
\qquad
\int_0^\infty e^{\sigma u}h(u) \dx[u]<\infty .
\]
In this case, \(l_\tau\) is also absolutely continuous and has density
\[
s\mapsto \frac{1}{\tau}h\!\left(\frac{s}{\tau}\right).
\]
The discrete-delay case is recovered by taking \(\mu=\delta_1\), where $\delta_1$ is the
Dirac delay kernel concentrated at $u = 1$. Then
\(l_\tau=\delta_\tau\), and
\(
(l_\tau\ast f)(t)=f(t-\tau).
\)
\end{remark}

\begin{remark}
\label{rem:laplace.admissible}
The Laplace--Stieltjes transform of the reference measure \(\mu\) is
\[
L(\xi)
\coloneq
\int_{[0,\infty)} e^{-\xi u} \dx[\mu(u)],
\qquad
\Re \xi>-\sigma .
\]
For an admissible family generated by \(\mu\), the Laplace--Stieltjes transform of
\(l_\tau\) is
\[
\widehat l_\tau(s)
\coloneq
\int_{[0,\infty)} e^{-s r} \dx[l_\tau(r)]
=
L(\tau s),
\]
whenever \(\Re s>-\sigma/\tau\), for
\(\tau>0\). For \(\tau=0\), one has \(\widehat l_0(s)=1\).

The exponential moment assumption implies
that \(L\) is holomorphic in the half-plane \(\{\Re \xi>-\sigma\}\). Moreover,
\[
L(0)=1,
\qquad
L'(0)=-1,
\qquad
|L(\xi)|\le1 \quad \text{for } \Re\xi\ge0.
\]
\end{remark}

The following regularity condition on the phase of the Laplace--Stieltjes transform
is required for the Hopf bifurcation analysis (see~\Cref{sec:type1.hopf} and~\Cref{sec:type2.hopf})
to ensure that the phase of $L(\ii \nu)$ can be continuously tracked.

\begin{definition}[Phase-regular admissible kernels]
\label{def:phase.regular.kernels}
An admissible delay kernel family is called \emph{phase-regular} if
\(L(\ii\nu)\neq0\) for all \(\nu\ge0\) and if there exist continuous functions
\(r:[0,\infty)\to(0,\infty)\) and
\(a:[0,\infty)\to[0,\infty)\) such that
\[
L(\ii\nu)=r(\nu)e^{-\ii a(\nu)},
\qquad \nu\ge0,
\]
where \(r(0) = 1\), \(a(0)=0\), and \(a\) is strictly increasing on \((0,\infty)\).
\end{definition}

The Dirac, weak Gamma, and strong Gamma kernels considered in the following sections are
admissible and phase-regular (shown by direct computation). Their transforms and phase
functions are specified explicitly in~\Cref{sec:type1.hopf}.

\subsection{The reduced model}
\label{sec:description.reduced.model}

The $N$-dimensional theta neuron network~\eqref{eq:sys.theta} admits a lower-dimensional
reduction, under the assumptions stated below. First, the Watanabe--Strogatz
transformation~\cite{watanabe,watanabe1993integrability,Birdac2022} reduces the system to three
global variables and $N - 3$ constants of motion.
Assuming uniformly distributed constants of motion and taking the thermodynamic limit $N
\to \infty$ (equivalent to the Ott--Antonsen reduction~\cite{laing2018dynamics}), we
obtain a single delayed differential equation for the complex order parameter $z(t)$:
\begin{equation} \label{eq:complex}
\dot{z} =
    \ii \left[\eta + \kappa (l_\tau \ast I) + 1\right] z
    + \ii \left[\eta + \kappa (l_\tau \ast I) - 1\right] \frac{1 + z^2}{2},
\end{equation}
where:
\begin{equation}\label{eq:i.inf}
I = I(z) = \frac{3}{2}
    - (z + \bar{z})
    + \frac{(z^2 + \bar{z}^2)}{4}.
\end{equation}

For the subsequent analysis, we express the complex equation in polar coordinates $z =
\rho e^{\ii \phi}$. This results in the following real two-dimensional system of
equations
\begin{equation} \label{eq:sys.2d}
\begin{cases}
\dot\rho
    = \dfrac{1 - \rho^{2}}{2} \left(\eta + \kappa(l_\tau \ast I) - 1\right) \sin (\phi), \\
\dot\phi
    = (\eta + \kappa (l_\tau \ast I) - 1) \left[\dfrac{1 + \rho^2}{2\rho} \cos (\phi) + 1\right] + 2,
\end{cases}
\end{equation}
where
\begin{equation}\label{eq:i.rho.phi}
I = I(\rho, \phi) = \frac{3}{2} - 2 \rho \cos \phi + \frac{1}{2} \rho^2 \cos(2 \phi).
\end{equation}

\section{Local stability and bifurcation analysis: general considerations}
\label{sec:local.stability}

\subsection{Equilibria}
\label{sec:local.stability.equilibria}

From~\eqref{eq:complex}, the equilibrium condition $\dot{z} = 0$ is equivalent to
\begin{equation}\label{eq:equilibrium}
F_\kappa(z) \coloneq \left(\frac{1-z}{1+z}\right)^2 - \kappa I(z) = \eta.
\end{equation}

As $I(z) \in \mathbb{R}$, for any $z \in \mathbb{C}$, it follows that the fraction on
the left-hand side of~\eqref{eq:equilibrium} must also be real. Therefore, $(1 - z) / (1
+ z)$ must be either real or purely imaginary, which corresponds to $z \in \mathbb{R}$
or $|z| = 1$, respectively. However, note that $z=-1$ is not an equilibrium point and
$z=1$ only occurs when $\eta=0$. As in~\cite{Birdac2022}, we consider the following two
types of equilibria that arise as solutions to $F_\kappa(z^\star) = \eta$.

\paragraph{Type 1 equilibria (on the unit circle).} In this first case, we seek
solutions on the unit circle $z^\star= e^{\ii \phi^\star}$, with $\phi^\star \in (-\pi,
\pi)$. The equilibrium equation~\eqref{eq:equilibrium} then reduces to finding the real
roots of the polynomial $P_1(\cos \phi^\star)$:
\begin{equation}\label{eq:P1.type1}
P_1(u) \coloneq \kappa u^3 - \kappa u^2 + (\eta - \kappa - 1) u + (\eta + \kappa + 1).
\end{equation}

\begin{remark}
We note that type 1 equilibria occur in conjugate pairs of the form $e^{\pm \ii \phi^\star}$.
\end{remark}

\paragraph{Type 2 equilibria (on the real axis).} In this second case, we seek solutions
on the real line $z^\star \in (-1, 1)$. The equilibrium
equation~\eqref{eq:equilibrium} then reduces to finding the real roots of the polynomial
\begin{equation} \label{eq:P2.type2}
P_2(u)\coloneq
\frac{\kappa}{2} u^4
- \kappa u^3
+ (\eta -2 \kappa - 1) u^2
+ (2 \eta + \kappa + 2) u
+ \left(\eta + \frac{3\kappa}{2} - 1\right).
\end{equation}

\begin{remark}
The point \(z^\star=1\) is a non-generic intersection of both classes. For \(\eta=0\) in
the delay-free system, it represents a degenerate center~\cite{llibre2016limit} and
has been analyzed in detail in~\cite{Birdac2022}. We, hereafter, exclude this singular
point from the analysis below.
\end{remark}

\subsection{Linearization and characteristic equation}
\label{sec:local.stability.linearization}

Linearizing the system~\eqref{eq:sys.2d} about an arbitrary equilibrium point $(\rho^{\star},
\phi^{\star})$ and letting $\vect{x} \coloneq (\tilde{\rho}, \tilde{\phi}) \coloneq (\rho
- \rho^\star, \phi - \phi^\star)$ denote the small perturbations, we obtain:
\begin{equation} \label{eq:matrix.equation}
\dot{\vect{x}} = \matr{A} \vect{x} + \matr{B}(l_\tau\ast \vect{x}),
\end{equation}
where
\[
\begin{aligned}
\matr{A} & \coloneq (\eta + \kappa I^{\star} - 1)
\begin{pmatrix}
-\rho^{\star} \sin(\phi^{\star}) &
\dfrac{1 - \rho^{\star 2}}{2} \cos(\phi^{\star})
\\
\dfrac{\rho^{\star 2} - 1}{2 \rho^{\star 2}} \cos(\phi^{\star}) &
-\dfrac{1 + \rho^{\star 2}}{2\rho^{\star}} \sin(\phi^{\star})
\end{pmatrix} \\
\matr{B} & \coloneq \kappa
\begin{pmatrix}
\displaystyle
\frac{1 - \rho^{\star 2}}{2} \sin(\phi^{\star})
I_\rho^\star&
\displaystyle
\frac{1 - \rho^{\star 2}}{2} \sin(\phi^{\star}) I_\phi^\star
\\[10pt]
\displaystyle
\left( 1 + \frac{1 + \rho^{\star 2}}{2\rho^{\star}} \cos(\phi^{\star}) \right)
I_\rho^\star &
\displaystyle
\left( 1 + \frac{1 + \rho^{\star 2}}{2\rho^{\star}} \cos(\phi^{\star}) \right)
I_\phi^\star\end{pmatrix} \\
& =
\kappa\left( 1 + \frac{1 + \rho^{\star 2}}{2\rho^{\star}} \cos(\phi^{\star}) \right)
\begin{pmatrix}
0  & 0 \\[10pt]
I_\rho^\star & I_\phi^\star
\end{pmatrix},
\end{aligned}
\]
and
\[
\begin{aligned}
I^\star & \coloneq
    \frac{3}{2} - 2 \rho^\star \cos \phi^\star + \frac{1}{2} \rho^{\star 2} \cos(2 \phi^\star),\\
I_\rho^\star & \coloneq
    \pd{I}{\rho}(\rho^\star, \phi^\star) =
    -2 \cos(\phi^\star) + \rho^\star \cos(2\phi^\star),\\
I_\phi^\star & \coloneq
    \pd{I}{\phi}(\rho^\star, \phi^\star) =
    2\rho^\star \sin(\phi^\star) - \rho^{\star 2}\sin(2\phi^\star).
\end{aligned}
\]

The structure of $\matr{B}$ simplifies significantly because every equilibrium point
satisfies either $\rho^\star=1$ or $\sin(\phi^\star)=0$. Then, substituting the exponential
ansatz $\vect{x}(t) = \vect{v} e^{s t}$, where $\vect{v} \in \mathbb{R}^2 \setminus \{0\}$,
into~\eqref{eq:matrix.equation} yields
\[
\bigl(s \matr{I}_2 - \matr{A} - \widehat l_\tau(s) \matr{B} \bigr) \vect{v} = 0.
\]

Hence, based on \Cref{rem:laplace.admissible}, the characteristic equation associated with \eqref{eq:matrix.equation} is
\begin{equation}\label{eq:general.characteristic}
\det(s \matr{I}_2 - \matr{A} - L(\tau s) \matr{B}) = 0.
\end{equation}

The characteristic equation obtained above will be adapted in
\Cref{sec:stability.type1,sec:stability.type2} to the two equilibrium families.
In both cases, the determinant factorizes in such a way that the delay-dependent
part of the characteristic equation is governed by a scalar factor
involving the Laplace--Stieltjes transform of the delay kernel. As the same
root-crossing arguments are used repeatedly for the type~1 and type~2 equilibrium branches,
we present them in \Cref{prop:general}, with the proof given in
Appendix~\ref{ax:prop.general.proof}. The first two statements only
use the general properties of admissible kernels stated in
\Cref{rem:laplace.admissible}, while the last statement is specialized to the
Dirac kernel and gives the explicit imaginary-axis crossings together with their
transversality signs. General background on characteristic roots of delay differential
equations and Hopf bifurcation for functional differential equations can be found
in \cite[Chapters~1--2]{HaleVerduynLunel1993} and
\cite[Chapter~3]{HassardKazarinoffWan1981}.

\begin{proposition}\label{prop:general}
Let $p \in \{1, 2, \dots\}$, $a, b \in \mathbb{R}$, with $b \neq 0$. Consider the
characteristic equation
\begin{equation}\label{eq:char.gen}
    \Delta_p(s,\tau)\coloneq s^p+a-bL(\tau s)=0,
\end{equation}
where $L$ satisfies the properties from~\Cref{rem:laplace.admissible}. Then:
\begin{enumerate}
\item[i.] \label{prop:general.itemi}
For any $\tau \geq 0$, $s = 0$ is a root of \eqref{eq:char.gen} if and only if $a = b$.
Furthermore, the root $s = 0$ is simple if and only if (a) $p = 1$ and $1 + b \tau \neq
0$ or if (b) $p \geq 2$ and $\tau > 0$.

\item[ii.] \label{prop:general.itemii}
For any $\tau \geq 0$ and $a < b$, equation \eqref{eq:char.gen} has at least one real
positive root.

\item[iii.] \label{prop:general.itemiii}
(Dirac kernel) For any $a > b$ and $L(s) = e^{-s}$, we have that

\begin{itemize}
\item
If $p = 1$, equation~\eqref{eq:char.gen} has a pair of complex conjugated roots $\pm \ii
\omega$, with $\omega>0$, if and only if $|b| > |a|$ and $\omega = \sqrt{b^2 - a^2}$.
The corresponding critical delays are
\begin{equation} \label{eq:general.critical.delays.1}
\tau_n = \frac{\arccos(a / b) + 2 n \pi}{\sqrt{b^2 - a^2}},
\qquad n \in \{0, 1, \dots\}.
\end{equation}

The roots \(\pm \ii \omega\) are simple and the following transversality condition holds:
\begin{equation} \label{eq:transversality.type1}
\Re s'(\tau_n) > 0, \qquad n \in \{0, 1, \dots\}.
\end{equation}

\item
If $p = 2$, equation~\eqref{eq:char.gen} has a pair of complex conjugated roots $\pm \ii
\omega$, with $\omega > 0$, if and only if there exists $n \ge 1$ such that
$a - (-1)^n b > 0$, in which case $\omega = \omega_n = \sqrt{a - (-1)^n b}$.
The corresponding critical delays are
\begin{equation} \label{eq:general.critical.delays.2}
\tau_n = \frac{n \pi}{\sqrt{a - (-1)^n b}},
\qquad n \in \{1, 2, \dots\}.
\end{equation}

The roots \(\pm \ii \omega\) are simple, and the following transversality condition holds:
\begin{equation} \label{eq:transversality.type2}
\sign \bigl(\Re s'(\tau_n)\bigr) = (-1)^{n + 1} \sign(b),
\qquad \forall n \in \{1, 2, \dots\}.
\end{equation}
\end{itemize}
\end{enumerate}
\end{proposition}

\section{Local stability and bifurcation analysis of type 1 equilibria}
\label{sec:stability.type1}

For an arbitrary type 1 equilibrium point, the matrices from~\eqref{eq:matrix.equation} simplify to
\[
\begin{aligned}
\matr{A} & = -(\eta + \kappa I^{\star} - 1) \sin(\phi^{\star})
\begin{pmatrix}
1 & 0 \\
0 & 1
\end{pmatrix}, \\
\matr{B} & = \kappa (1 + \cos(\phi^{\star}))
\begin{pmatrix}
0 & 0 \\
-2 \cos(\phi^{\star}) + \cos(2 \phi^{\star}) &
2 \sin(\phi^{\star}) - \sin(2 \phi^{\star})
\end{pmatrix},
\end{aligned}
\]
and the characteristic equation~\eqref{eq:general.characteristic} becomes:
\begin{equation}
[
    s + (\eta + \kappa I^{\star} - 1) \sin(\phi^{\star})
][
    s + (\eta + \kappa I^{\star} - 1) \sin(\phi^{\star})
    - 2 \kappa \sin^3(\phi^{\star}) L(\tau s)
] = 0.
\end{equation}

The first factor of the characteristic equation gives the $s_1$ root directly:
\[
s_1 \equiv s_1(\phi^\star)
    \coloneq -(\eta + \kappa I^{\star} - 1) \sin(\phi^{\star})
    = 2 \tan \left(\frac{\phi^{\star}}{2}\right),
\]
where the last equality is derived from the equilibrium condition $\dot{\phi} = 0$
in~\eqref{eq:sys.2d}. The remaining roots are governed by the transcendental equation
(see also~\Cref{prop:general})
\begin{equation} \label{eq:char.type1}
\Delta_1(s,\tau)\coloneq s + a_1 - b_1 L(\tau s) = 0,
\end{equation}
where
\begin{equation}\label{eq:a1.b1}
a_1 \equiv a_1(\phi^\star) \coloneq -2 \tan \left(\frac{\phi^{\star}}{2}\right)
\qquad \text{and} \qquad
b_1 \equiv b_1(\kappa, \phi^\star) \coloneq 2 \kappa \sin^3(\phi^{\star}).
\end{equation}

\begin{remark}\label{rem:instab.type1}
Type 1 equilibria $z^\star = e^{\ii \phi^\star}$ on the upper semicircle (with
$\phi^{\star} \in [0, \pi]$) are unstable, regardless of the admissible delay kernel
$l_\tau$. Indeed, in this case, we always have $s_1 > 0$. Therefore, in what follows, we
focus on investigating the stability properties of type 1 equilibria on the lower
semicircle, i.e. with $\phi^{\star} \in [-\pi, 0]$.
\end{remark}

\subsection{Classification of type 1 equilibria in the delay-free case}

The number of type 1 equilibria is determined by two curves in the $(\kappa, \eta)$
plane: a saddle-node curve $\Gamma^{(1)}_{\mathrm{SN}}$, along which two equilibria collide and
annihilate, and the line $\eta = 0$, where an additional equilibrium point attains the
boundary value $z^\star = 1$. These two curves partition the parameter plane into
regions of constant equilibrium count, as seen in~\Cref{fig:number.type.1}. The
following proposition completely classifies the type 1 equilibria in the delay-free
case.

\begin{figure}[ht!]
\centering
\includegraphics[width=0.55\linewidth]{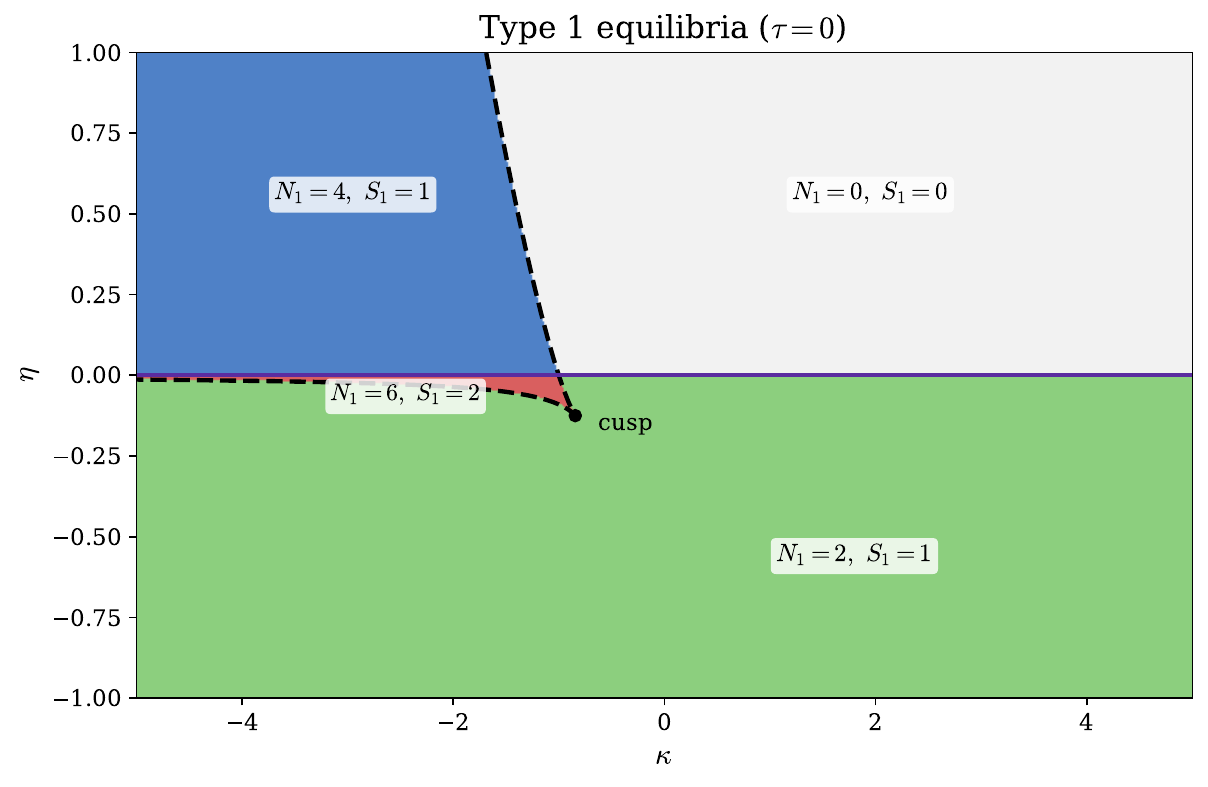}

\caption{Number of type 1 equilibria ($N_1$) and number of asymptotically stable type 1
equilibria ($S_1$) in the delay-free case (see~\Cref{prop:type1.regions}) in
the regions of the $(\kappa,\eta)$ parameter plane delimited by the saddle-node curve
$\Gamma_{\mathrm{SN}}^{(1)}$ (dashed) and the boundary line $\Sigma$ (horizontal axis).}
\label{fig:number.type.1}
\end{figure}

\begin{proposition}[Number and stability of type 1 equilibria when $\tau = 0$]
\label{prop:type1.regions}
Consider the delay-free system~\eqref{eq:sys.2d}. Let $\Gamma_{\mathrm{SN}}^{(1)}$ be the
saddle-node bifurcation curve defined parametrically by
\begin{equation}\label{eq:bif.curve.SN.type1}
\Gamma_{\mathrm{SN}}^{(1)}:
\qquad
\kappa = -\frac{1}{(1 - u)(1 + u)^2},
\qquad
\eta = -\frac{u (1 - u)}{(1 + u)^2},
\qquad u \in (-1, 1)
\end{equation}
and $\Sigma \coloneq \{(\kappa, \eta) : \eta = 0\}$ be a boundary line. For $\kappa <
-\nicefrac{27}{32}$, the saddle-node curve $\Gamma_{\mathrm{SN}}^{(1)}$ given by
\eqref{eq:bif.curve.SN.type1} has two branches $\eta_{\pm}(\kappa)$ with
\(\eta_+(\kappa) > \eta_-(\kappa)\). These branches converge at the cusp point
$(\kappa, \eta) = \left(-\nicefrac{27}{32}, -\nicefrac{1}{8}\right)$ (see~\Cref{fig:number.type.1}).

Let $N_1$ be the number of type 1 equilibria and $S_1$ the number of asymptotically
stable type 1 equilibria. Then, in each connected component of the complement of
$\Gamma_{\mathrm{SN}}^{(1)} \cup \Sigma$, we have:
\[
(N_1,S_1) =
\begin{cases}
(0,0), & \eta>0 \text{ and either } \kappa\ge -1
         \text{ or } (\kappa<-1 \text{ and } \eta>\eta_+(\kappa)),\\
(4,1), & \kappa<-1 \text{ and } 0<\eta<\eta_+(\kappa),\\
(6,2), & \kappa<-\frac{27}{32}
         \text{ and } \eta_-(\kappa)<\eta<\min\{0,\eta_+(\kappa)\},\\
(2,1), & \text{otherwise.}
\end{cases}
\]
\end{proposition}

\begin{proof}
The saddle-node bifurcation curve is determined using \Cref{prop:general}(i). First, we
have that $s = 0$ is a solution of the characteristic equation~\eqref{eq:char.type1} if and
only if $a_1(\phi^\star) = b_1(\kappa, \phi^\star)$, as defined in~\eqref{eq:a1.b1}.
Solving for $\kappa$ and replacing $u \equiv \cos \phi^\star$ gives the first parametric
equation. Second, the expression for $\eta$ is determined from the equilibrium
equation $P_1(u) = 0$ from~\eqref{eq:P1.type1}.

To determine the number of roots, we start by looking at the boundary line $\Sigma$.
From~\eqref{eq:P1.type1}, we note that $P_1(1) = 2 \eta$, so the boundary point $u = 1$
is only a root of the equilibrium equation if $\eta = 0$, i.e. $(\kappa, \eta) \in \Sigma$.
By definition, $u = 1$ corresponds to an equilibrium point $\phi^\star = 0$ and $z^\star = 1$.
Performing an expansion of $\cos \phi^\star$ for small $\phi^\star$ (i.e. around $u = 1$) gives
\[
P_1(\cos \phi^\star) = 2\eta + \frac{\phi^{\star 2}}{2} + O(\eta \phi^{\star 2}+\phi^{\star 4}).
\]

Consequently, near $\Sigma$, we have that $\phi^{\star 2} = -4 \eta + O(\eta^2)$, yielding a
pair of conjugate type 1 equilibria for $\eta < 0$ and none for $\eta > 0$. Then,
$\Gamma^{(1)}_{\mathrm{SN}}$ accounts for the remaining changes in the root count of $P_1$ in
$(-1, 1)$. To make this explicit for type 1 equilibria, we can
write~\eqref{eq:equilibrium} and \eqref{eq:P1.type1} as
\[
F_\kappa(u) = -\kappa (1 - u)^2 - \frac{1 - u}{1 + u}
\quad \text{and} \quad
P_1(u) = (1 + u) \bigl(\eta - F_\kappa(u)\bigr).
\]

Since $1 + u > 0$ on $(-1, 1)$, the roots of $P_1$ coincide with those of the second
factor. The function $F_\kappa$ is increasing for $\kappa \ge -\nicefrac{27}{32}$. For
$\kappa<-\nicefrac{27}{32}$, it has exactly one local maximum and one local minimum,
whose loci are the two branches $\eta_+(\kappa)$ and $\eta_-(\kappa)$ of
$\Gamma_{\mathrm{SN}}^{(1)}$. Hence, $\Gamma_{\mathrm{SN}}^{(1)} \cup \Sigma$ partitions the parameter
plane into regions where the number of roots of $P_1$ in $(-1,1)$ is constant, and
consequently, the values of $N_1$ given in the statement of the proposition follow.

For the stability analysis, let $u\in(-1,1)$ be a simple root of $P_1$. The
corresponding type 1 equilibria satisfy $\phi^\star_\pm=\pm\arccos u$. In the delay-free
case of~\eqref{eq:char.type1}, the characteristic roots are
\[
s_1=2\tan\frac{\phi^\star}{2}
\quad \text{and} \quad
s_2=2\tan\frac{\phi^\star}{2}+2\kappa\sin^3\phi^\star.
\]
The upper equilibrium $\phi^\star_+$ is unstable for all $u \in (-1, 1)$, since $s_1
= 2 \tan \nicefrac{\phi^\star_+}{2} > 0$. For the lower equilibrium $\phi^\star_-$, we
have
\[
s_1 < 0
\quad \text{and} \quad
s_2 = \sqrt{1-u^2}\,P_1'(u),
\]
so it becomes asymptotically stable if and only if $P_1'(u)<0$. At the three consecutive
simple roots $u_1 < u_2 < u_3$ of $P_1(u) = \eta$, the signs of $P_1'$ alternate by
Rolle's theorem. Consequently, in the six-equilibria region exactly two lower type 1
equilibria are asymptotically stable, while in the two- and four-equilibria regions
exactly one such equilibrium is asymptotically stable.
\end{proof}

\subsection{Effect of distributed delays on type 1 equilibria on the lower semicircle}

When $\tau>0$, the reduced characteristic equation~\eqref{eq:char.type1} determines the
stability properties of type 1 equilibria with $\phi^\star < 0$. The following proposition
characterizes the regions where the stability of the type 1 equilibria does not change
with the introduction of the delay.

\begin{proposition}[Delay-independent results for type 1 equilibria] \label{prop:stab.type1}
Assume that $\kappa < 0$ and $\phi^\star\in(-\pi, 0)$. Then,
\begin{enumerate}
\item[i.]
If $z^\star$ is asymptotically stable for the delay-free system, then
$z^\star$ remains asymptotically stable for every admissible delay kernel $l_\tau$, with $\tau > 0$.

\item[ii.]
If $z^\star$ is unstable for the delay-free system, then $z^\star$ remains unstable
for every admissible delay kernel $l_\tau$, with $\tau > 0$.
\end{enumerate}
\end{proposition}

\begin{proof}
For $\kappa < 0$ and $\phi^\star < 0$, we have that the coefficients of the reduced
characteristic equation \eqref{eq:char.type1} satisfy $a_1 > 0$ and $b_1 > 0$. We also note
that, for $\tau = 0$, the reduced equation becomes $s + a_1 - b_1 = 0$. Therefore, the
equilibrium $z^\star$ is asymptotically stable if and only if $a_1 > b_1$.

Assuming by contradiction that, for some $\tau>0$, \eqref{eq:char.type1} has a root $s$ with
$\Re s \ge 0$, we have
\[
|s+a_1| = b_1 |L(\tau s)| \le b_1,
\]
by the property of admissible kernels from~\Cref{rem:laplace.admissible}. On the other
hand, we also have that
\[
|s+a_1| \ge \Re(s+a_1) = \Re s + a_1 \ge a_1.
\]

This contradicts the stability condition $a_1 > b_1$. Consequently, all characteristic
roots satisfy $\Re s < 0$, and the equilibrium is asymptotically stable for every $\tau
\ge 0$. On the other hand, instability for $\tau = 0$ implies $a_1 < b_1$. Therefore,
based on~\Cref{prop:general}(ii), which guarantees the existence of a root with positive
real part for any $\tau \geq 0$, we obtain that the equilibrium remains unstable.
\end{proof}

\Cref{rem:instab.type1} and \Cref{prop:stab.type1} show that, when $\kappa < 0$, the
stability properties of all type 1 equilibria are independent of the delay. In addition,
when $\kappa>0$ and $\eta>0$, there are no type 1 equilibria. Therefore, distributed
delays can only destabilize type 1 equilibria on the lower semicircle if and only if
$\kappa > 0$ and $\eta < 0$. In the remainder of this section, we restrict our attention
to the fourth quadrant of the $(\kappa,\eta)$-parameter plane.

In this quadrant, \Cref{prop:type1.regions} guarantees the existence of a unique type 1
equilibrium on the lower semicircle, which is asymptotically stable in the delay-free case.
We denote this unique equilibrium by $z_a^\star$ and analyze all its possible bifurcations
in the next sections.

\subsection{Hopf curves at \texorpdfstring{$z_a^\star$}{z_a} in the fourth quadrant of the \texorpdfstring{$(\kappa, \eta)$}{(kappa, eta)}-plane}
\label{sec:type1.hopf}

To analyze the Hopf curves in the fourth quadrant, we consider an admissible delay
kernel $l_\tau$ (in the sense of~\Cref{def:admissible.kernels}), such that the polar
representation  $L(\ii \nu) = r(\nu)e^{-\ii a(\nu)}$ from~\Cref{def:phase.regular.kernels}
holds. Then, we look for purely imaginary solutions of the characteristic equation of
the form $s = \ii \frac{\nu}{\tau}$ with $\nu > 0$. Therefore, \eqref{eq:char.type1}
becomes
\[
\ii \frac{\nu}{\tau}
- 2 \tan \left(\frac{\phi^{\star}}{2}\right)
- 2 \kappa \sin^3(\phi^{\star}) r(\nu) e^{-\ii a(\nu)} = 0.
\]

By separating the real and imaginary parts, we have that
\begin{equation} \label{eq:real.imag.parts}
\begin{cases}
\kappa \sin^3(\phi^{\star}) r(\nu) \cos(a(\nu))
    = -\tan \left(\dfrac{\phi^{\star}}{2}\right), \\
\kappa \sin^3(\phi^{\star}) r(\nu) \sin(a(\nu))
    = -\dfrac{\nu}{2\tau}.
\end{cases}
\end{equation}

Assuming that $\cos(a(\nu)) \ne 0$, we divide the two equations
in~\eqref{eq:real.imag.parts} and find that
\begin{equation} \label{eq:omega}
\frac{\tan(a(\nu))}{\nu}
= \frac{1}{2\tau} \cot \left(\frac{\phi^{\star}}{2}\right)
\implies
\phi^{\star} = 2 \arctan \left(\frac{\nu}{2\tau} \cot (a(\nu))\right).
\end{equation}

For $\phi^{\star} \in (-\pi, 0)$ we have that $\tan(\phi^\star / 2) < 0$. Since
$\nu > 0$ and $\tau > 0$, \eqref{eq:omega} implies that $\cot (\phi^\star / 2) < 0$,
which restricts $a(\nu) \in (\frac{\pi}{2} + n \pi, \pi + n \pi)$, for $n \in \{0, 1,
\dots\}$. The continuous function $a$ is strictly increasing and positive for all
admissible kernels (see~\Cref{def:phase.regular.kernels}), therefore it is invertible. Its
inverse gives the following intervals
\[
I_n \coloneq a^{-1} \left[\left(\frac{\pi}{2} + n \pi, \pi + n \pi\right)\right].
\]

Using \eqref{eq:real.imag.parts} and the equilibrium equation~\eqref{eq:P1.type1}, we obtain
that the corresponding curves of purely imaginary roots can be represented parametrically by:
\begin{equation}\label{eq:kappa.eta.system}
\Gamma^{(1)}_{\mathrm{H},n}(l_\tau): \quad
\begin{cases}
\kappa =
    -\dfrac{(4 \tau^2 \tan^2(a(\nu))+\nu^2)^3}
           {128 \tau^4 \nu^2 \tan^3(a(\nu)) r(\nu) \sin(a(\nu))}, \\
\eta =
    \dfrac{\nu^2}{4 \tau^2 \tan^2 (a(\nu))} \left(
        \dfrac{4 \tau^2 \tan^2(a(\nu)) + \nu^2}
              {8 \tau^2 \tan(a(\nu)) r(\nu) \sin(a(\nu))}
        - 1
    \right),
\end{cases} \qquad \nu\in I_n.
\end{equation}

Next, we specialize to some common delay kernels used in the literature.

\begin{example}[Weak Gamma kernel] \label{ex:weak.gamma.type1}
For the kernel \(l_\tau(t)=\tau^{-1} \exp(-\frac{t}{\tau})\), we have that
\[
\begin{cases}
r(\nu) = \dfrac{1}{\sqrt{1+\nu^2}}, \\
a(\nu) = \arctan(\nu).
\end{cases}
\]

Here, $a(\nu) \in \left(0, \frac{\pi}{2}\right)$, for all $\nu$. It follows that $I_n =
\emptyset$, for all $n \in \{0, 1, \dots\}$. Hence, \eqref{eq:char.type1} has no purely
imaginary roots of the form $s = \ii \frac{\nu}{\tau}$ and no Hopf bifurcation occurs.
This is consistent with the findings from \cite{laing2018dynamics}.
\end{example}

\begin{example}[Strong Gamma kernel] \label{ex:strong.gamma.type1}
For the kernel \(l_\tau(t)=4 \tau^{-2} t \exp(-\frac{2 t} {\tau})\), we have that
\begin{equation*}
\begin{cases}
r(\nu) = \dfrac{4}{4 + \nu^2}, \\
a(\nu) = 2 \arctan \left(\dfrac{\nu}{2}\right).
\end{cases}
\end{equation*}

We find that $I_0 = (2, \infty)$, since $a(\nu) = \nicefrac{\pi}{2}$ when $\nu = 2$ and
$a$ is increasing, and $I_n = \emptyset$ for all $n \geq 1$. Replacing $r(\nu)$ and
$a(\nu)$ in system \eqref{eq:kappa.eta.system} yields the parametric equation of the
unique candidate curve:
\begin{equation} \label{eq:kappa.eta.strong.gamma}
\Gamma_{\mathrm{H}, 0}^{(1)}: \quad
\begin{cases}
\kappa =
    -\dfrac{(4 + \nu^2)^2 (64 \tau^2 + (4 - \nu^2)^2)^3}{2^{17} \tau^4 (4 - \nu^2)^3}, \\
\eta =
    \left(\dfrac{4 - \nu^2}{8 \tau}\right)^2
    \left(\dfrac{(4 + \nu^2)^2 (64 \tau^2 + (4 - \nu^2)^2)}{2^9 \tau^2 (4 - \nu^2)} - 1\right)
\end{cases} \qquad \nu>2.
\end{equation}

A numerical minimization of $\kappa(\nu,\tau)$ over the domain $\nu > 2$ and $\tau > 0$
gives the approximate lower bound $\kappa \approx 6.75$. Hence, the curve $(\Gamma_{\mathrm{H},
0}^{(1)})$ lies outside the parameter window $(\kappa, \eta) \in [-5,5] \times [-1,1]$
considered in the numerical section below, which matches the one used in \cite{laing2018dynamics}.
\end{example}

\begin{example}[Dirac kernel] \label{ex:dirac.type1}
For the Dirac kernel \(l_\tau(t)=\delta(t-\tau)\), we have that
\begin{equation*}
\begin{cases}
r(\nu) = 1, \\
a(\nu) = \nu.
\end{cases}
\end{equation*}

In this case, $I_n = \left(\frac{\pi}{2} + n \pi, \pi + n \pi\right)$, for $n \in \{0, 1, \dots\}$.
By replacing $r(\nu)$ and $a(\nu)$ in system \eqref{eq:kappa.eta.system}, we
find an infinite number of candidate Hopf bifurcation curves, given parametrically by:
\begin{equation}\label{eq:kappa.eta.dirac}
\Gamma_{\mathrm{H},n}^{(1)}: \quad
\begin{cases}
\kappa =
    -\dfrac{(4 \tau^2 + \nu^2 \cot^2(\nu))^3}{128 \tau^4 \nu^2 \cot^3 (\nu) \sin(\nu)}, \\
\eta =
    \dfrac{\nu^2 \cot^2(\nu)}{4 \tau^2}
    \left(\dfrac{4 \tau^2 + \nu^2 \cot^2(\nu)}{8 \tau^2 \cos(\nu)} - 1\right),
\end{cases} \qquad \nu\in I_n.
\end{equation}
\end{example}

\subsection{Dirac case: Hopf bifurcations at \texorpdfstring{$z_a^\star$}{za} with respect to \texorpdfstring{$\tau$}{tau}}

In the case of a discrete time delay (Dirac kernel), the following result provides a
direct method to calculate the critical values of $\tau$ and determine the criticality
of the resulting Hopf bifurcations (see~\Cref{fig:critical.delay.type1}).

\begin{figure}[ht!]
\centering
\begin{subfigure}{0.49\textwidth}
\includegraphics[width=\linewidth]{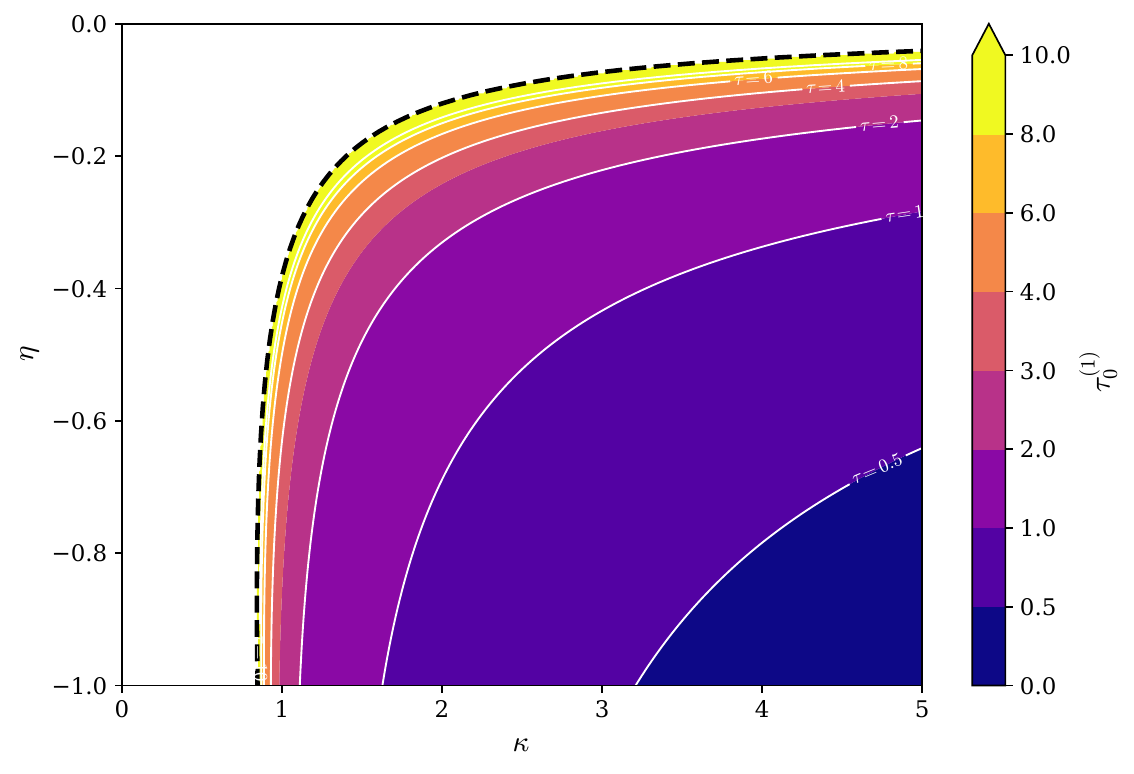}

\caption{\tiny First critical delay $\tau_0^{(1)}$ for the type 1 Hopf.}
\end{subfigure}
\begin{subfigure}{0.49\textwidth}
\includegraphics[width=\linewidth]{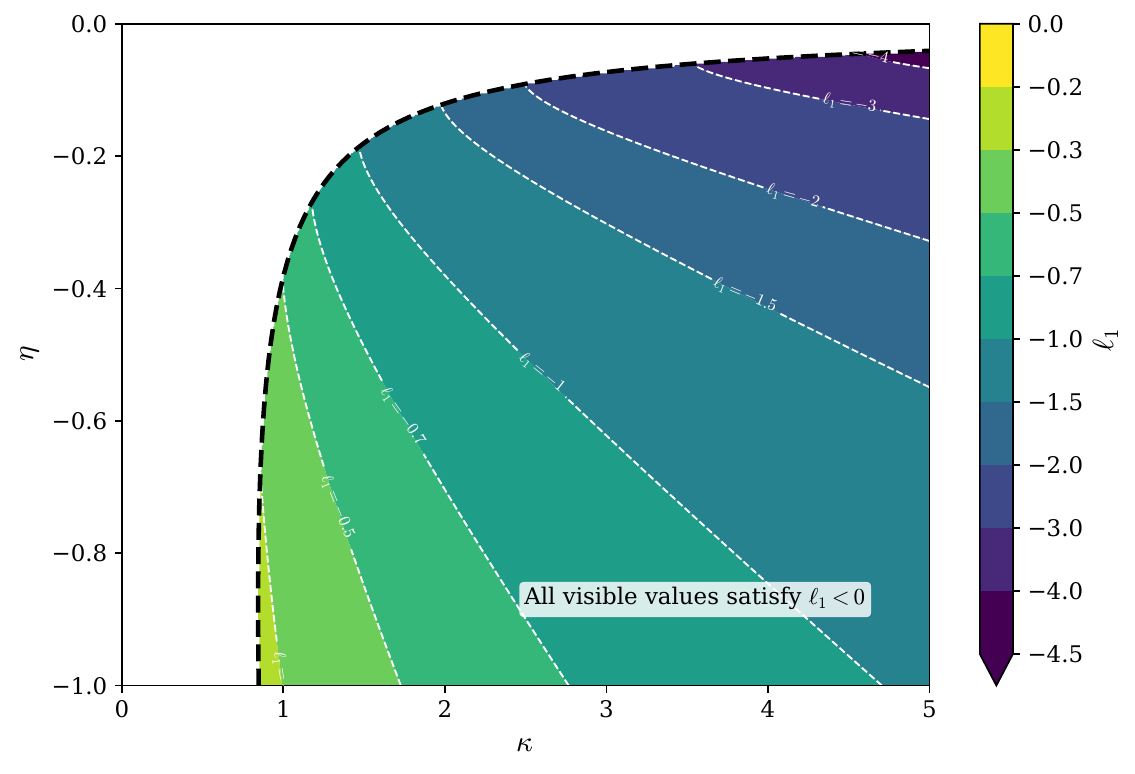}

\caption{\tiny First Lyapunov coefficient $\ell_1$ at the first critical delay $\tau_0^{(1)}$.}
\end{subfigure}

\caption{(left) First critical delay $\tau_0^{(1)}$ in the discrete-delay case and (right) the
corresponding first Lyapunov coefficient $\ell_1$ evaluated at $\tau = \tau_0^{(1)}$ for
the Hopf bifurcation of the unique type 1 equilibrium $z_a^\star$ on the lower
semicircle, in the parameter region $\kappa>0$, $\eta<0$. The dashed curve
$\Gamma_H^{(1)}$ delimits the subset of the $(\kappa,\eta)$-plane where this equilibrium
loses stability through a Hopf bifurcation. In the right panel, negative values of $\ell_1$
indicate a supercritical Hopf bifurcation.}
\label{fig:critical.delay.type1}
\end{figure}

\begin{proposition}[Dirac kernel: Hopf bifurcation of the stable type 1 equilibrium]
\label{prop:type1.dirac.hopf}
Let $l_\tau(t)=\delta(t-\tau)$, $\kappa > 0$ and $\eta < 0$.
By~\Cref{prop:type1.regions}, the delay-free system has a unique asymptotically stable
type 1 equilibrium \(z_a^\star = e^{\ii \phi^\star}\), with \(\phi^\star\in(-\pi,0)\).
A Hopf bifurcation may take place in a neighborhood of the type 1 equilibrium
$z_a^\star$ if and only if  $(\kappa,\eta)$ belongs to the open set bounded by the
curve
\[
\Gamma^{(1)}_{\mathrm{H}}:\qquad \kappa=\frac{(1+u)^3}{8u}, \qquad \eta=-\frac{u(u+3)}{2},\qquad u>0,
\]
which does not contain the origin (see \Cref{fig:critical.delay.type1}). In this case,
the critical delays for the Hopf bifurcation are
\begin{equation}\label{eq:type1.dirac.tau}
\tau_n^{(1)} = \frac{\arccos(\frac{a_1}{b_1}) + 2 n \pi}{\sqrt{b_1^2 - a_1^2}},
\qquad n\in \{0, 1, \dots\},
\end{equation}
where $a_1=a_1(\phi^\star)>0$ and $b_1=b_1(\kappa,\phi^\star)<0$ are given by
\eqref{eq:a1.b1} and satisfy $a_1 < -b_1$. The equilibrium $z_a^\star$ is asymptotically
stable if and only if $\tau \in [0, \tau_0^{(1)})$ and no stability switching occurs for $\tau
> \tau_0^{(1)}$.
\end{proposition}

\begin{proof}
The Hopf bifurcation results follow from \Cref{prop:general}(iii). We note that in the
case of a Dirac kernel, the characteristic equation \eqref{eq:char.type1} may have pure
imaginary roots for critical values of $\tau$ if and only if $a_1<-b_1$.

Denoting \(u=\tan^2\left(\frac{\phi^\star}{2}\right)\) in~\eqref{eq:a1.b1} leads to
\[
a_1=2\sqrt{u},\qquad
b_1
=
-\frac{16\kappa u\sqrt{u}}{(1+u)^3}.
\]
Therefore, the boundary case \(a_1 = -b_1\) together with the equilibrium equation
\eqref{eq:P1.type1} leads to the parametric equations of the curve $\Gamma_H^{(1)}$,
which represents the boundary of the Hopf region. The inequality \(a_1<-b_1\) is
equivalent to \((\kappa,\eta)\) belonging to the component of the complement of
\(\Gamma_H^{(1)}\) that does not contain the origin.

By the transversality condition~\eqref{eq:transversality.type1}, as all characteristic
roots have negative real part at \(\tau=0\), and roots can leave the open left
half-plane only through the imaginary axis, it follows that \(z_a^\star\) is
asymptotically stable for all \(\tau\in[0,\tau_0^{(1)})\). At \(\tau=\tau_0^{(1)}\) the
first conjugate pair crosses into the right half-plane, so stability is lost via a Hopf
bifurcation. Because every subsequent crossing at \(\tau=\tau_n^{(1)}\) has the same
sign \(\Re\,s'(\tau_n^{(1)})>0\), it follows that the number of characteristic roots in
the open right half-plane is nondecreasing as $\tau$ increases. Therefore, the
equilibrium never regains stability for \(\tau>\tau_0^{(1)}\), and no stability
switching occurs.
\end{proof}

Consequently, when a discrete time delay is considered and $(\kappa,\eta)$ are in the
region bounded by the curve $\Gamma^{(1)}_H$, namely the component not containing the
origin, at the first critical value $\tau=\tau_0^{(1)}$, a Hopf bifurcation takes place in a
neighborhood of $z_a^\star$, resulting in the appearance of a limit cycle. Numerical
simulations (see~\Cref{fig:critical.delay.type1}, right) show that the first Lyapunov
coefficient $\ell_1$ remains negative for all fourth quadrant $(\kappa,\eta)$ values
in the Hopf region. This confirms that the Hopf bifurcation is supercritical and leads
to the appearance of a stable limit cycle.

\section{Local stability and bifurcation analysis of type 2 equilibria}
\label{sec:stability.type2}

For an arbitrary type 2 equilibrium, the matrices \(\matr{A}\) and \(\matr{B}\)
from~\eqref{eq:matrix.equation} simplify to
\[
\matr{A} =
(\eta+\kappa I^\star-1)\,
\frac{z^{\star 2}-1}{2z^\star|z^\star|}
\begin{pmatrix}
0 & -z^{\star 2} \\
1 & 0
\end{pmatrix},
\qquad
\matr{B} =
\kappa\,\frac{(1+z^\star)^2}{2|z^\star|}(z^\star-2)
\begin{pmatrix}
0 & 0\\
1 & 0
\end{pmatrix},
\]
where $I^\star = \frac{3}{2}-2z^\star+\frac{1}{2}z^{\star 2}$. Hence, the characteristic
equation \eqref{eq:general.characteristic} becomes
\begin{equation}\label{eq:type2.char}
\Delta_2(s,\tau)\coloneq s^2
+
a_2-b_2L(\tau s)
=
0,
\end{equation}
where
\begin{equation}
\label{eq:a2.b2}
a_2 = a_2(z^\star) = 4\left(\frac{1-z^\star}{1+z^\star}\right)^2,
\qquad
b_2 = b_2(\kappa,z^\star) = \kappa(1-z^{\star 2})(2-z^\star).
\end{equation}

\begin{remark}
This characteristic equation remains valid at \(z^\star=0\) (by Cartesian
linearization), even though the matrices $\matr{A}$ and $\matr{B}$ become singular in this
formulation.
\end{remark}

\subsection{Classification of type 2 equilibria in the delay-free case}

As in the case of type 1 equilibria (see~\Cref{prop:type1.regions}), the number of type
2 equilibria is determined by the position relative to two curves in the $(\kappa,
\eta)$ plane: a saddle-node curve $\Gamma^{(2)}_{\mathrm{SN}}$ and the line $\eta = 0$. These two
curves partition the parameter plane into regions of constant equilibrium count, as seen
in~\Cref{fig:type2.no.equilibria}. The following proposition characterizes the class of
type 2 equilibria in the delay-free case. Note that, unlike for the type 1 case, there
are no asymptotically stable type 2 equilibria.

\begin{figure}[h]
\centering
\includegraphics[width=0.55\linewidth]{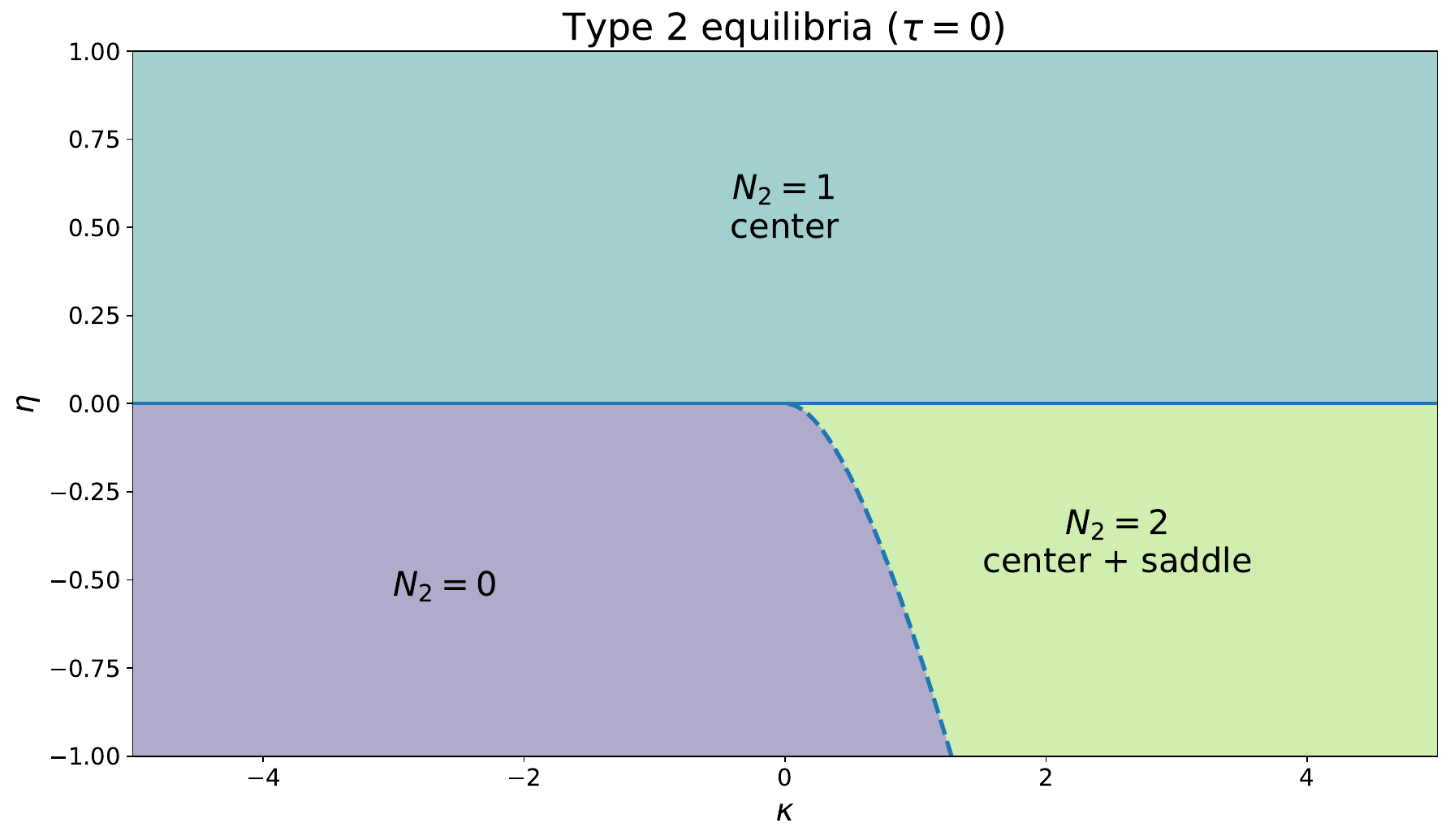}

\caption{Number ($N_2$) and classification of type 2 equilibria in the delay-free case
(see~\Cref{prop:type2.tau0}) in the regions of the $(\kappa,\eta)$ parameter plane
delimited by the saddle-node curve $\Gamma_{\mathrm{SN}}^{(2)}$ (dashed) and the boundary
line $\eta = 0$.}

\label{fig:type2.no.equilibria}
\end{figure}

\begin{proposition}[Number and classification of type 2 equilibria when \(\tau=0\)]
\label{prop:type2.tau0}
Consider the delay-free system~\eqref{eq:sys.2d}. Let $\Sigma \coloneq \{(\kappa, \eta):
\eta = 0\}$ and $\Gamma^{(2)}_{\mathrm{SN}}$ be the saddle-node bifurcation curve of the type 2
family of equilibria, defined parametrically by
\begin{equation}\label{eq:type2.SN.curve}
\Gamma_{\mathrm{SN}}^{(2)}: \qquad
\left\{
\begin{aligned}
\kappa & =\dfrac{4(1 - u)}{(1 + u)^3(2 - u)}, \\
\eta & =-\dfrac{(1-u)^2(u^2-3u+4)}{(1+u)^3(2-u)},
\end{aligned}
\right.
\qquad u\in(-1,1),
\end{equation}
which represents the graph of a smooth function \(\eta=\eta_{\mathrm{SN}}^{(2)}(\kappa) < 0\),
for all \(\kappa > 0\). Then, in each connected component of the complement of
$\Gamma_{\mathrm{SN}}^{(2)} \cup \Sigma$ (see~\Cref{fig:type2.no.equilibria}), we have:
\[
N_2=
\begin{cases}
0, & \eta<0 \text{ and either } \kappa\le0
     \text{ or } (\kappa>0 \text{ and } \eta<\eta_{\mathrm{SN}}^{(2)}(\kappa)),\\[1mm]
2, & \kappa>0 \text{ and } \eta_{\mathrm{SN}}^{(2)}(\kappa)<\eta<0,\\[1mm]
1, & \eta>0.
\end{cases}
\]

More precisely, in the $N_2 = 1$ region the unique type 2 equilibrium is a center and in the
$N_2 = 2$ region the equilibria are a center and a saddle. For the $N_2 = 2$ case, the
center corresponds to the smaller-valued equilibrium $z^\star \in (-1, 1)$, while the saddle
corresponds to the larger $z^\star$.
\end{proposition}

\begin{proof}
From \eqref{eq:equilibrium}, type 2 equilibria are the roots of \(F_\kappa(z)=\eta\) with
\(z \in (-1,1)\). In this case, $F_\kappa$ satisfies the following properties:
\[
F_\kappa'(z)=\kappa(2-z)-\frac{4(1-z)}{(1+z)^3},
\qquad
\lim_{z\to -1^+}F_\kappa(z)=+\infty,
\qquad\lim_{z\to 1^-} F_\kappa(z)=0.
\]

The saddle-node bifurcation for type 2 equilibria occurs when $F_\kappa(z) = \eta$ has
a double root in $(-1, 1)$, i.e. when $F_\kappa(z) = \eta$ and $F_\kappa'(z) = 0$
hold simultaneously. Setting $F_\kappa'(z) = 0$ gives
\[
\kappa = K(z) \coloneq \frac{4 (1 - z)}{(1 + z)^3 (2 - z)}.
\]

Since $K$ is strictly decreasing on $(-1, 1)$, this defines $\kappa$ uniquely as a
function of $z$ and, therefore, of $z^\star$. Substituting this expression back into
$F_\kappa(z^\star) = \eta$ then determines $\eta$ as a function of $z^\star$ as well.
These two expressions form the parametrization of the $\Gamma^{(2)}_{\mathrm{SN}}$ curve. Since
$K$ is bijective, the parametrization can also be considered the graph of a function
$\eta^{(2)}_{\mathrm{SN}}(\kappa)$ for $\kappa > 0$.

To determine the number of equilibria in each region, we start by looking at the $\kappa
\le 0$ half-plane. There, we have that \(F_\kappa'(z)<0\) for all \(z\in(-1,1)\), so
\(F_\kappa\) is strictly decreasing from \(+\infty\) to \(0\). Therefore, $F_\kappa(z) =
\eta$ has exactly one solution for \(\eta>0\) and no solutions for \(\eta<0\). On the
other hand, for $\kappa > 0$, we have that $\eta^{(2)}_{\mathrm{SN}}(\kappa) < 0$ by its
parametric representation (see also~\Cref{fig:type2.no.equilibria}). As
$F_\kappa(z^\star) = \eta^{(2)}_{\mathrm{SN}}(\kappa) < 0$, we can deduce that $F_\kappa$
decreases from $+\infty$ to a negative $\eta^{(2)}_{\mathrm{SN}}(\kappa)$ and then rises back to
$0$ as $z^\star \to 1$. Counting the intersections with the horizontal line for every
$\eta$ gives the number of equilibria shown in the proposition.

To determine the type of each equilibrium, we examine the characteristic equation~\eqref{eq:type2.char}
in the delay-free case. Setting $\tau = 0$ and $L(0) = 1$, the equation simplifies to
\begin{equation}\label{eq:type2.char.tau0}
s^2 - (1 - z^{\star 2}) F_\kappa'(z^\star) = 0.
\end{equation}

Therefore, if \(F_\kappa'(z^\star)<0\) the equilibrium is a center and if
\(F_\kappa'(z^\star)>0\) the equilibrium is a saddle point. On \(\Gamma_{\mathrm{SN}}^{(2)}\)
one has \(F_\kappa'(z^\star)=0\), corresponding to a double zero eigenvalue. In the $N_2
= 2$ region, $F_\kappa$ has a unique minimum on $(-1, 1)$ and the two equilibria lie on
either side of this minimum. At the smaller root $F_\kappa' < 0$, so the equilibrium is
a center, and at the larger root $F_\kappa' > 0$, so the equilibrium is a saddle.
\end{proof}

\Cref{prop:type2.tau0} shows that there are at most two type 2 equilibria. In what follows,
we denote them by $z_s^\star$ and $z_c^\star$ for the type 2 delay-free saddle and
center equilibria, respectively.

\subsection{Effect of distributed delays on the delay-free saddle \texorpdfstring{$z_s^\star$}{zs}}

\begin{proposition}[Delay-independent instability of the type 2 delay-free saddle]
\label{prop:type2.saddle.persists}
Let \(z_s^\star\in(-1,1)\) be the type 2 equilibrium which is a saddle point for the
delay-free system, as defined in~\Cref{prop:type2.tau0}. Then, \(z_s^\star\) is unstable
regardless of the admissible delay kernel (in the sense
of~\Cref{def:admissible.kernels}) considered in~\eqref{eq:sys.2d}.
\end{proposition}

\begin{proof}
If \(z_s^\star\) is a saddle point for \(\tau=0\), the delay-free characteristic
equation \(s^2 + a_2 - b_2 = 0\), where $a_2$ and $b_2$ are given by \eqref{eq:a2.b2},
has two real roots of opposite sign. Hence \(a_2 < b_2\) and the conclusion follows
from~\Cref{prop:general}(ii).
\end{proof}

\subsection{Effect of distributed delays on the delay-free center \texorpdfstring{$z_c^\star$}{zc}}

We now restrict our attention to the type 2 delay-free center $z_c^\star$ and study the
effects of adding a delay. The following proposition describes the stability for small
mean delay $\tau$.

\begin{proposition}[Small-delay effect on a type 2 center]
\label{prop:type2.small.delay}
Consider system \eqref{eq:sys.2d} with an admissible delay kernel family
$\{l_\tau\}_{\tau \ge 0}$, in the sense of~\Cref{def:admissible.kernels}. Then, there
exists \(\tau_0 > 0\) such that, for every \(0 < \tau < \tau_0\), the type 2 delay-free
center $z_c^\star$, defined in~\Cref{prop:type2.tau0}, is asymptotically stable for
\(\kappa>0\) and unstable for \(\kappa<0\).
\end{proposition}

\begin{proof}
Since \(z_c^\star\) is a type 2 center of the delay-free system, we have \(a_2 > b_2\),
where $a_2$ and $b_2$ are given by \eqref{eq:a2.b2}. Based
on~\Cref{rem:laplace.admissible}, the Laplace transform $L$ satisfies \(L(\xi) = 1 - \xi +
O(\xi^2)\) as \(\xi \to 0\). Hence, for small $\tau$, the characteristic equation
\eqref{eq:type2.char} formally becomes
\[
s^2 + a_2 - b_2 + b_2 \tau s + O((\tau s)^2) = 0.
\]

We consider a standard asymptotic expansion for the roots $s_\pm(\tau)$ of the form
$s(\tau) = s_0 + \tau s_1 + O(\tau^2)$. Substituting this expansion in the characteristic
equation above and equating the terms, we find that
\begin{equation} \label{eq:type2.small.delay.expansion}
s_\pm(\tau) = \pm \ii \sqrt{a_2 - b_2} - \frac{b_2}{2} \tau + O(\tau^2),
\end{equation}
where the zero-order term represents the delay-free roots discussed
in~\Cref{prop:type2.tau0} and the first-order term represents the perturbation.
From~\eqref{eq:a2.b2}, it follows that $\sign(b_2) = \sign(\kappa)$. Therefore, the pair
$s_\pm(\tau)$ moves into the right half-plane for $\kappa < 0$ and into the left half-plane
for $\kappa > 0$. In particular, when $\kappa < 0$, the equilibrium $z_c^\star$ is
unstable for all sufficiently small $\tau > 0$.

For $\kappa > 0$, this argument is not sufficient to determine the asymptotic stability
of $z_c^\star$. To prove asymptotic stability in this case, we make use of the
continuity of roots of the characteristic equation~\eqref{eq:type2.char} with respect to
its parameters. For admissible kernels and $\Re s \ge 0$, we have that $|L(\tau s)| \le 1$
(see~\Cref{rem:laplace.admissible}), which implies that any root with $\Re s \ge 0$ is bounded by
\[
|s|^2 = |b_2 L(\tau s) - a_2| \le a_2 + b_2,
\]
i.e. it lies in the compact disk $|s| \le R_0 \coloneq \sqrt{a_2 + b_2}$ for all $\tau >
0$. By continuity of the roots and compactness of the domain, any root of $\Delta_2(s,
\tau)$ with $\Re s \ge 0$ must lie near the roots of $\Delta_2(s, 0)$ for small $\tau > 0$.
From~\Cref{prop:type2.tau0}, we know that the only delay-free roots are purely
imaginary. Moreover, by~\eqref{eq:type2.small.delay.expansion}, the roots of $\Delta_2(s,
\tau)$ for small $\tau > 0$ satisfy $\Re{s_\pm(\tau)} < 0$. Therefore, no root with
$\Re{s} \ge 0$ exists for sufficiently small $\tau > 0$ and $\kappa > 0$, so $z_c^\star$
is asymptotically stable.
\end{proof}

\subsection{Hopf curves at \texorpdfstring{$z_c^\star$}{zc} in the
            \texorpdfstring{$(\kappa, \eta)$}{(kappa, eta)}-plane}
\label{sec:type2.hopf}

To determine whether the stability of the delay-free center $z_c^\star$ may be lost or
regained as \(\tau\) increases, we look for purely imaginary characteristic roots
\(s=\ii \nu/\tau\), with \(\nu > 0\), of the characteristic equation
\eqref{eq:type2.char}. Using the polar representation of the Laplace transform from
\Cref{def:phase.regular.kernels}, we can write
\[
-\frac{\nu^2}{\tau^2} + a_2 - b_2 r(\nu)\cos(a(\nu)) + \ii b_2\,r(\nu)\sin(a(\nu)) = 0.
\]

Separating the real and imaginary parts, we obtain the following equivalent system
\begin{equation}\label{eq:type2.realimag.rewrite}
\begin{cases}
-\dfrac{\nu^2}{\tau^2} + a_2 - b_2 r(\nu) \cos(a(\nu)) = 0, \\
b_2 r(\nu) \sin(a(\nu)) = 0.
\end{cases}
\end{equation}

As \(r(\nu) > 0\), assuming $\kappa \neq 0$ in~\eqref{eq:a2.b2}, the second equation in
\eqref{eq:type2.realimag.rewrite} implies
\begin{equation}\label{eq:a.n.pi}
a(\nu)=n\pi,
\qquad n \in \{1, 2, \dots\}.
\end{equation}

Let \(\nu_n > 0\) be the $n$-th solution to~\eqref{eq:a.n.pi}, when it exists. Then, the
first equation in \eqref{eq:type2.realimag.rewrite}, together with the type 2
equilibrium equation \eqref{eq:P2.type2}, leads to a parametric representation of the
Hopf candidate curves associated with the type 2 family. They are given by
\begin{equation}\label{eq:type2.Hopf.curves.rewrite}
\Gamma_{\mathrm{H},n}^{(2)}(l_\tau): \qquad
\left\{
\begin{aligned}
\kappa & = (-1)^n \frac{4\left(\frac{1 - u}{1 + u}\right)^2 - \frac{\nu_n^2}{\tau^2}}
                       {(1 - u^2)(2 - u) r(\nu_n)}, \\
\eta & =
\left(\frac{1 - u}{1 + u}\right)^2
- (-1)^n \frac{3 - u}{2(1 + u)(2 - u) r(\nu_n)} \left(
    4 \left(\dfrac{1 - u}{1 + u}\right)^2
    - \frac{\nu_n^2}{\tau^2}
\right),
\end{aligned}
\right.
\end{equation}
where the parameter is $u \in (-1, 1)$. The expressions for $a_2$ and $b_2$ have been
written out explicitly from~\eqref{eq:a2.b2}. Among these parametric curves, we are
interested in those corresponding to the center branch of the delay-free system
(see~\Cref{prop:type2.tau0}), i.e. for which $F_\kappa'(u) < 0$.

\begin{example}[Weak Gamma kernel]
For the weak Gamma kernel (see~\Cref{ex:weak.gamma.type1}), we have that
\(a(\nu)\in(0,\pi/2)\) for all \(\nu>0\), so equation \eqref{eq:a.n.pi} does not have a
solution. Hence, the type 2 delay-free center does not undergo a Hopf bifurcation in the
case of a weak Gamma kernel. This means that for any $\tau > 0$, this equilibrium point
remains asymptotically stable when $\kappa > 0$ and unstable when $\kappa<0$, as stated
in~\Cref{prop:type2.small.delay} for small $\tau$.
\end{example}

\begin{example}[Strong Gamma kernel]
For the strong Gamma kernel (see~\Cref{ex:strong.gamma.type1}), we have that
\(a(\nu) \in (0, \pi)\) for all \(\nu > 0\), so equation \eqref{eq:a.n.pi}
has no solution. Therefore, the type 2 delay-free center does not undergo a Hopf
bifurcation for the strong Gamma kernel either.
\end{example}

\begin{example}[Dirac kernel]
For the Dirac kernel (see~\Cref{ex:dirac.type1}), $a(\nu) = \nu$, so~\eqref{eq:a.n.pi}
has infinitely many solutions of the form $\nu_n = n \pi$, for $n \in \{1, 2, \dots\}$.
Then, the bifurcation curve~\eqref{eq:type2.Hopf.curves.rewrite} becomes
\begin{equation}\label{eq:type2.Hopf.curves.Dirac.rewrite}
\Gamma_{\mathrm{H},n}^{(2)}: \qquad
\left\{
\begin{aligned}
\kappa & = (-1)^n \frac{4\left(\frac{1 - u}{1 + u}\right)^2 - \frac{\pi^2 n^2}{\tau^2}}
                       {(1 - u^2)(2 - u)}, \\
\eta & =
\left(\frac{1 - u}{1 + u}\right)^2
- (-1)^n \frac{3 - u}{2(1 + u)(2 - u)} \left(
    4 \left(\dfrac{1 - u}{1 + u}\right)^2
    - \frac{\pi^2 n^2}{\tau^2}
\right).
\end{aligned}
\right.
\end{equation}

This yields an infinite family of candidate Hopf curves $\Gamma_{\mathrm{H},n}^{(2)}$, $n \in
\{1, 2, \dots\}$, in the $(\kappa,\eta)$-plane for a fixed $\tau > 0$. The corresponding
critical values of $\tau$ for fixed $(\kappa,\eta)$ are determined in
\Cref{prop:type2.dirac.hopf}.
\end{example}

\subsection{Dirac case: Hopf bifurcations at \texorpdfstring{$z_c^\star$}{zc} with respect to \texorpdfstring{$\tau$}{tau}}

Among the three kernel families considered in this paper, the Dirac kernel is the only
one for which the type 2 center $z_c^\star$ can undergo a Hopf bifurcation, as
established above. In what follows, we deduce the critical values of
the delay $\tau > 0$ for which Hopf bifurcations take place in a neighborhood of
$z_c^\star$ and for $\kappa \ne 0$ (see also~\Cref{fig:critical.delay.type2}).

\begin{proposition}[Dirac kernel: Hopf bifurcation of the type 2 delay-free center]
\label{prop:type2.dirac.hopf}
Let $l_\tau(t) = \delta(t - \tau)$ be a family of Dirac kernels, with $\tau > 0$. Let
$\kappa \ne 0$ and \(z_c^\star=z_c^\star(\kappa,\eta)\in(-1,1)\) be the type 2 equilibrium
which is a center of the delay-free system, as shown in~\Cref{prop:type2.tau0}. Then,
the characteristic equation \eqref{eq:type2.char} has purely imaginary roots if and only if
\[
\tau = \tau_n^{(2)} \coloneq \frac{n\pi}{\sqrt{a_2-(-1)^n b_2}},
\]
for those integers \(n \in \{1, 2, \dots\}\) such that \(a_2 - (-1)^n b_2 > 0\), where
$a_2$ and $b_2$ are given by \eqref{eq:a2.b2}. At the critical values $\tau_n^{(2)}$, a Hopf
bifurcation takes place in a neighborhood of $z_c^\star$, and the following
transversality condition holds:
\begin{equation}\label{eq:type2.transversality}
\sign (\Re s'(\tau_n^{(2)})) = (-1)^{n + 1} \sign (\kappa).
\end{equation}
\end{proposition}

\begin{proof}
The proof follows from \Cref{prop:general}(iii) and $\sign(b_2) = \sign(\kappa)$
from~\eqref{eq:a2.b2}.
\end{proof}

\begin{figure}[h]
    \centering
    \begin{subfigure}{0.49\textwidth}
    \includegraphics[width=\linewidth]{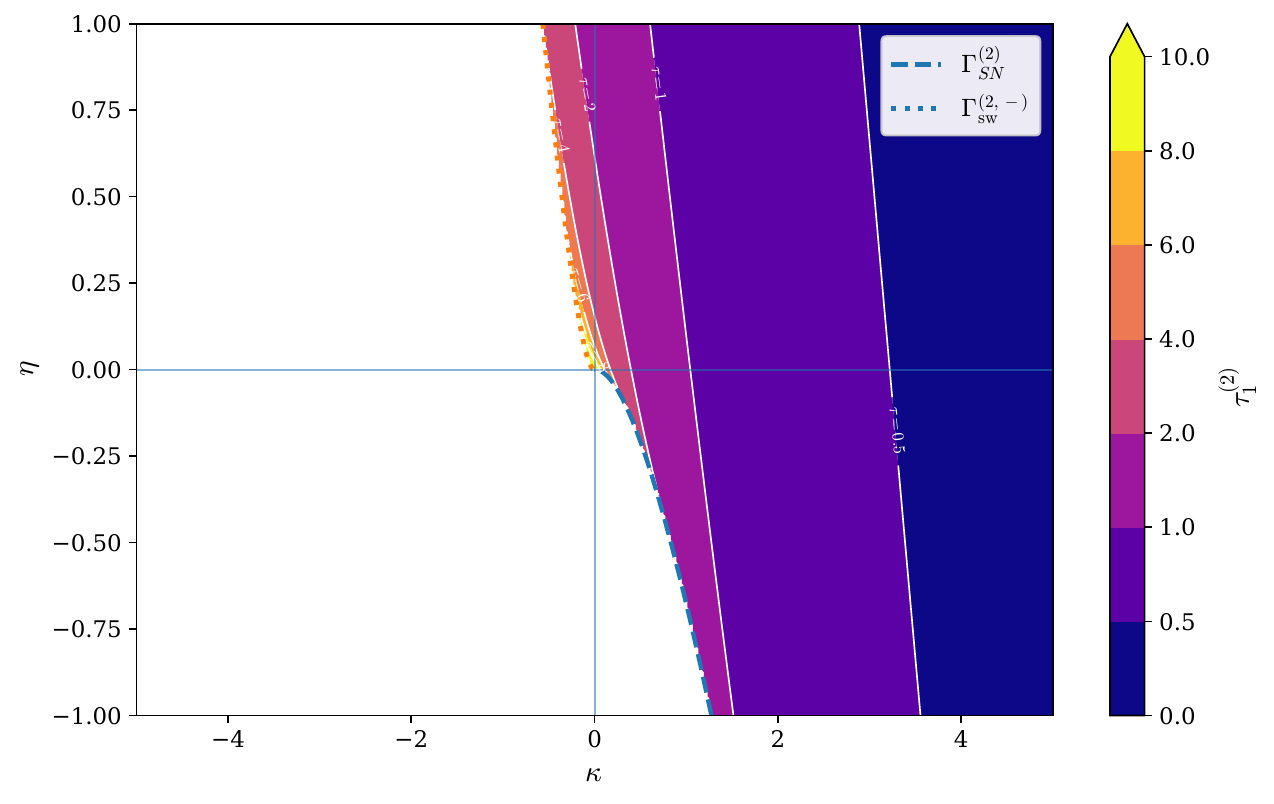}

    \caption{\tiny First critical delay $\tau_1^{(2)}$ for the type 2 center.}
    \end{subfigure}
    \begin{subfigure}{0.49\textwidth}
    \includegraphics[width=\linewidth]{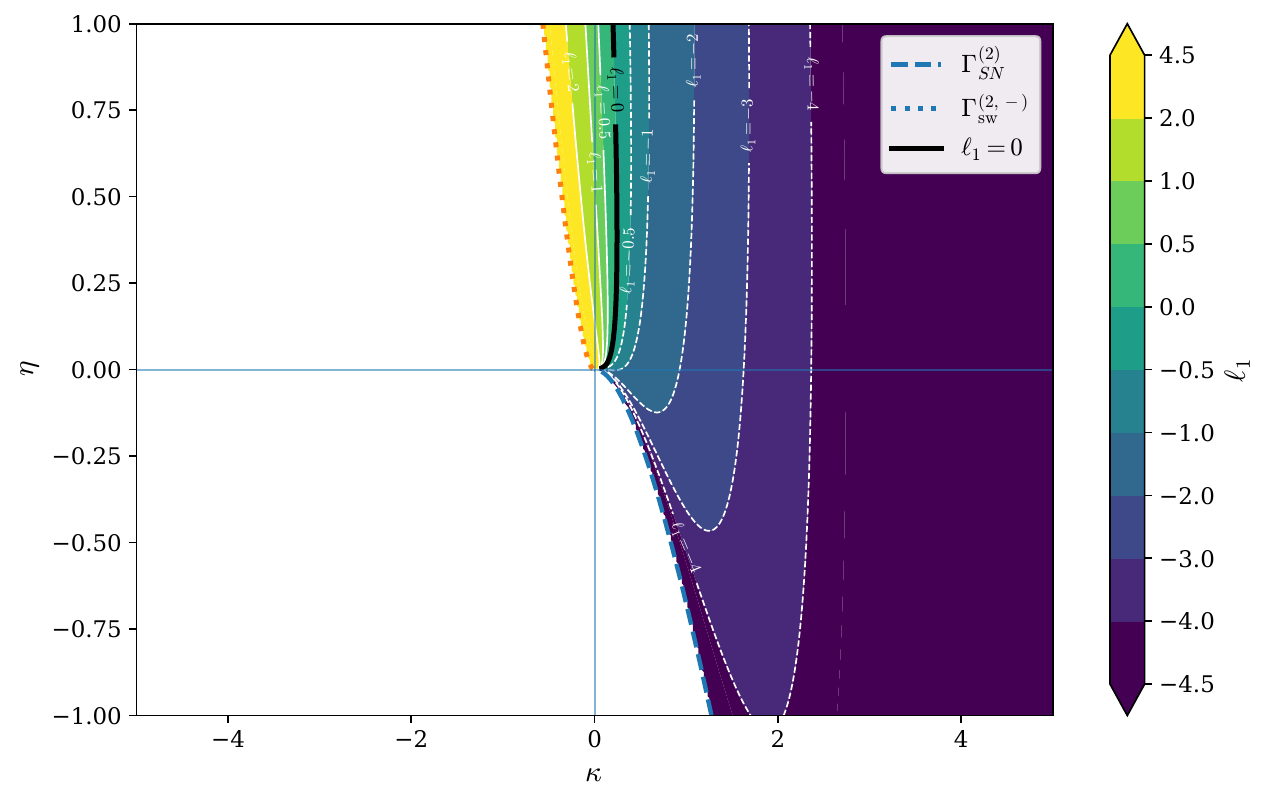}

    \caption{\tiny First Lyapunov coefficient $\ell_1$ at the first critical delay $\tau_1^{(2)}$.}
    \end{subfigure}

    \caption{(left) First critical delay $\tau_1^{(2)}$ in the discrete-delay case and (right)
    the corresponding first Lyapunov coefficient $\ell_1$ evaluated at $\tau = \tau_1^{(2)}$ for the
    Hopf bifurcation of the unique type 2 delay-free center $z_c^\star$.}
    \label{fig:critical.delay.type2}
\end{figure}

In~\Cref{fig:critical.delay.type2}, the left panel shows that the first odd critical
value for the type 2 delay-free center is \(\tau_1^{(2)}=\pi/\sqrt{a_2+b_2}\). For
\(\kappa>0\), this is the first critical value at which the equilibrium, which is
asymptotically stable for small \(\tau > 0\), loses stability by a Hopf bifurcation. For
\(\kappa<0\), whenever \(\tau_1^{(2)}<\tau_2^{(2)}\), the same critical value $\tau_1^{(2)}$ instead
corresponds to the first stabilizing threshold of the delay-free center. The right panel
displays the first Lyapunov coefficient \(\ell_1\) evaluated at \(\tau=\tau_1^{(2)}\). The
contour \(\ell_1=0\) separates the parameter values for which the Hopf bifurcation changes
criticality. In particular, for \(\kappa>0\), negative values of \(\ell_1\) correspond to a
supercritical Hopf bifurcation at $z_c^\star$, hence to the appearance of a stable
periodic orbit in a neighborhood of \(\tau_1^{(2)}\), while positive values correspond to a
subcritical Hopf bifurcation.

\begin{figure}[ht!]
    \centering
    \includegraphics[width=0.55\linewidth]{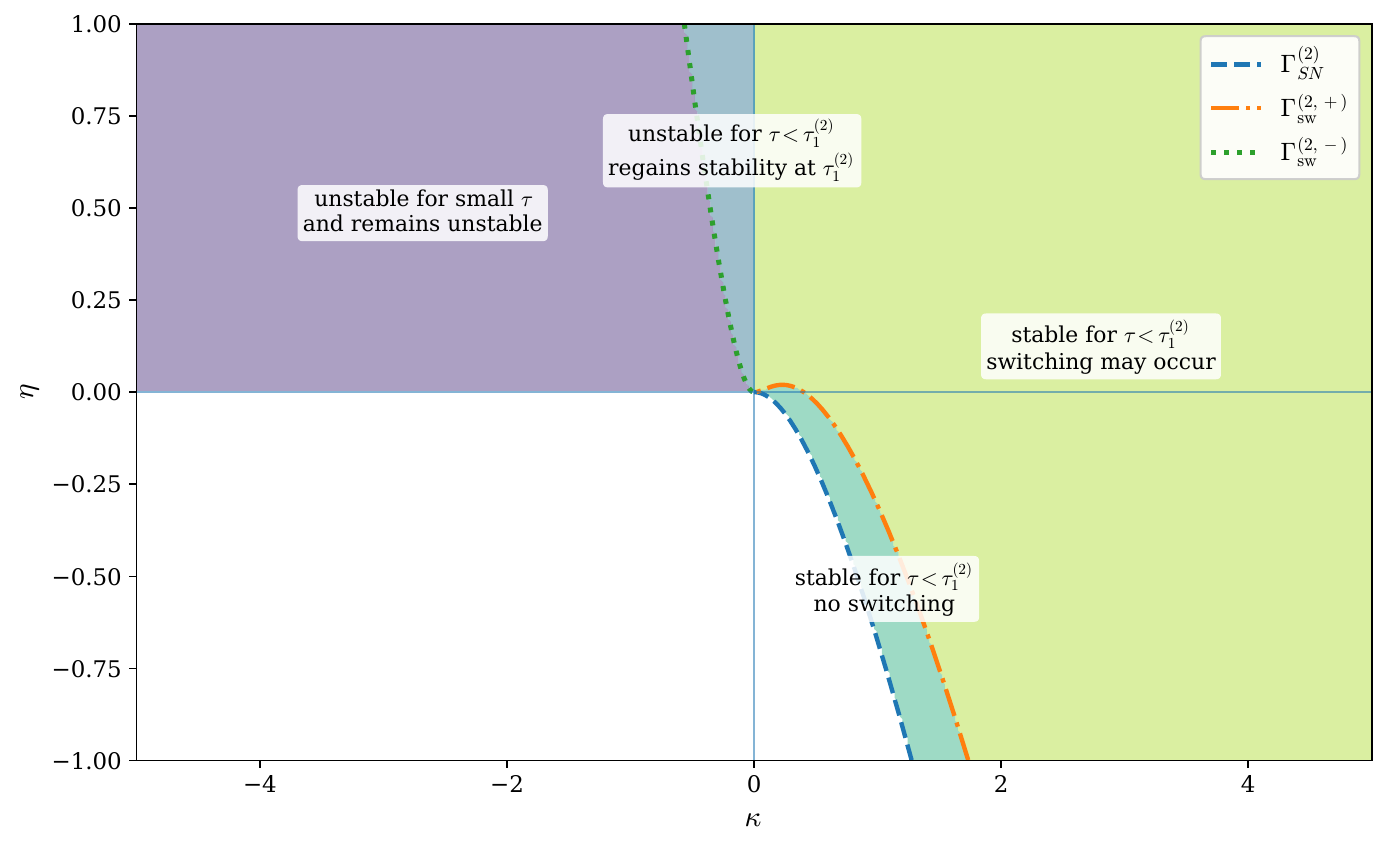}
    \caption{Stability scenarios of the type 2 delay-free center $z_c^\star$ when a discrete Dirac kernel with delay $\tau$ is considered in the system.}
    \label{fig:stab.switching.center}
\end{figure}

\begin{remark}[Stability switching in the Dirac case]
\label{rem:type2.stability.switching}
We construct the boundaries of the first stability-switching and stabilization regions
shown in~\Cref{fig:stab.switching.center}. Further stability switching may occur for
larger values of the delay whenever successive critical values satisfy the appropriate
ordering. The transversality formula \eqref{eq:type2.transversality} shows that the
crossing directions alternate with \(n\).

If \(\kappa>0\), then odd crossings are destabilizing, and even crossings are
stabilizing. The delay-free center $z_c^\star$ is asymptotically stable for
$\tau\in(0,\tau_1^{(2)})$ and loses stability by a Hopf bifurcation at $\tau=\tau_1^{(2)}$.
Stability can be regained after this first loss if and only if
\[
\tau_2^{(2)} < \tau_3^{(2)}
\quad \iffarrow \quad
13 b_2 < 5 a_2.
\]

The boundary of this stability-switching region is deduced by imposing $13 b_2 = 5 a_2$
and substituting into the equilibrium equation~\eqref{eq:P2.type2}. This gives the
parametric curve (see~\Cref{fig:stab.switching.center}):
\[
\Gamma_{\mathrm{sw}}^{(2,+)}:\qquad
\left\{
\begin{aligned}
\kappa & = \frac{20(1-u)}{13(1+u)^3(2-u)}, \\
\eta & = -\frac{(1 - u)^2(13 u^2 - 23u + 4)}{13 (1 + u)^3(2 - u)},
\end{aligned}
\right.
\qquad u \in(-1,1).
\]
Therefore, for $\kappa > 0$, the equilibrium is stable for $\tau\in(0,\tau_1^{(2)})\cup(\tau_2^{(2)},\tau_3^{(2)})$.

If \(\kappa<0\), then odd crossings are stabilizing and even crossings are
destabilizing. In this case, the equilibrium is unstable for small values of $\tau$, but
may become asymptotically stable at the first odd critical delay \(\tau_1^{(2)}\) (which
exists, provided that
$a_2 + b_2 > 0$) if
\[
\tau_1^{(2)} < \tau_2^{(2)}
\quad \iffarrow \quad
3 a_2 > -5 b_2.
\]

Equivalently, the boundary of the delay-induced stabilization region is
\[
\Gamma_{\mathrm{sw}}^{(2,-)}:\qquad
\left\{
\begin{aligned}
\kappa & = -\frac{12 (1 - u)}{5 (1 + u)^3 (2 - u)}, \\
\eta & = \frac{(1 - u)^2 (28 - u - 5 u^2)}{5 (1 + u)^3 (2 - u)},
\end{aligned}
\right.
\qquad u \in(-1, 1).
\]
Therefore, for $\kappa < 0$, the equilibrium is asymptotically stable for
$\tau\in(\tau_1^{(2)},\tau_2^{(2)})$.
\end{remark}

\section{Synthesis of the local bifurcation structure in the discrete-delay case}
\label{sec:classification.figure}

\Cref{fig:attractor.classification} provides an overview of the local bifurcation
structure in the \((\kappa,\eta)\)-plane for four values of the delay $\tau$ in a Dirac
kernel: \(\tau = 0.1, 0.5, 1\) and \(2\). The two saddle-node curves
\(\Gamma_{\mathrm{SN}}^{(1)}\) (see~\Cref{prop:type1.regions}) and \(\Gamma_{\mathrm{SN}}^{(2)}\)
(see~\Cref{prop:type2.tau0}) delimit the regions of existence of type 1 and type 2
equilibria, respectively, and are independent of the delay. On the other hand, the Hopf
curves \(\Gamma_H^{(1)}\) (see~\Cref{prop:type1.dirac.hopf}) and \(\Gamma_H^{(2)}\)
(see~\Cref{prop:type2.dirac.hopf}) depend on \(\tau\) and describe the loss or recovery
of stability of the corresponding equilibrium branches. Thus,
\Cref{fig:attractor.classification} should be understood as a local branch-based summary
of the delay-dependent stability changes induced by the Hopf mechanism.

\begin{figure}[ht!]
\centering
\includegraphics[width=0.85\linewidth]{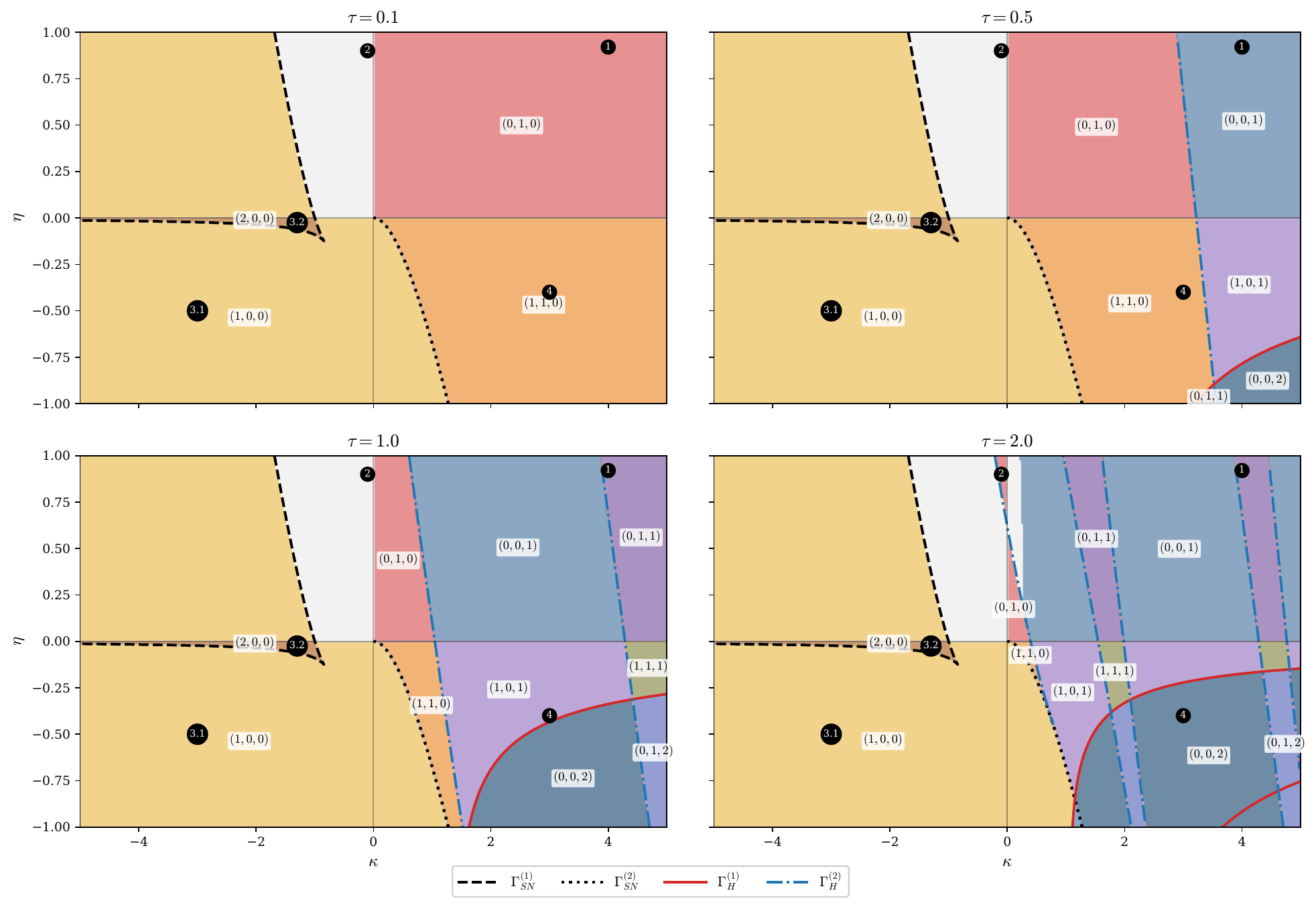}

\caption{Attractor classification for selected values of the discrete delay $\tau$.
Region labels denote $(S_1, S_2, C)$, with $S_1$ the number of stable type 1 equilibria,
$S_2$ the number of stable type 2 equilibria, and $C$ the number of stable cycles arising from
the first Hopf branches. Gray regions indicate parameter values for which the present local
bifurcation-based classification detects no stable attractor among the tracked branches.
The numbered points denote the parameter sets found in~\Cref{tab:representative.parameters}
and used in the numerical simulations.}
\label{fig:attractor.classification}
\end{figure}

As defined in the caption of~\Cref{fig:attractor.classification}, each colored region is
labeled with a triple $(S_1, S_2, C)$. This labeling reflects a \emph{local branch-based
classification}: the equilibrium counts \(S_1\) and \(S_2\) follow directly from the
stability results proved in the previous subsections, while the cycle count \(C\) is
obtained by combining the Hopf thresholds with the sign of the first Lyapunov
coefficient $\ell_1$ on the corresponding Hopf branch.

The panel corresponding to \(\tau=0.1\) is especially useful as a reference, since this
delay value lies below the smallest Hopf threshold across the entire displayed parameter
window. Therefore, no Hopf curve has yet been crossed, and the only delay effects
present in that panel are those consistent with the small-delay analysis of the type 2
center from~\Cref{prop:type2.small.delay}: for \(\kappa>0\), the delay-free center
$z_c^\star$ becomes asymptotically stable, while for \(\kappa<0\) it becomes unstable.
At the same time, the stable type 1 equilibria in the half-plane \(\kappa<0\) remain
stable, in agreement with the delay-independent type 1 stability result
from~\Cref{prop:type1.regions}.

As \(\tau\) increases, the delay-dependent Hopf curves \(\Gamma_H^{(1)}\) and
\(\Gamma_H^{(2)}\) intersect the equilibrium-existence regions and induce stability
changes. For \(\kappa>0\), \(\eta<0\), the type 2 delay-free center $z_c^\star$ becomes
stable for small delays and then loses stability through a Hopf bifurcation whenever
\(\tau=\tau_1^{(2)}\) exists (see~\Cref{prop:type2.dirac.hopf}). In the same region, the
stable type 1 equilibrium $z_a^\star$ also loses stability through a Hopf bifurcation
whenever \(\tau=\tau_0^{(1)}\) exists (see~\Cref{prop:type1.dirac.hopf}). This explains
the appearance, for intermediate and larger delays, of regions in which a stable
equilibrium coexists with a stable cycle, or in which a cycle remains the only stable
attractor detected by the local analysis. In the first quadrant, where type 1 equilibria
are absent, the changes are entirely due to the type 2 branch: the panels with
\(\tau=0.5, 1\) and \(2\) reflect the succession of stabilization, loss of stability by
a Hopf bifurcation, and possible restabilization established
in~\Cref{prop:type2.small.delay}, \Cref{prop:type2.dirac.hopf},
and~\Cref{rem:type2.stability.switching}.

The gray regions of~\Cref{fig:attractor.classification} correspond to parameter values
for which the present local classification does not detect a stable branch among the
tracked families. However, this should not be interpreted as the absence of attractors.
In particular, when
\begin{equation} \label{eq:gray.region}
\kappa < 0
\qquad \text{and} \qquad
\eta+4\kappa>0,
\end{equation}
we can prove the existence of at least one periodic orbit on the invariant circle
\(\rho=1\) (see~\Cref{prop:unit.circle.periodic.orbit}). Therefore, the gray region
satisfying~\eqref{eq:gray.region} reflects a limitation of
the present local branch classification rather than the absence of periodic dynamics.
The numerical simulations presented below suggest that this unit-circle
periodic orbit is attracting when the type 2 delay-free center is unstable.

\section{Numerical simulations}
\label{sec:num}

In this section, we illustrate the local bifurcation results obtained above by numerical
simulations of the reduced system~\eqref{eq:complex} in the discrete-delay case of the
Dirac kernel. The simulations are not intended to provide a complete global bifurcation
analysis. Rather, they serve three purposes: first, to verify the stability changes
predicted by the characteristic equations; second, to illustrate the coexistence of
stable equilibria and stable periodic orbits after Hopf bifurcations; and third, to show
what happens in parameter regions where the local branch-based classification of
\Cref{fig:attractor.classification}  does not detect a stable equilibrium branch.

\begin{table}[ht!]
\centering
\caption{Parameter values and delay values used in the numerical simulations.}
\label{tab:representative.parameters}

{
\small
\begin{tabular}{@{}l @{}c @{}c p{0.25\linewidth} l@{}}
\toprule
& \textbf{Quadrant} & \((\kappa,\eta)\)
& \textbf{Critical delays}
& \textbf{Simulation delays} \\
\midrule

\textbf{Set~1}
& I
& \((4,0.92)\)
& \(\tau_1^{(2)}\approx0.424,\ \tau_2^{(2)}\approx0.986\)\par
  \(\tau_3^{(2)}\approx1.273,\ \tau_4^{(2)}\approx1.971\)
& \(\{0,0.3,0.5,1.05,1.4,2.05\}\) \\[1mm]

\textbf{Set~2}
& II
& \((-0.1,0.9)\)
& \(\tau_1^{(2)}\approx1.857,\ \tau_2^{(2)}\approx3.488\)
& \(\{0,0.5,2.5,3.6\}\) \\[1mm]

\textbf{Set~3.1}
& III
& \((-3,-0.5)\)
& None among the tracked Hopf thresholds
& \(\{0,1\}\) \\[1mm]

\textbf{Set~3.2}
& III
& \((-1.3,-0.025)\)
& None among the tracked Hopf thresholds
& \(\{0,1\}\) \\[1mm]

\textbf{Set~4}
& IV
& \((3,-0.4)\)
& \(\tau_1^{(2)}\approx0.536,\ \tau_0^{(1)}\approx1.076\)\par
  \(\tau_2^{(2)}\approx1.318,\ \tau_3^{(2)}\approx1.607\)
& \(\{0,0.1,0.6,0.9,1.25,1.4,1.55,2\}\) \\

\bottomrule
\end{tabular}
}
\end{table}

Since the delayed system has an infinite-dimensional phase space, the basin plots below
should be interpreted as basins in a two-dimensional slice of this phase space. For each
initial point $z_0$ in the unit disk we use the constant history $z(t)=z_0$ for
$t\in[-\tau,0]$, integrate the delayed system, and classify the long-time behavior as
convergence to a type 1 equilibrium, convergence to a type 2 equilibrium, or convergence
to a periodic orbit. Therefore, the resulting plots describe the dependence of the
observed attractor on constant initial histories, rather than full basins in the
complete history space.

The five parameter sets used in the simulations are marked in
\Cref{fig:attractor.classification} and listed in \Cref{tab:representative.parameters}.
They were chosen to sample the main regimes predicted by the local theory:
delay-independent stability of a type 1 equilibrium, interaction between type 1 and type
2 Hopf mechanisms, stability switching of a single type 2 branch, and a parameter region
where an attracting periodic orbit on the invariant circle is observed, although it is
not predicted by the local equilibrium classification.

\newlength{\localcolumnwidth}

\subsection{Set 1: first quadrant, single type 2 equilibrium}

\begin{figure}[ht!]
\centering
\setlength{\localcolumnwidth}{0.2\linewidth}

\begin{subfigure}{\localcolumnwidth}
\includegraphics[width=\linewidth]{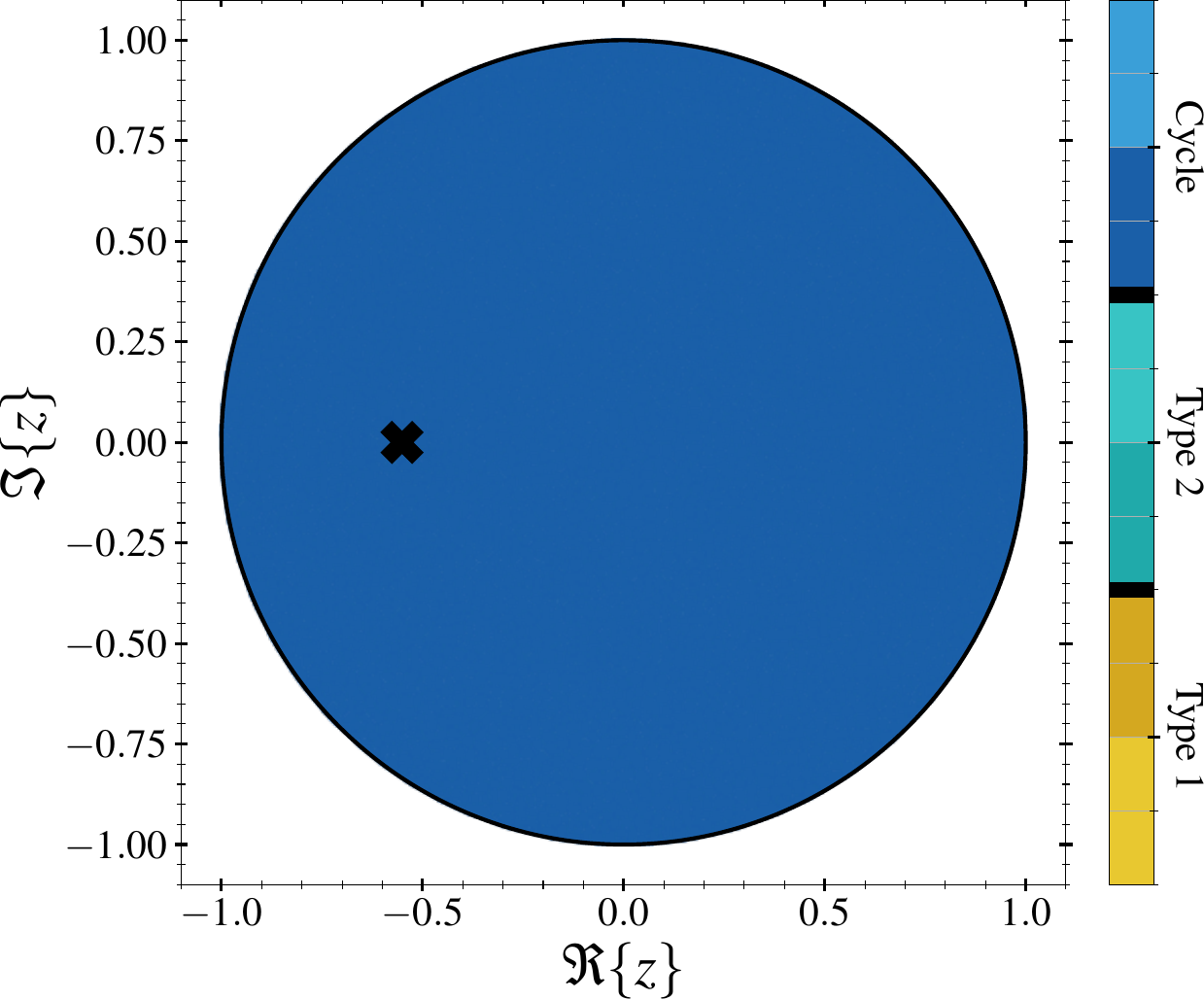}
\includegraphics[width=\linewidth]{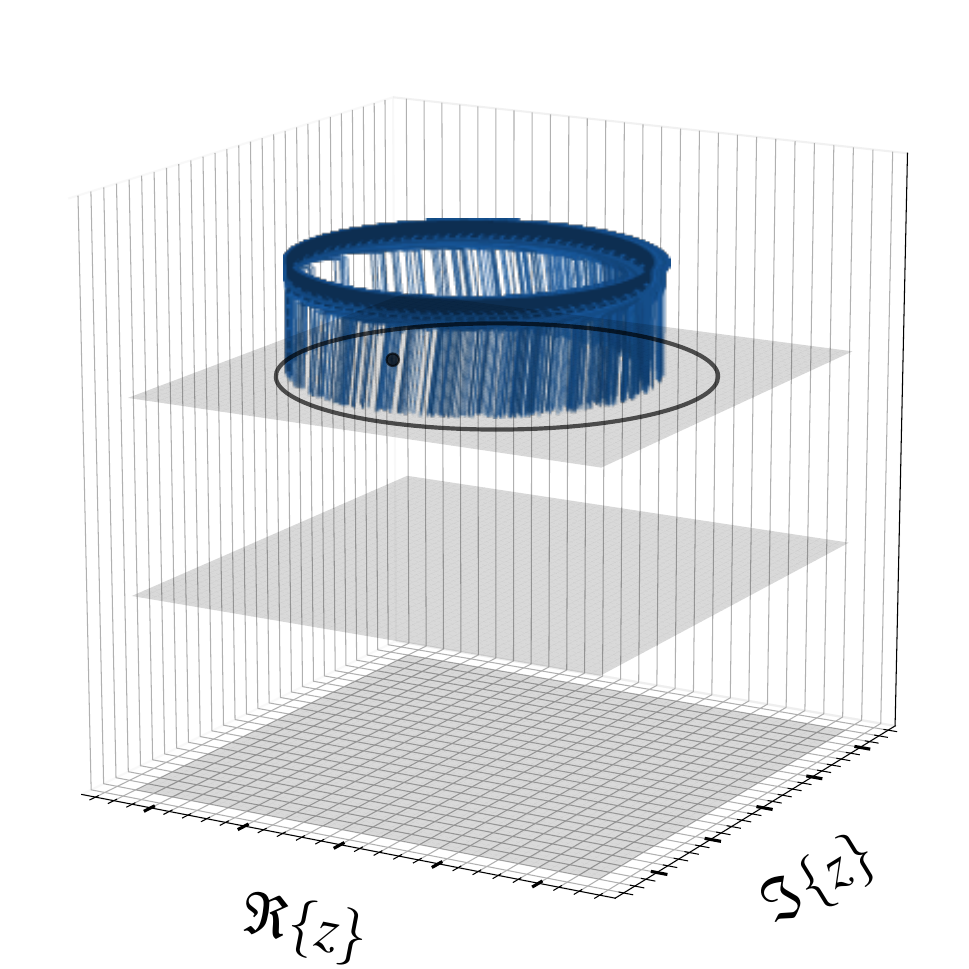}
\caption{$\tau = 0$}
\end{subfigure}
\begin{subfigure}{\localcolumnwidth}
\includegraphics[width=\linewidth]{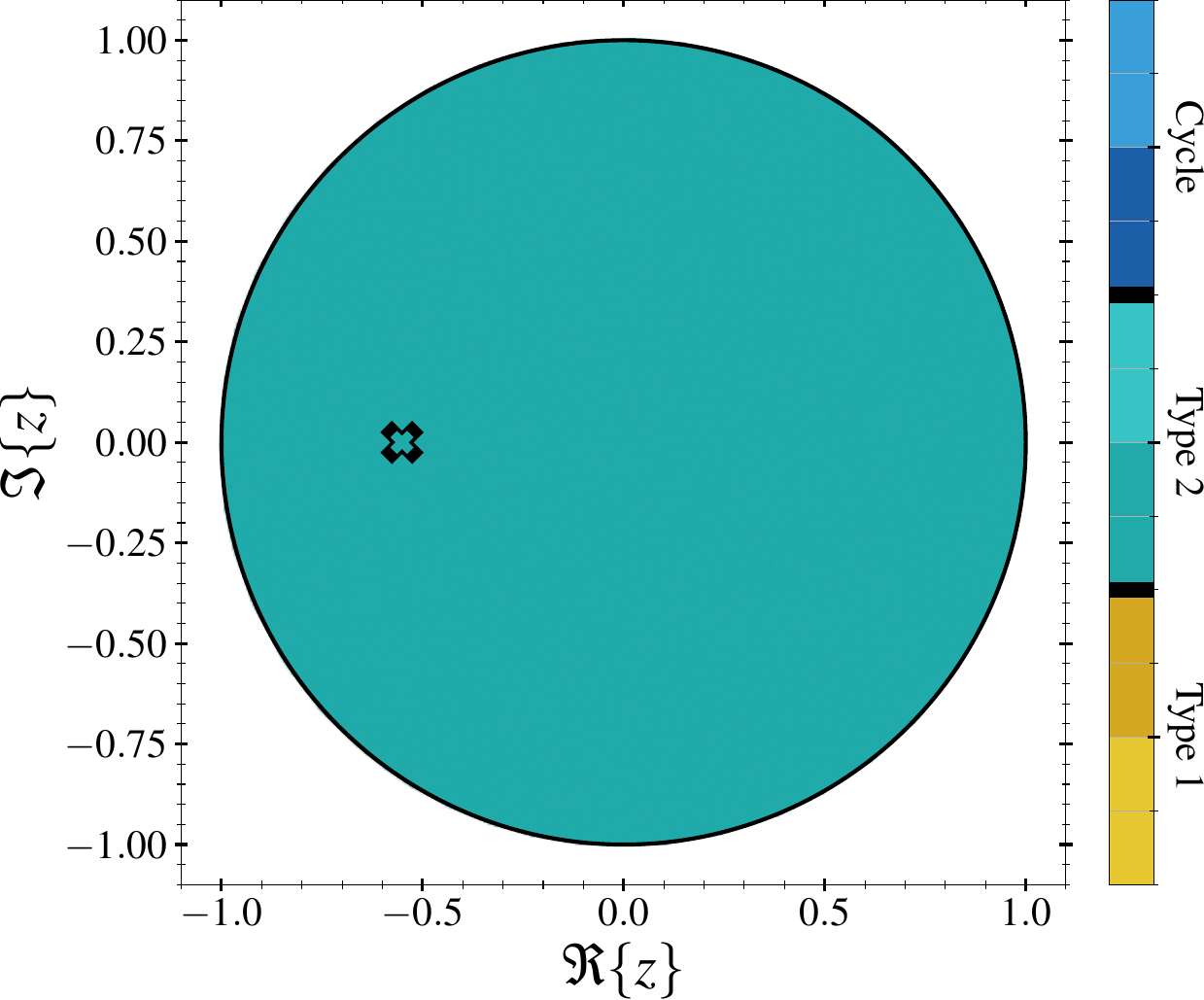}
\includegraphics[width=\linewidth]{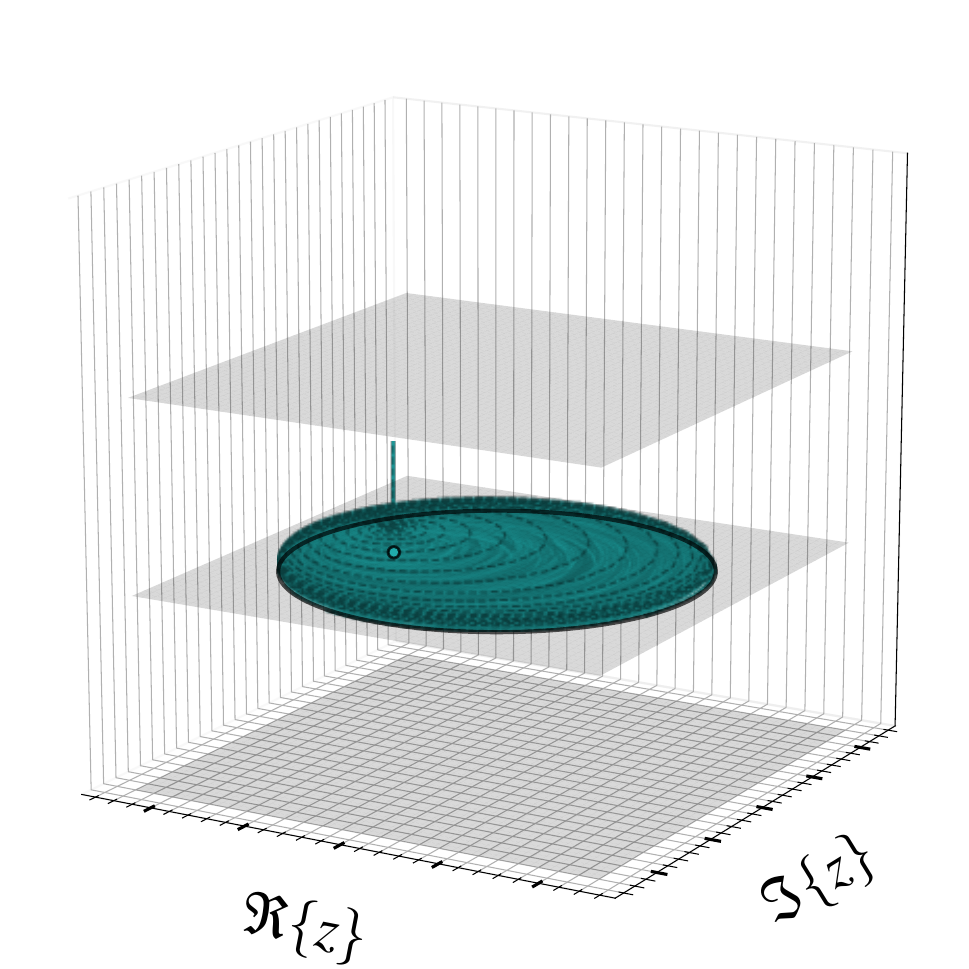}
\caption{$\tau = 0.3$}
\end{subfigure}
\begin{subfigure}{\localcolumnwidth}
\includegraphics[width=\linewidth]{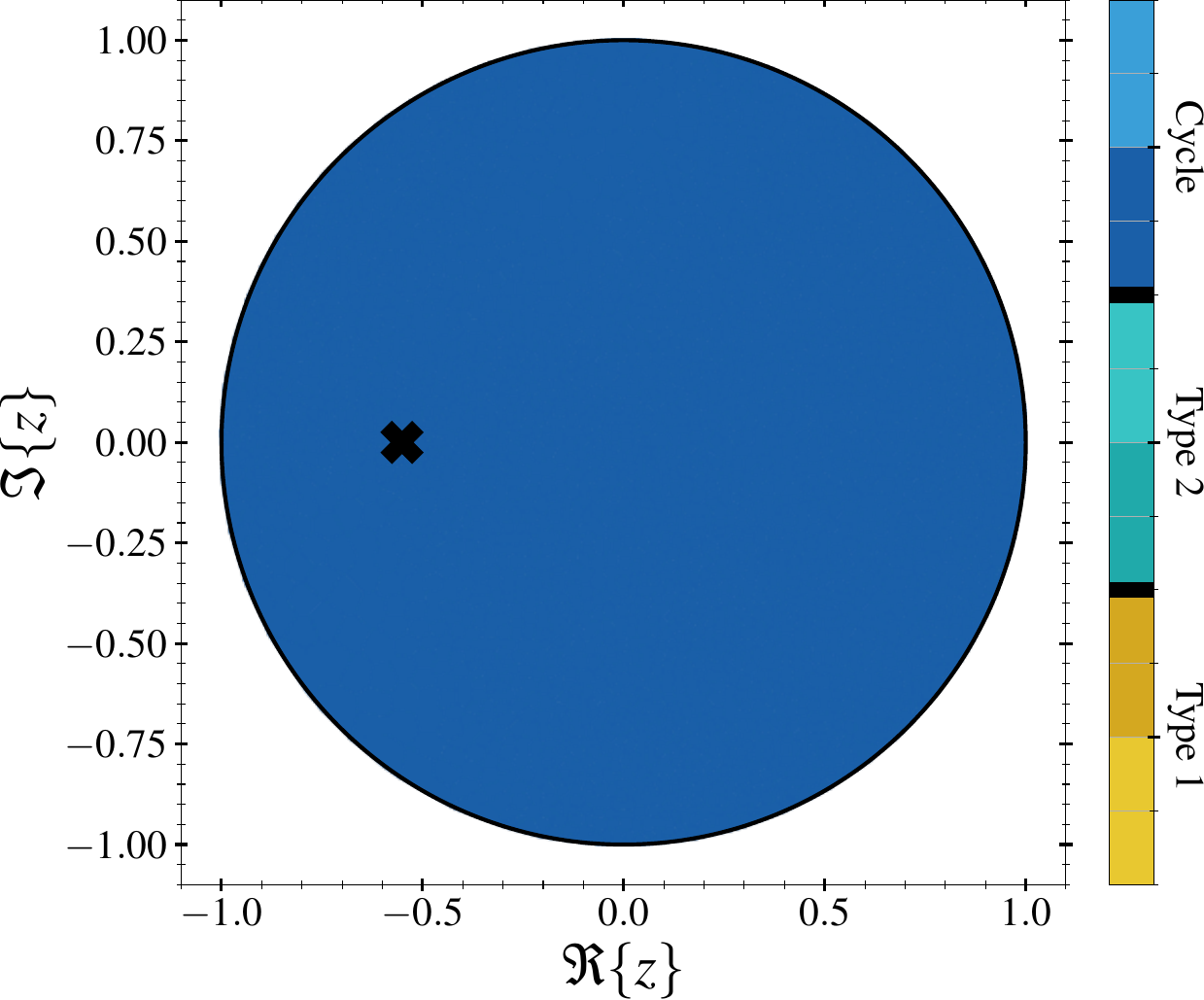}
\includegraphics[width=\linewidth]{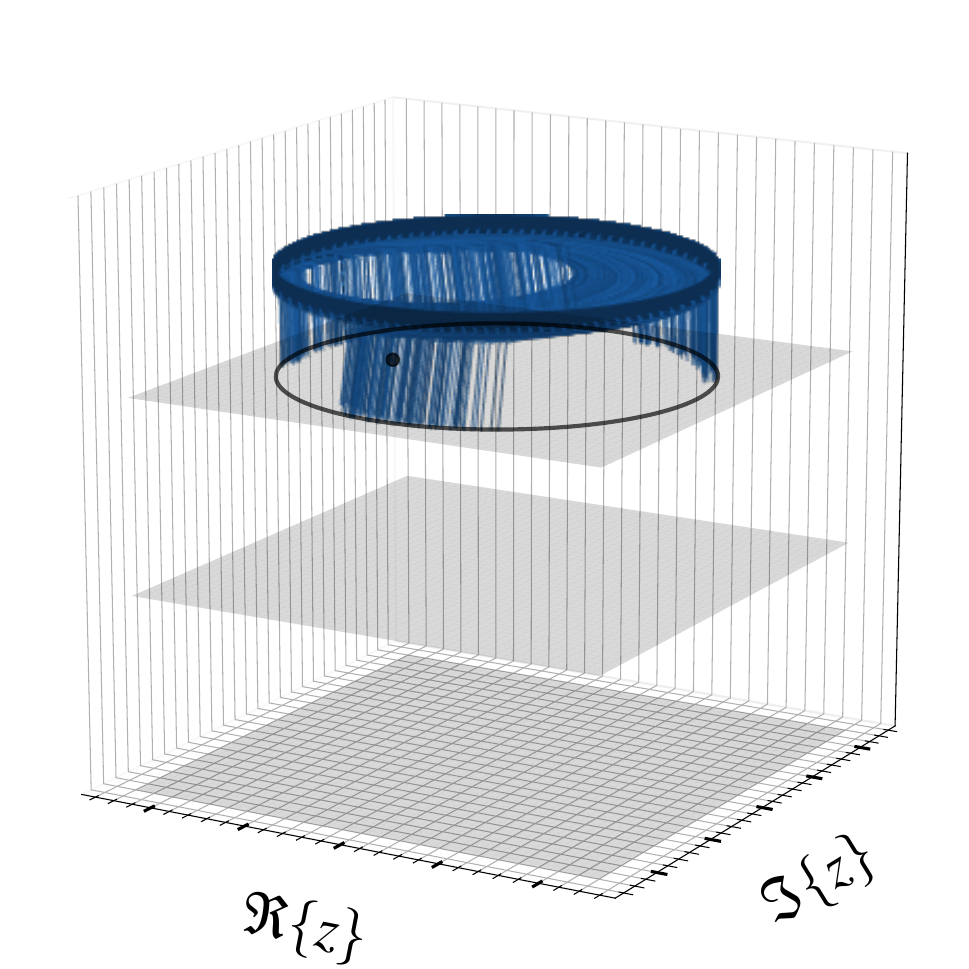}
\caption{$\tau = 0.5$}
\end{subfigure}

\begin{subfigure}{\localcolumnwidth}
\includegraphics[width=\linewidth]{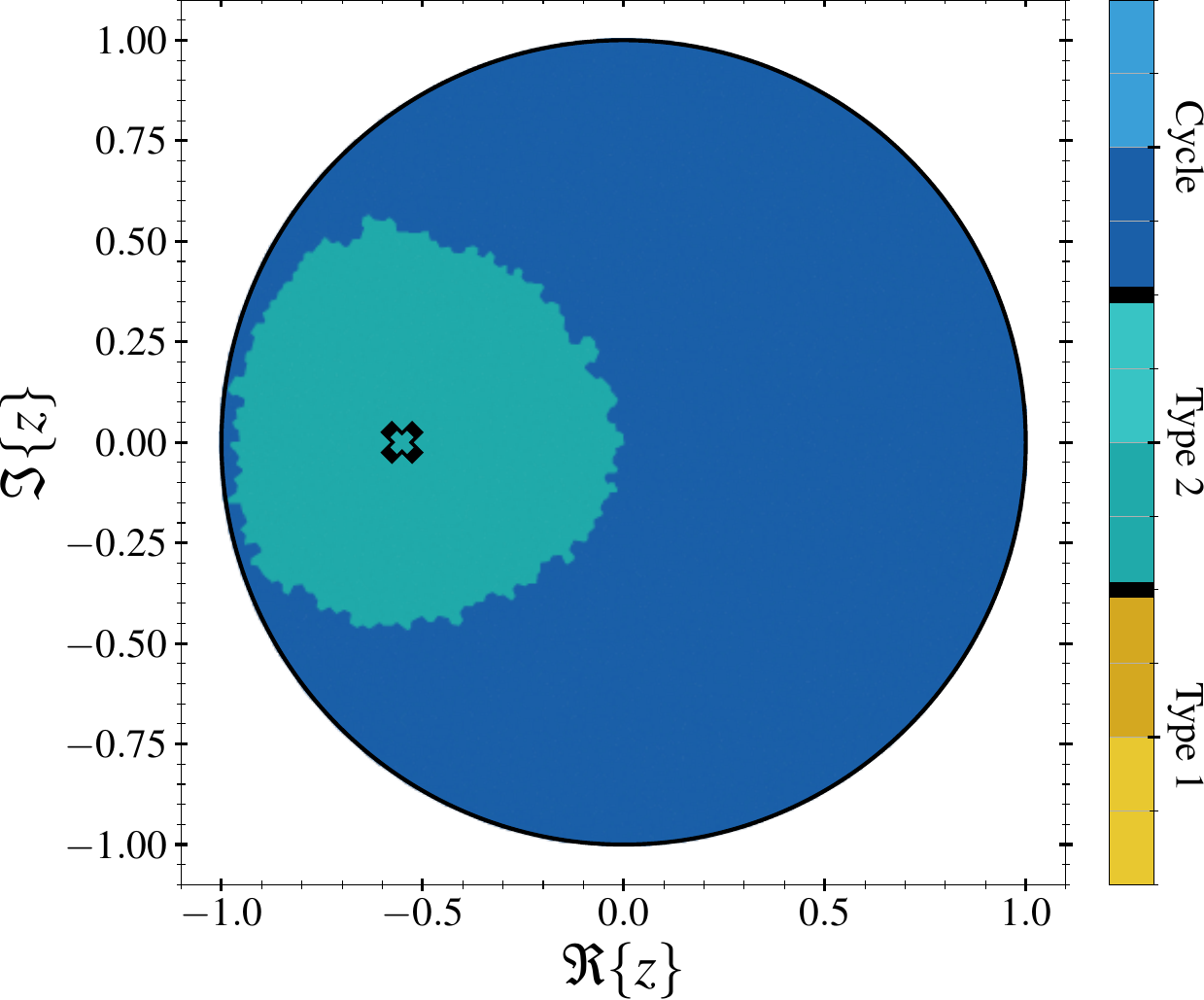}
\includegraphics[width=\linewidth]{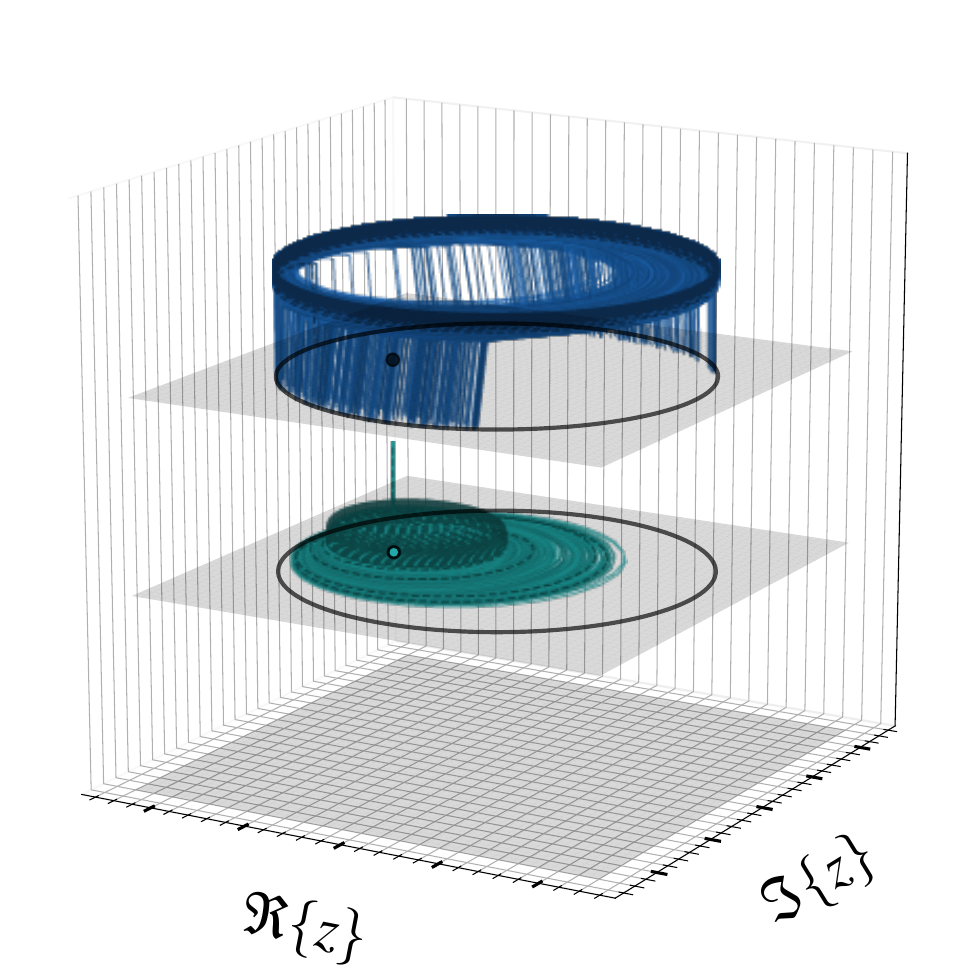}
\caption{$\tau = 1.05$}
\end{subfigure}
\begin{subfigure}{\localcolumnwidth}
\includegraphics[width=\linewidth]{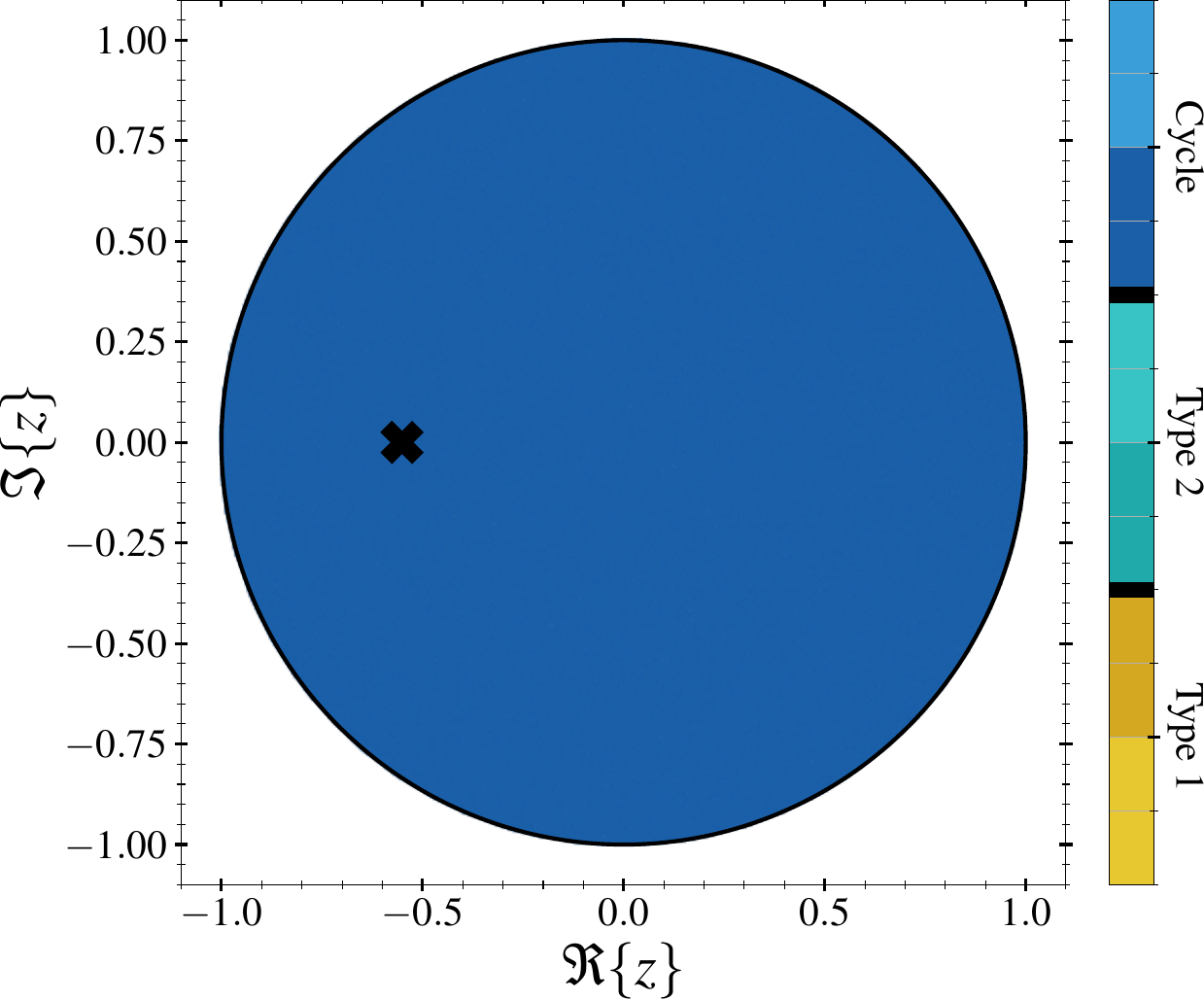}
\includegraphics[width=\linewidth]{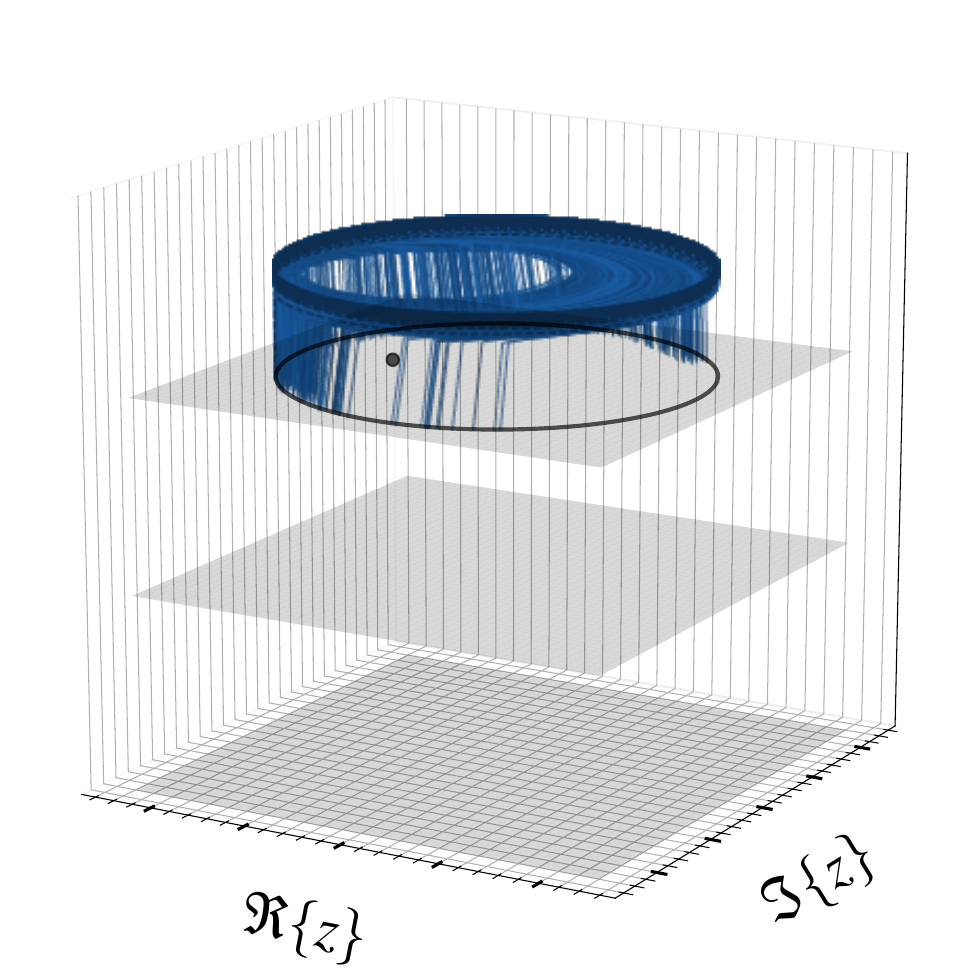}
\caption{$\tau = 1.4$}
\end{subfigure}
\begin{subfigure}{\localcolumnwidth}
\includegraphics[width=\linewidth]{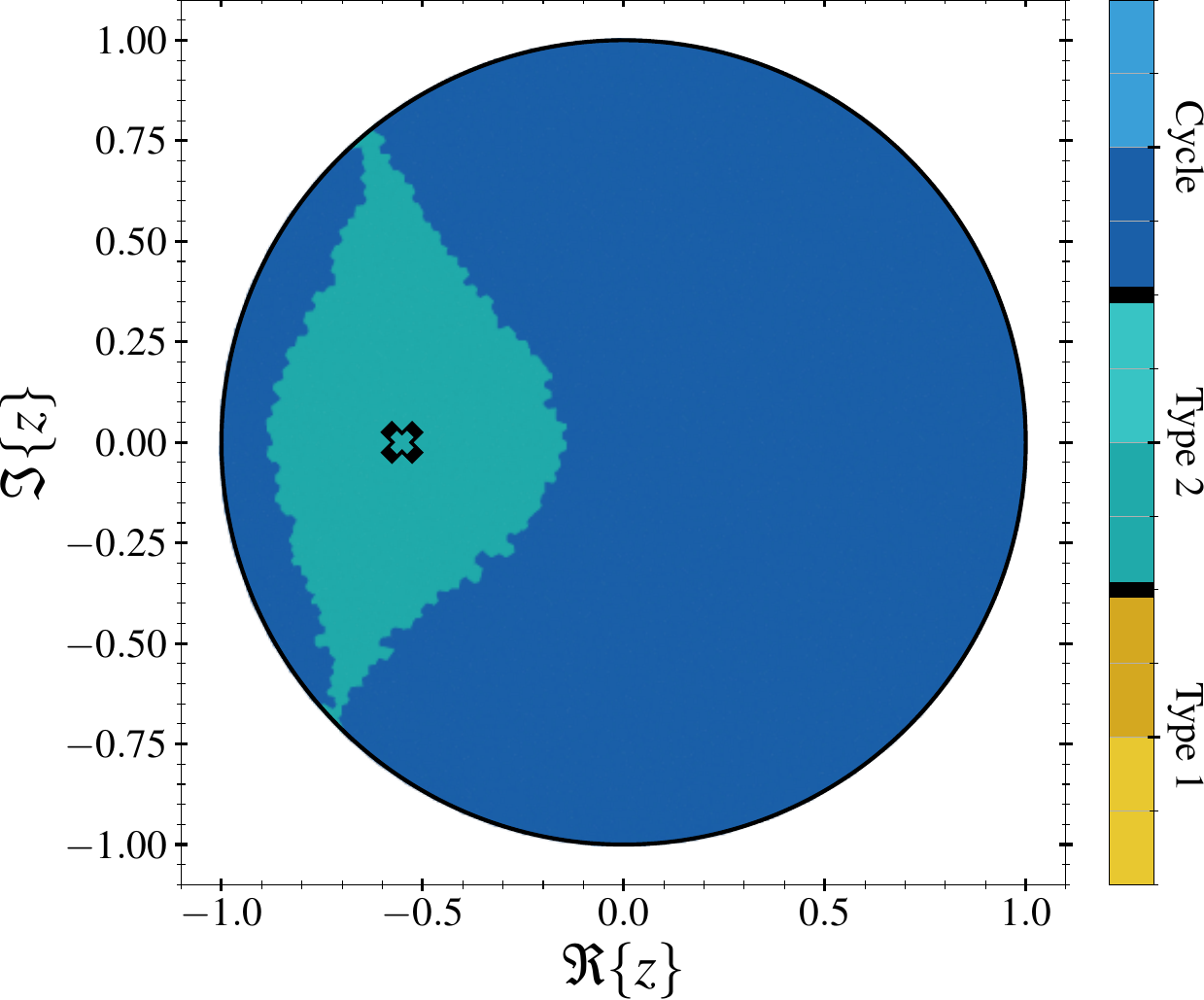}
\includegraphics[width=\linewidth]{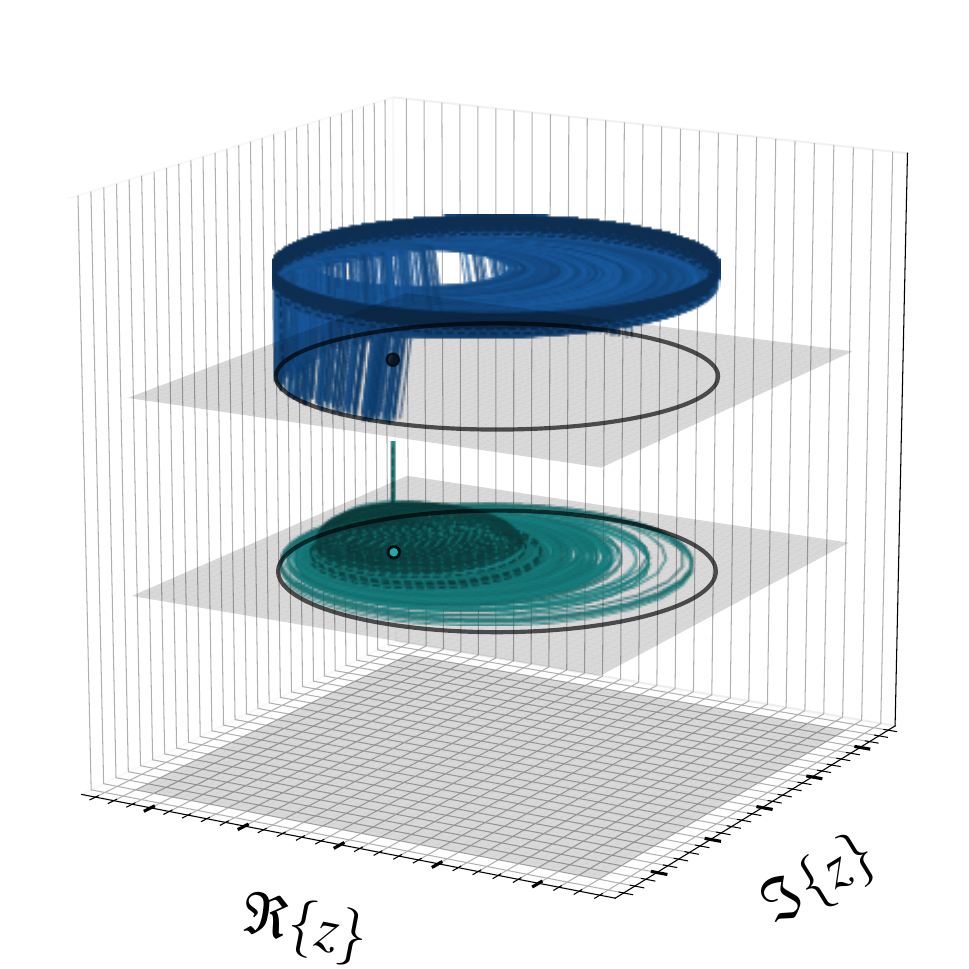}
\caption{$\tau = 2.05$}
\end{subfigure}
\caption{Set 1: (top) Regions of attraction and (bottom) trajectories for $(\kappa,
\eta) = (4.0, 0.92)$. The trajectories are shown for all types of behavior of interest:
(top) limit cycles; (middle) type 2 equilibria; and (bottom) type 1 equilibria.}
\label{fig:num.set1}
\end{figure}

For $(\kappa,\eta)=(4.0,0.92)$, no type 1 equilibria are present and the delay-free
system has a unique type 2 equilibrium $z_c^\star$, which is a center. Since $\kappa>0$,
\Cref{prop:type2.small.delay} implies that $z_c^\star$ becomes asymptotically stable for
sufficiently small delays. The first four Hopf thresholds for the discrete
time delay are listed in \Cref{tab:representative.parameters}. By the transversality
formula \eqref{eq:type2.transversality}, the crossings with odd index are
destabilizing, whereas those with even index are stabilizing.

The simulations in \Cref{fig:num.set1} illustrate this alternating sequence of stability
loss and recovery of the type 2 equilibrium. For $0<\tau<\tau_1^{(2)}$, the equilibrium
$z_c^\star$, which is neutrally stable at $\tau=0$, is stabilized by the delay. After
the first Hopf threshold $\tau_1^{(2)}$, this equilibrium loses stability, and the
simulations show the appearance of an attracting periodic orbit. At $\tau_2^{(2)}$, the
equilibrium regains stability and the corresponding basin plots indicate that the stable
equilibrium coexists with an attracting periodic orbit. The next threshold,
$\tau_3^{(2)}$, destabilizes $z_c^\star$ again, while $\tau_4^{(2)}$ produces another
recovery of stability. Therefore, Set 1 provides an example of delay-induced stability
switching along the type 2 branch, together with coexistence of $z_c^\star$ and an
attracting periodic orbit for some of the sampled delay values.

\subsection{Set 2: second quadrant, single type 2 equilibrium}

For $(\kappa,\eta)=(-0.1,0.9)$, the delay-free system has a single type 2 center
$z_c^\star$. Since $\kappa<0$, \Cref{prop:type2.small.delay} implies that this
equilibrium becomes unstable for sufficiently small positive delay. However, the
condition $\kappa<0$ and $\eta+4\kappa>0$ is satisfied, and
\Cref{prop:unit.circle.periodic.orbit} guarantees the existence of a periodic orbit on
the invariant circle $\rho=1$. The first two Hopf thresholds for the discrete delay are
given in \Cref{tab:representative.parameters}. The simulations in \Cref{fig:num.set2}
show that the unit-circle periodic orbit attracts a substantial part of the
constant-history slice when the type 2 equilibrium is unstable, while the type 2
equilibrium becomes attracting between the stabilizing and destabilizing thresholds.

\begin{figure}[ht!]
\centering
\setlength{\localcolumnwidth}{0.24\linewidth}

\begin{subfigure}{\localcolumnwidth}
\includegraphics[width=\linewidth]{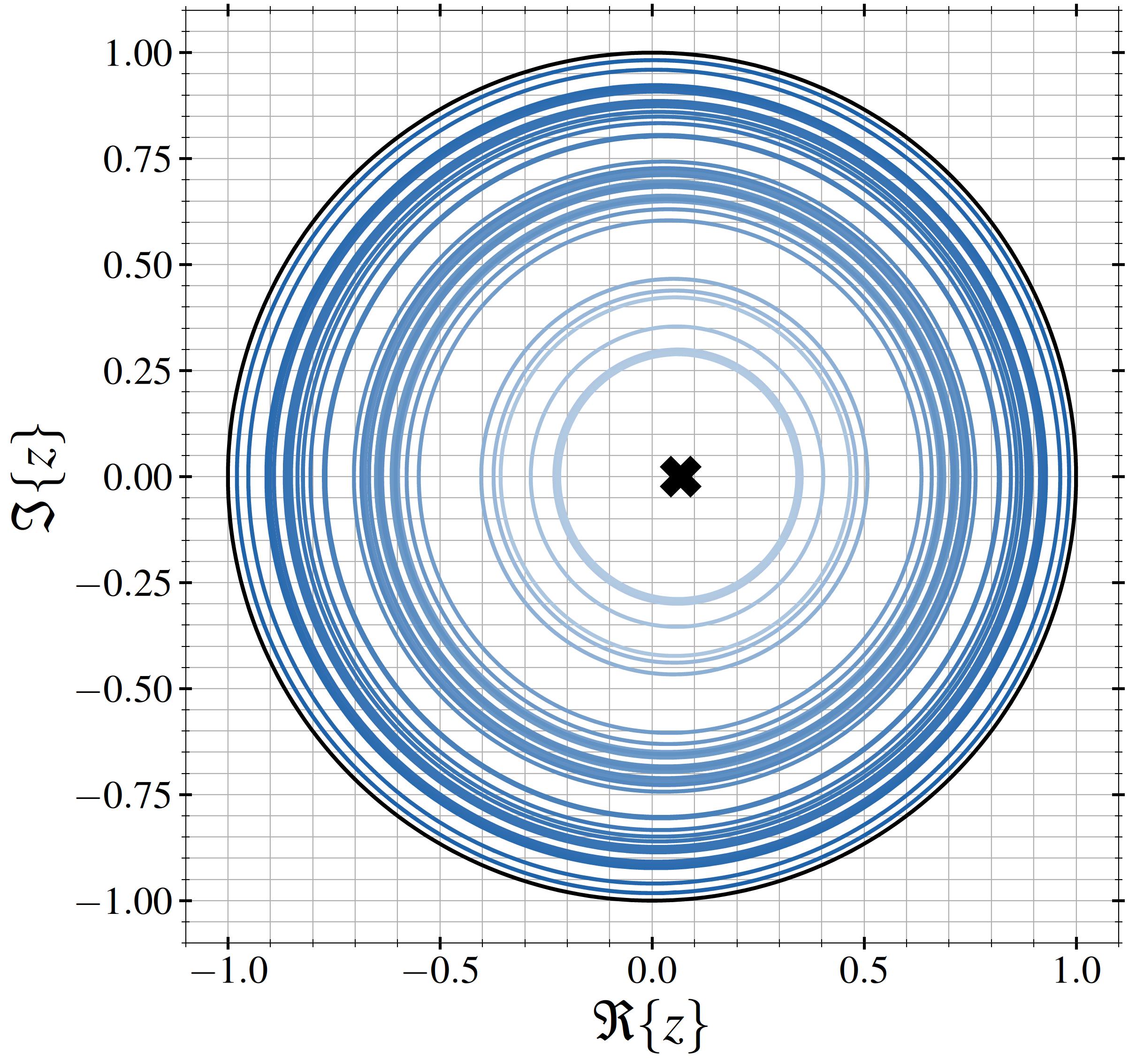}
\caption{$\tau = 0.0$}
\end{subfigure}%
\begin{subfigure}{\localcolumnwidth}
\includegraphics[width=\linewidth]{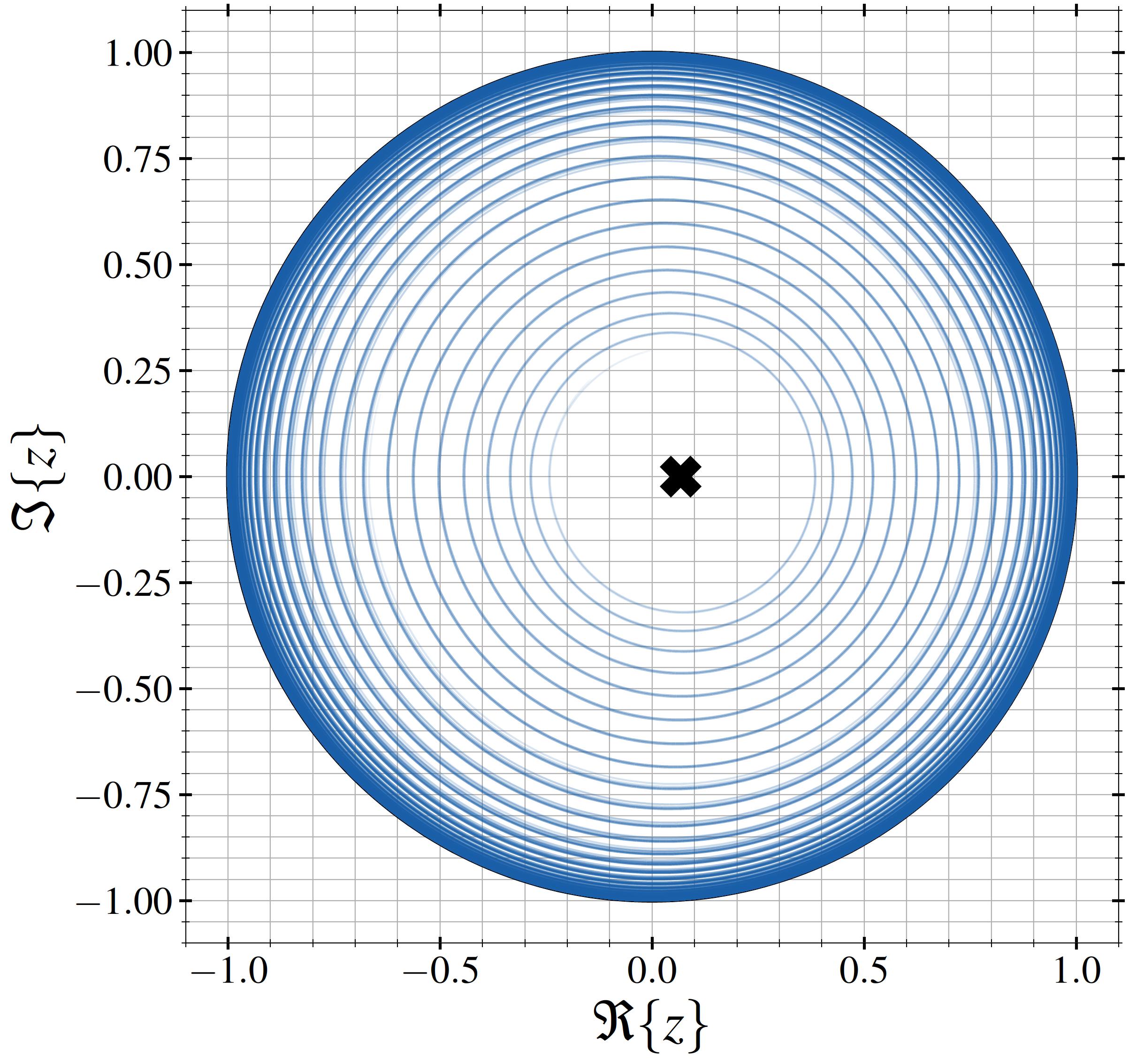}
\caption{$\tau = 0.5$}
\end{subfigure}%
\begin{subfigure}{\localcolumnwidth}
\includegraphics[width=\linewidth]{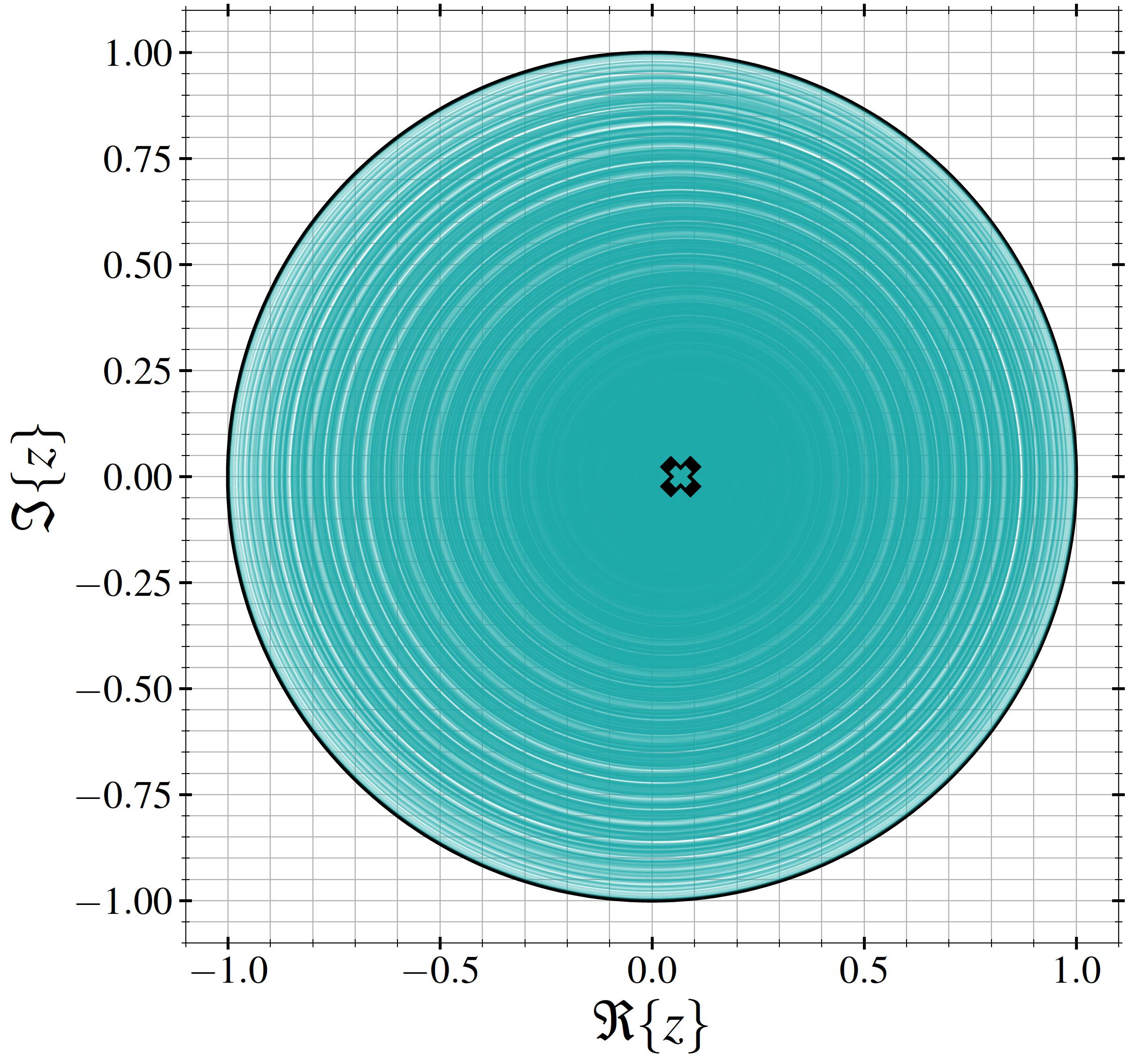}
\caption{$\tau = 2.5$}
\end{subfigure}%
\begin{subfigure}{\localcolumnwidth}
\includegraphics[width=\linewidth]{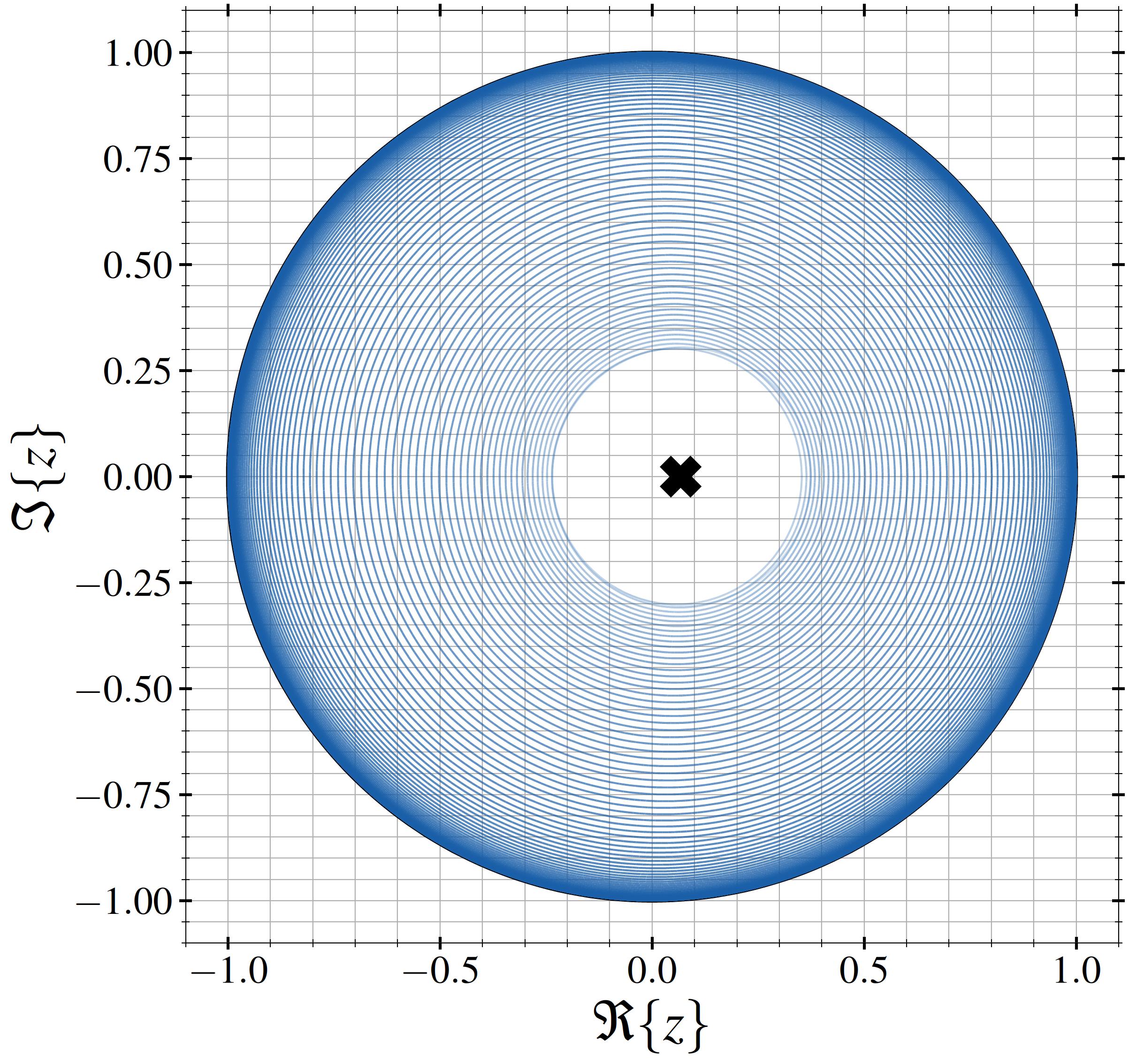}
\caption{$\tau = 3.6$}
\end{subfigure}

\caption{Set 2: Phase portraits for $(\kappa, \eta) = (-0.1, 0.9)$.}
\label{fig:num.set2}
\end{figure}

\subsection{Sets 3.1 and 3.2: third quadrant, two and six type 1 equilibria}

For $(\kappa,\eta)=(-3,-0.5)$ corresponding to Set 3.1, the delay-free system has two
type 1 equilibria, one of which is asymptotically stable, and no type 2 equilibria.
Since $\kappa<0$, \Cref{prop:stab.type1} implies that the stability of the type 1 branch
is independent of the delay. The simulations in \Cref{fig:num.set31} are consistent with
this prediction: increasing $\tau$ does not produce a qualitative change in the observed
attractor.

On the other hand, for $(\kappa,\eta)=(-1.3,-0.025)$ corresponding to Set 3.2, there are
six type 1 equilibria, exactly two of them (on the lower semi-circle) being
asymptotically stable, regardless of the considered delay. However, the basin plots in
\Cref{fig:num.set32} show that the two stable equilibria do not attract comparable
portions of the constant-history slice, as the stability region associated with one of
them is substantially smaller.

These examples serve as reference cases in which the introduction of a discrete delay has little effect on the global portrait.

\begin{figure}[ht!]
\centering

\begin{subfigure}[t]{0.49\linewidth}
\centering

\begin{minipage}[t]{0.48\linewidth}
\centering
\includegraphics[width=\linewidth]{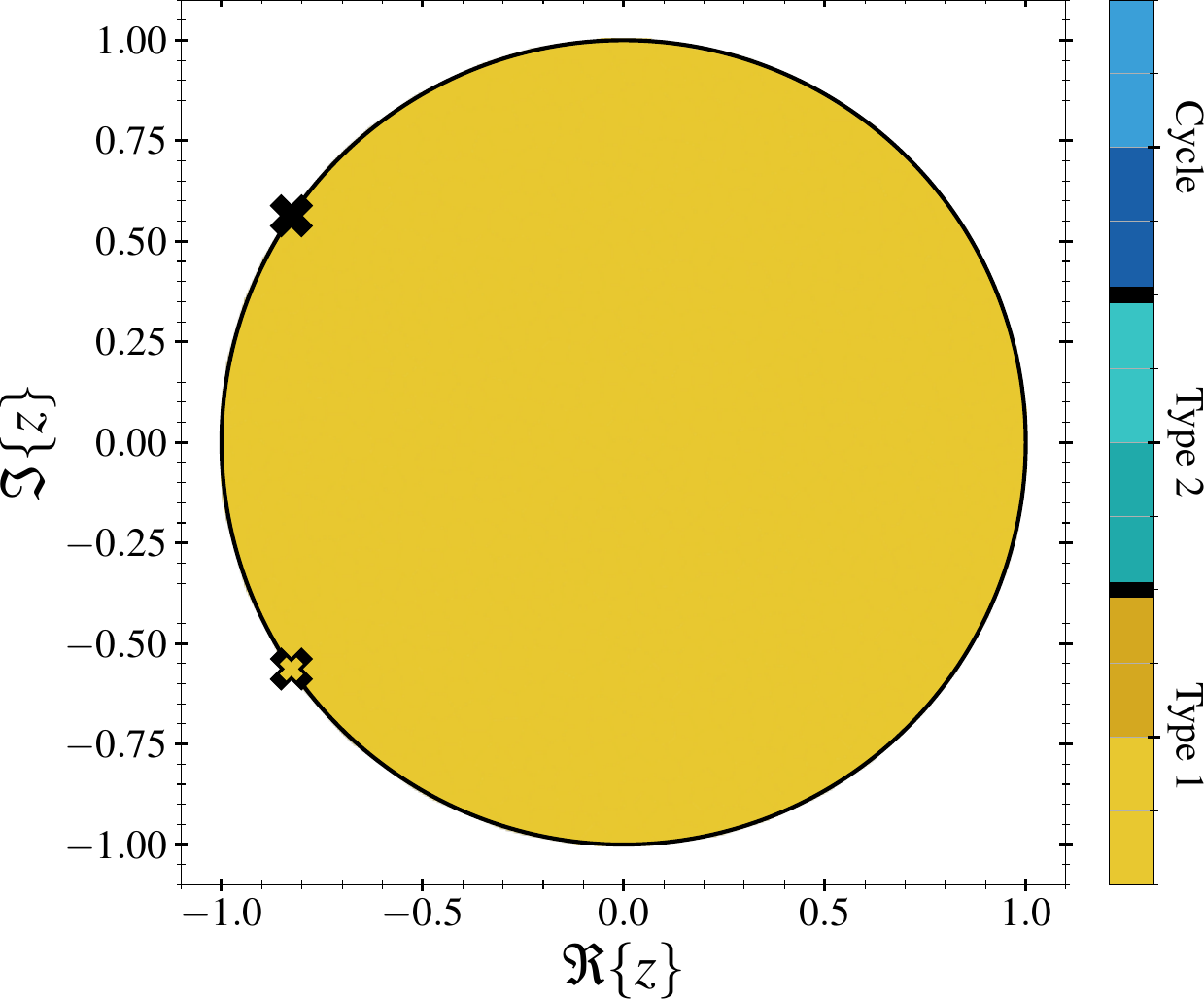}
\includegraphics[width=\linewidth]{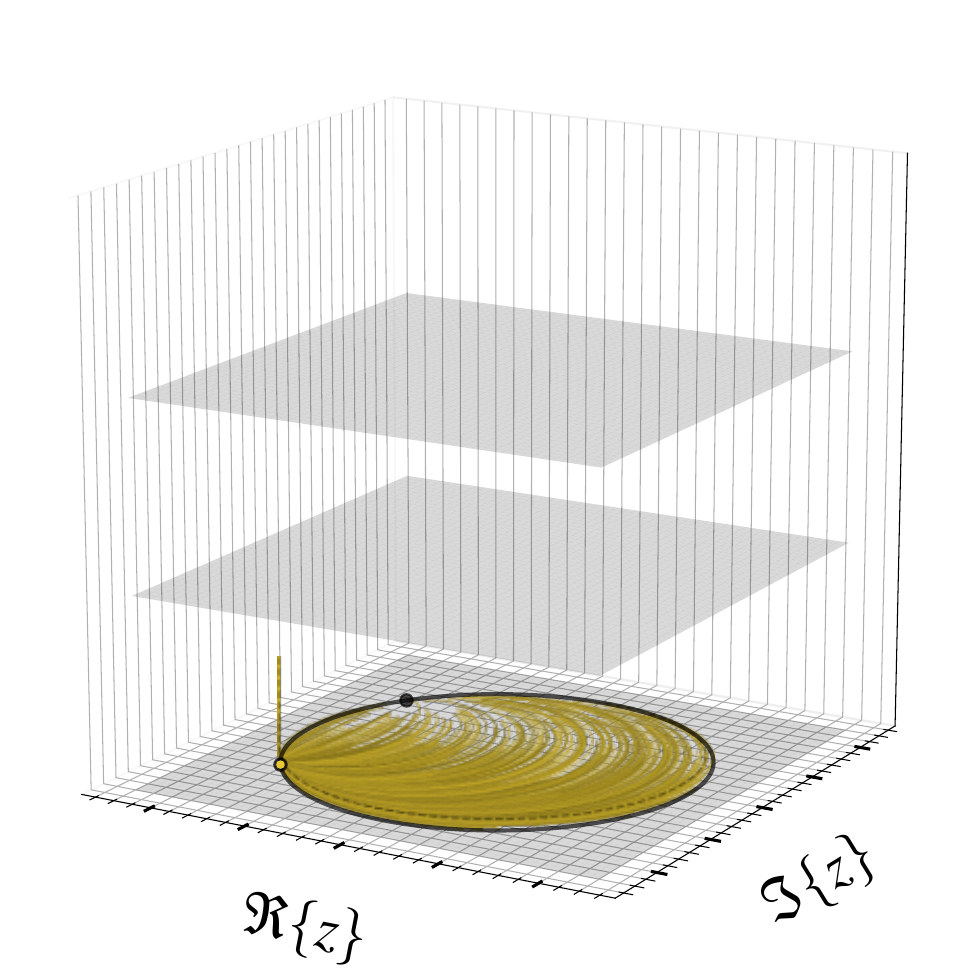}
\par\smallskip
{\small $\tau=0$}
\end{minipage}\hfill
\begin{minipage}[t]{0.48\linewidth}
\centering
\includegraphics[width=\linewidth]{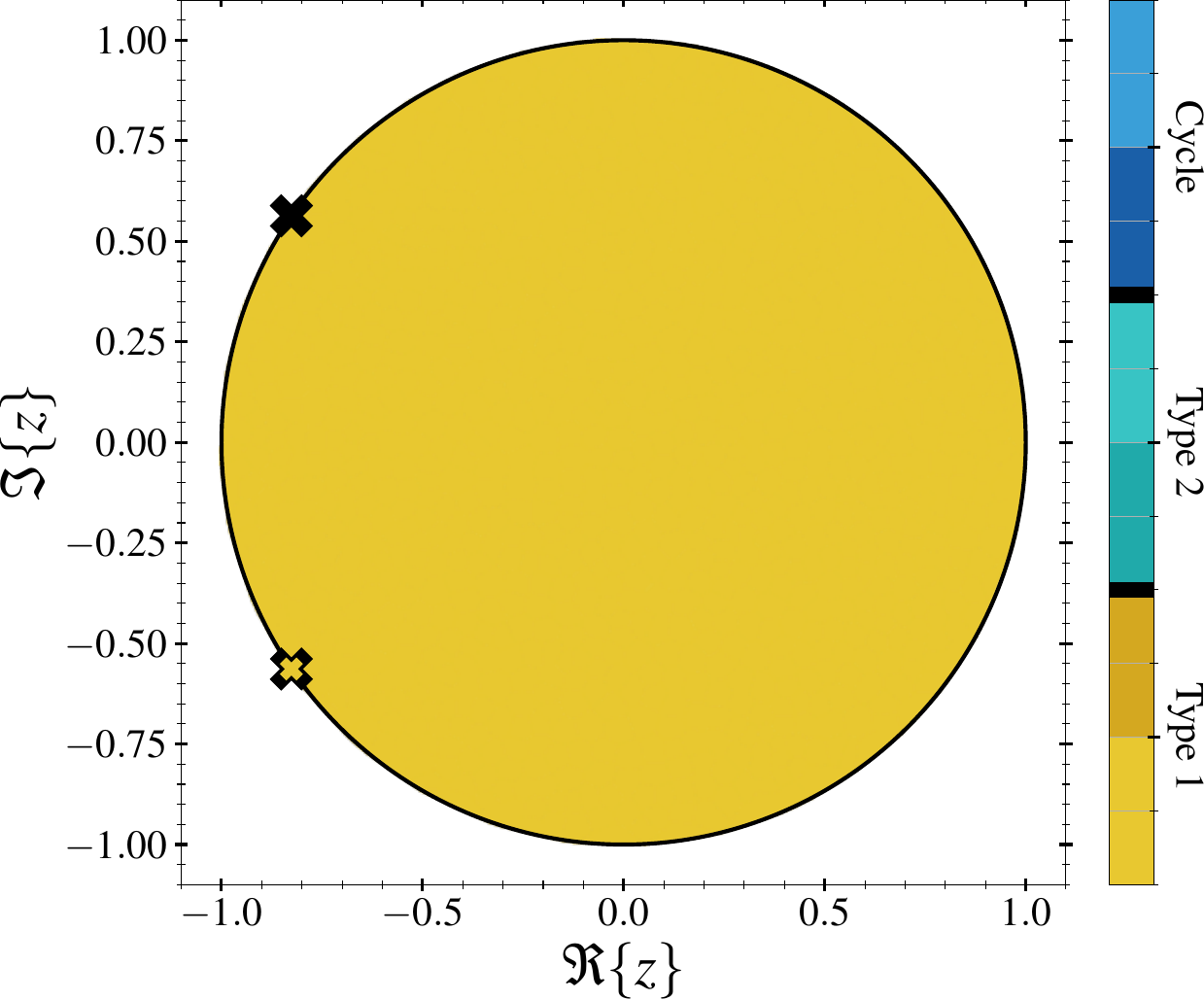}
\includegraphics[width=\linewidth]{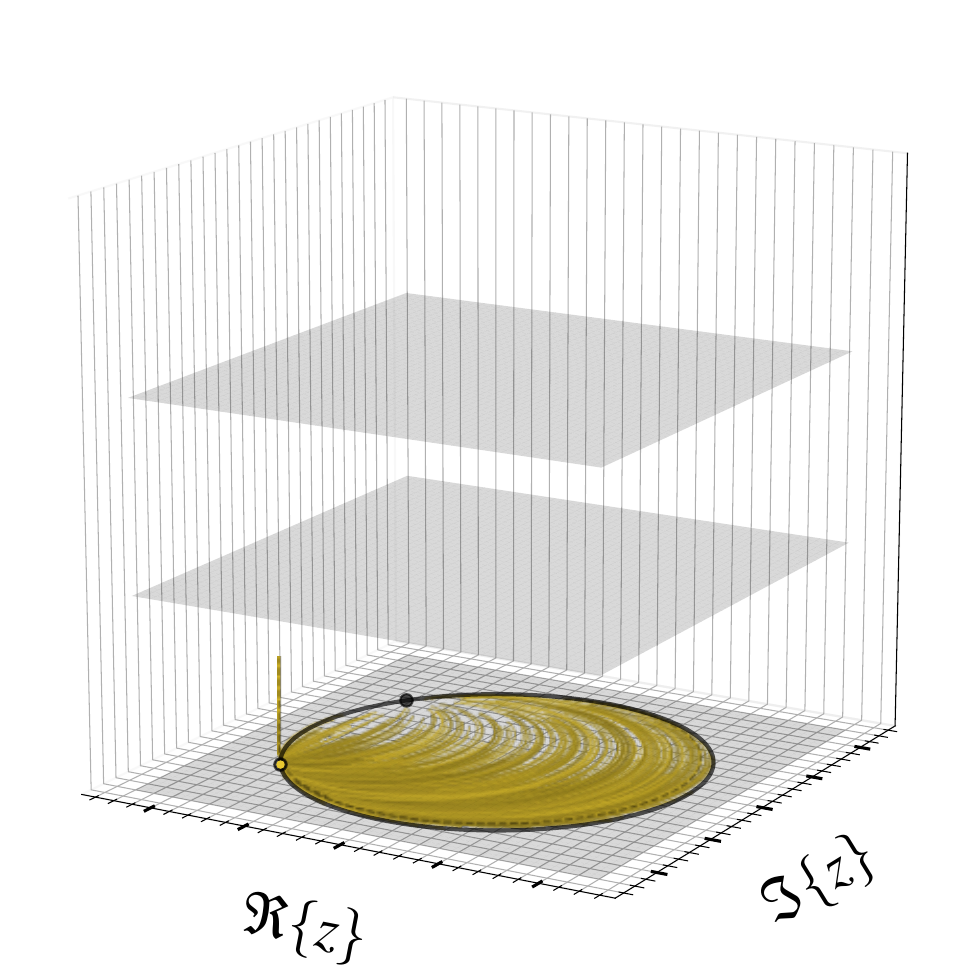}
\par\smallskip
{\small $\tau=1$}
\end{minipage}

\caption{Set~3.1: $(\kappa,\eta)=(-3,-0.5)$.}
\label{fig:num.set31}
\end{subfigure}
\hfill
\begin{subfigure}[t]{0.49\linewidth}
\centering

\begin{minipage}[t]{0.48\linewidth}
\centering
\includegraphics[width=\linewidth]{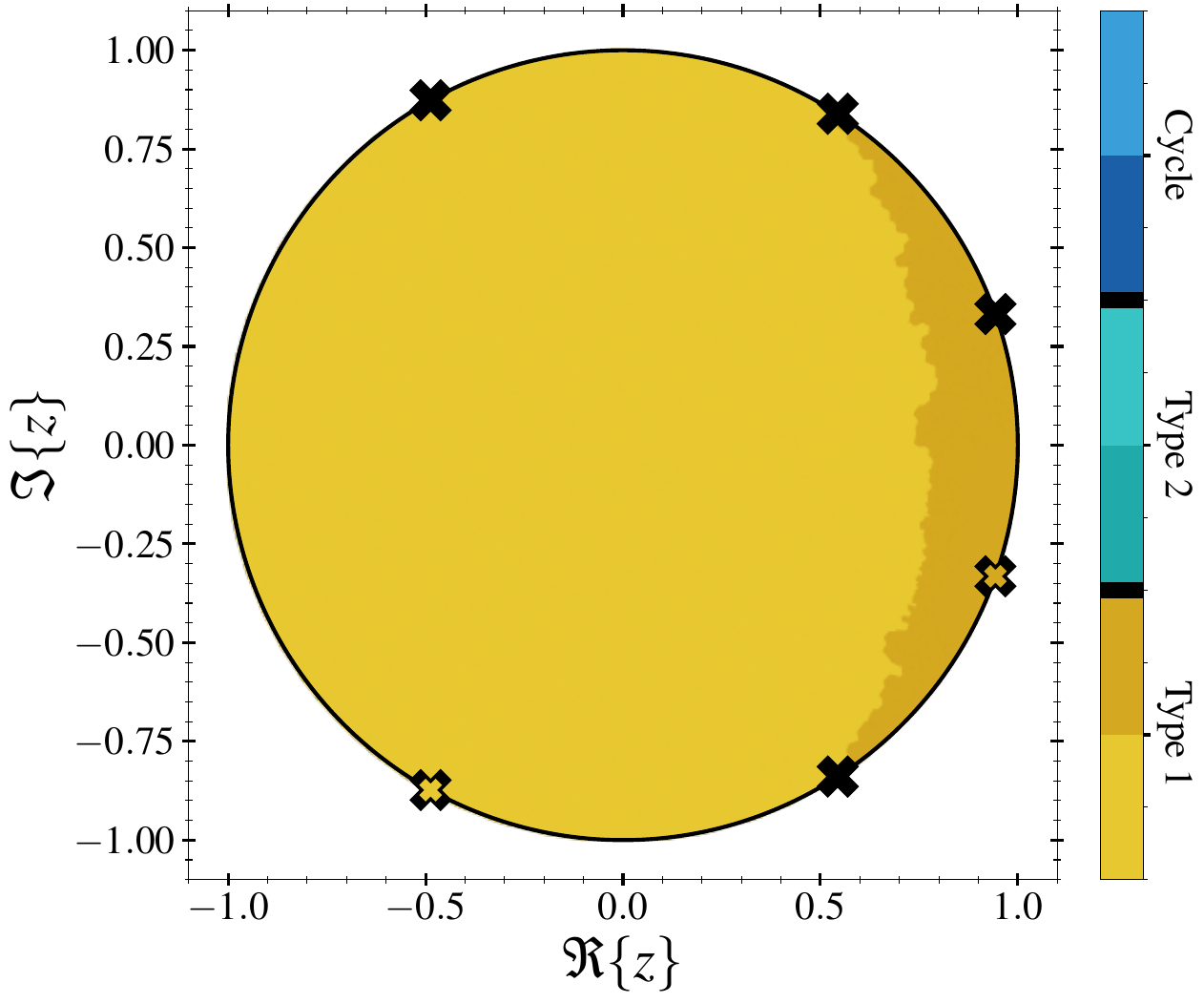}
\includegraphics[width=\linewidth]{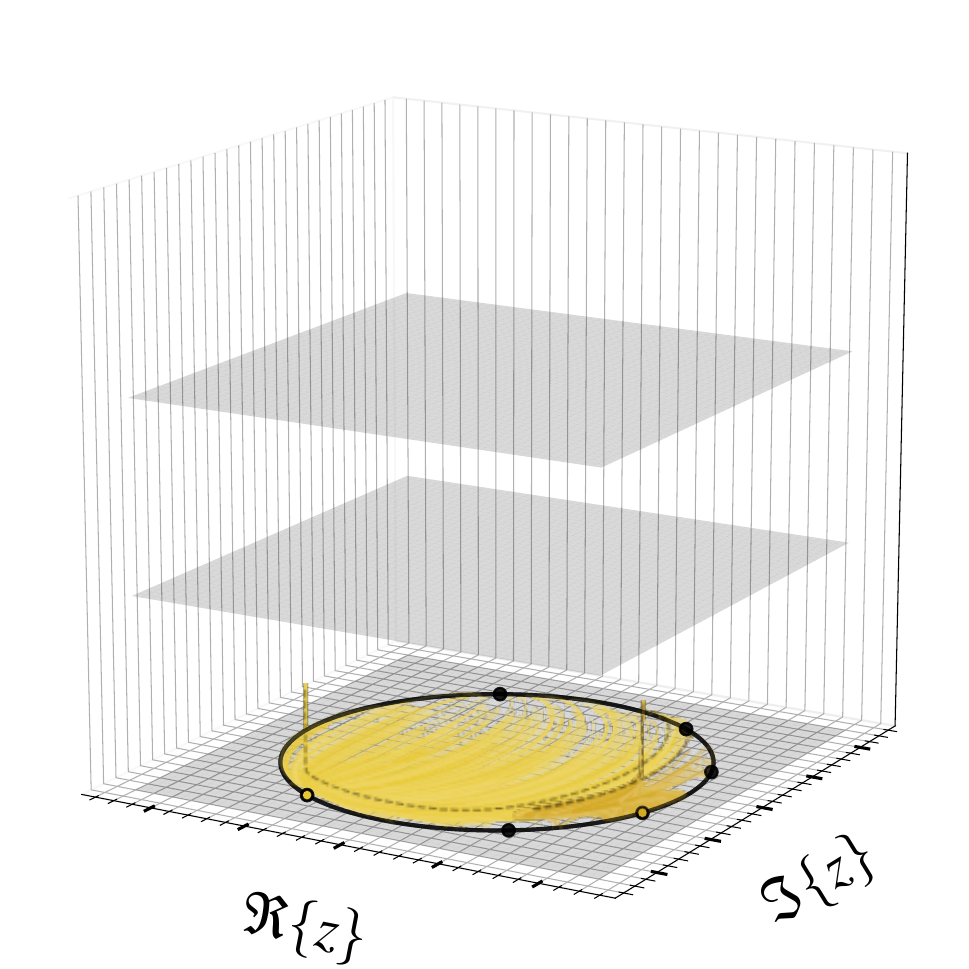}
\par\smallskip
{\small $\tau=0$}
\end{minipage}\hfill
\begin{minipage}[t]{0.48\linewidth}
\centering
\includegraphics[width=\linewidth]{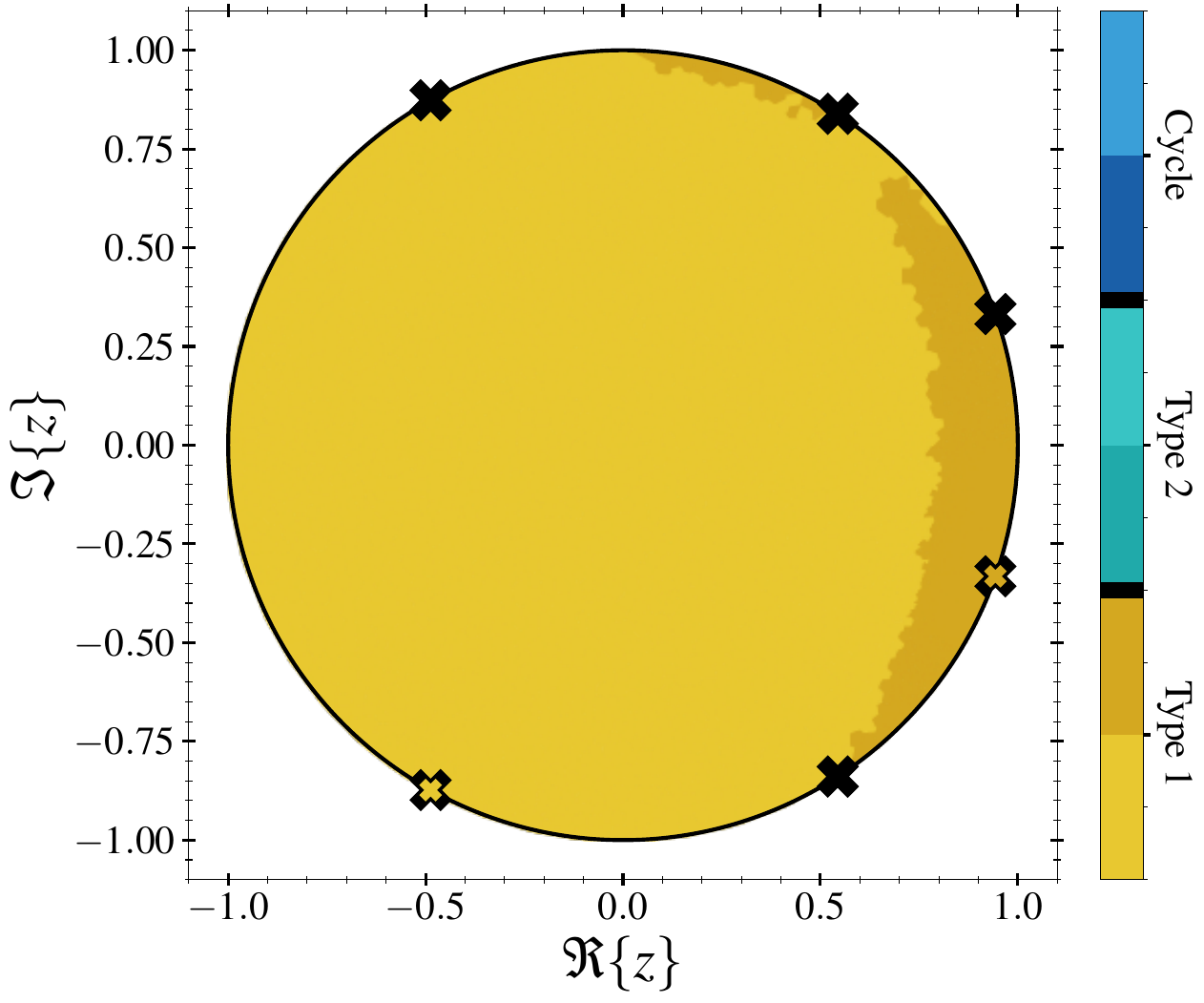}
\includegraphics[width=\linewidth]{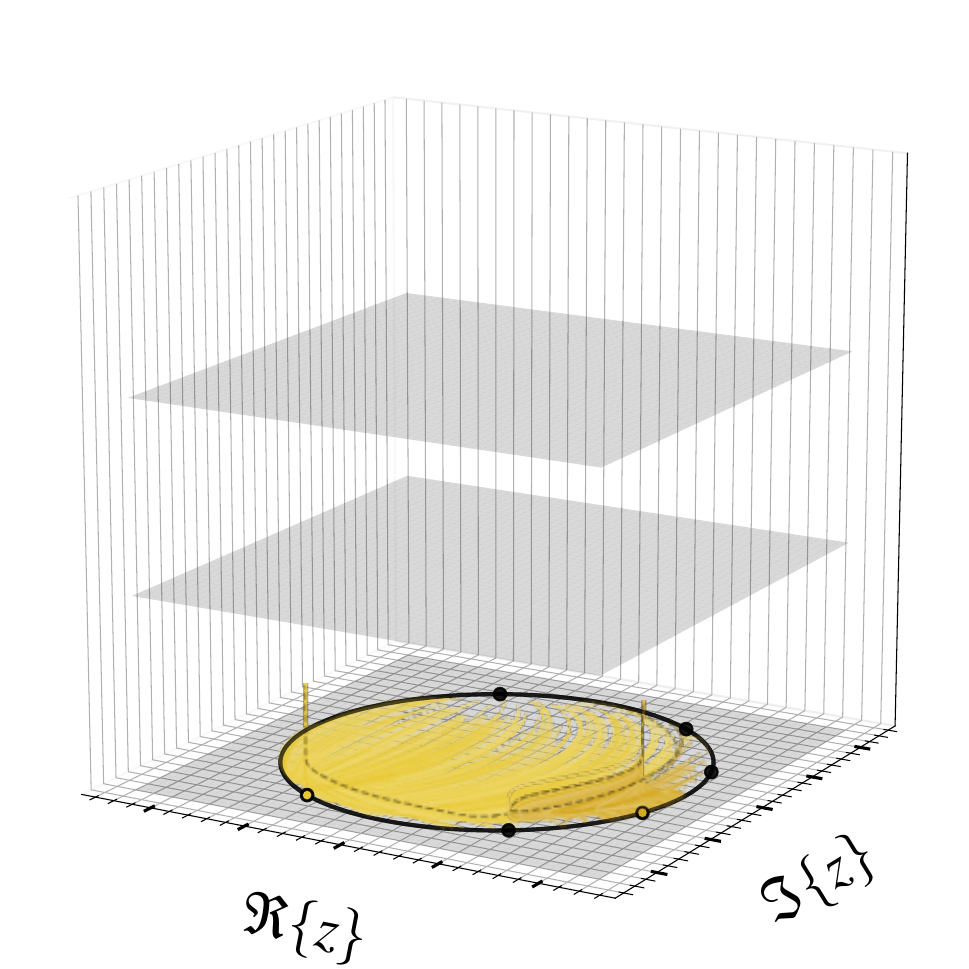}
\par\smallskip
{\small $\tau=1$}
\end{minipage}

\caption{Set~3.2: $(\kappa,\eta) = (-1.3, -0.025)$.}
\label{fig:num.set32}
\end{subfigure}

\caption{Regions of attraction and representative trajectories for the two
third-quadrant parameter sets. In each panel, the upper image shows the regions of
attraction and the lower image shows representative trajectories.}
\label{fig:num.set3}
\end{figure}

\subsection{Set 4: fourth quadrant, interaction of type 1 and type 2 mechanisms}

For $(\kappa,\eta)=(3,-0.4)$, the delay-free system has an asymptotically stable type 1
equilibrium $z_a^\star$ on the lower semicircle, an unstable type 1 equilibrium on the
upper semicircle, and two type 2 equilibria, namely a center $z_c^\star$ and a saddle.
For a small positive delay, the type 2 center $z_c^\star$ becomes asymptotically stable
because $\kappa>0$ (see \Cref{prop:type2.small.delay}), while the type 1 equilibrium
$z_a^\star$ remains asymptotically stable until its first type 1 Hopf threshold. Hence,
two asymptotically stable equilibria coexist.

The relevant critical delays are given in \Cref{tab:representative.parameters}. The
selected delay values first cross the type 2 destabilizing Hopf threshold
$\tau_1^{(2)}$, then the type 1 Hopf threshold $\tau_0^{(1)}$, and then the subsequent
type 2 restabilization/destabilization thresholds $\tau_2^{(2)}$ and $\tau_3^{(2)}$.
\Cref{fig:num.set4} shows the corresponding changes in the observed attractors. More
precisely, for small values of the delay, the equilibrium $z_c^\star$ becomes
asymptotically stable and it coexists with the asymptotically stable type 1 equilibrium
$z_a^\star$. At the critical value $\tau_1^{(2)}$, the stability of $z_c^\star$ is lost
by a Hopf bifurcation, and a stable cycle appears. At the Hopf threshold $\tau_0^{(1)}$,
the type 1 equilibrium $z_a^\star$ also loses stability, and the coexistence of two
stable cycles is observed. However, at $\tau_2^{(2)}$, the equilibrium $z_c^\star$
regains stability, and coexists with two attracting cycles. Finally, at $\tau_3^{(2)}$,
the equilibrium $z_c^\star$ loses stability once more through a Hopf bifurcation,
and two coexisting stable cycles are observed in the system.

\begin{figure}[ht!]
\centering
\setlength{\localcolumnwidth}{0.2\linewidth}

\begin{subfigure}{\localcolumnwidth}
\includegraphics[width=\linewidth]{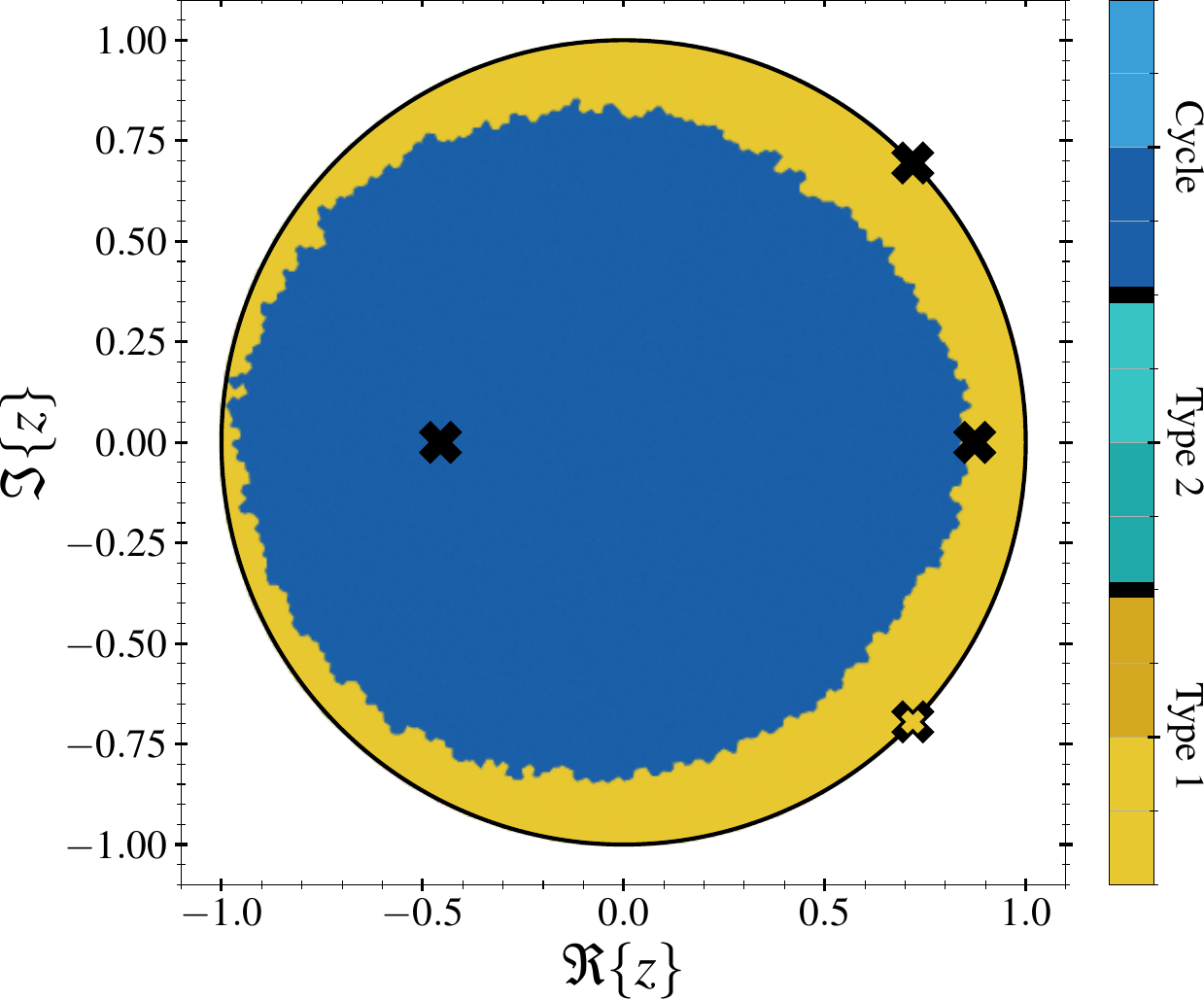}
\includegraphics[width=\linewidth]{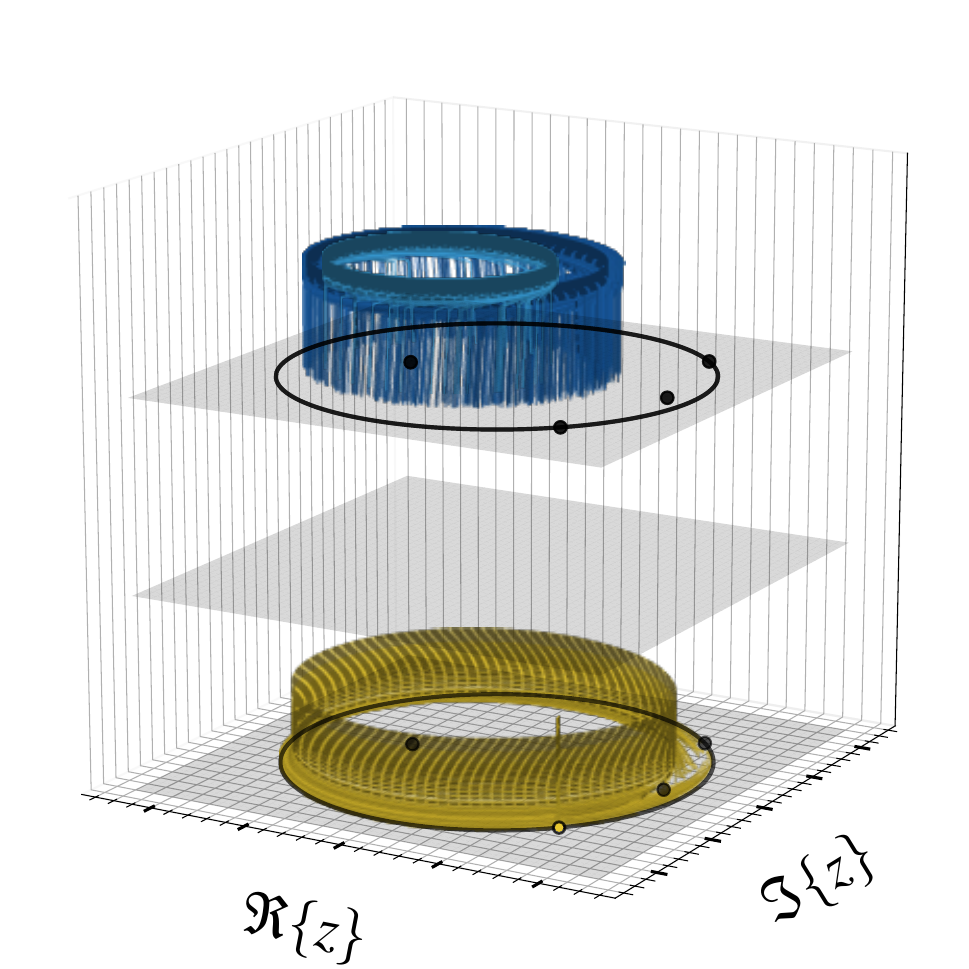}
\caption{$\tau = 0$}
\end{subfigure}
\begin{subfigure}{\localcolumnwidth}
\includegraphics[width=\linewidth]{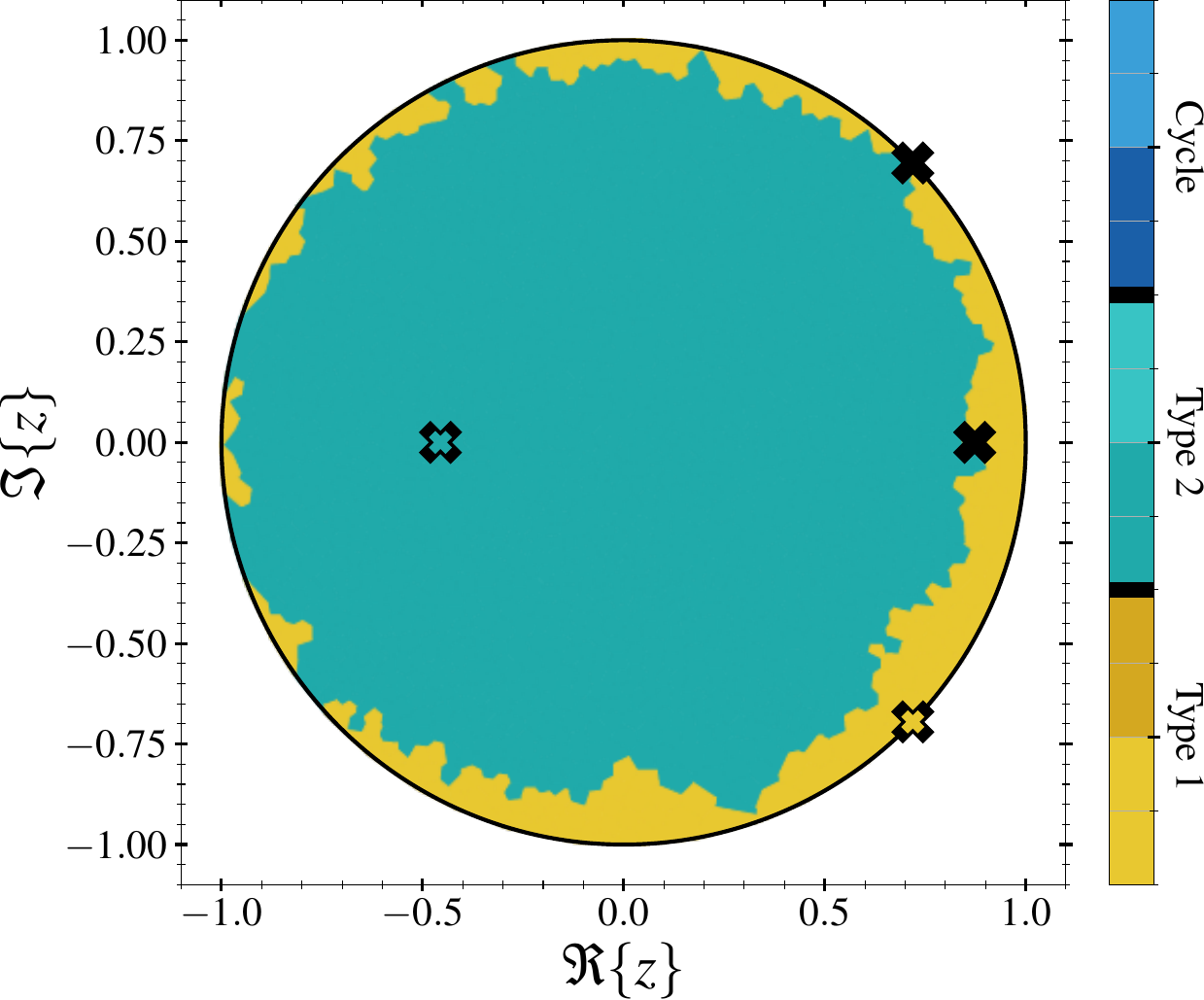}
\includegraphics[width=\linewidth]{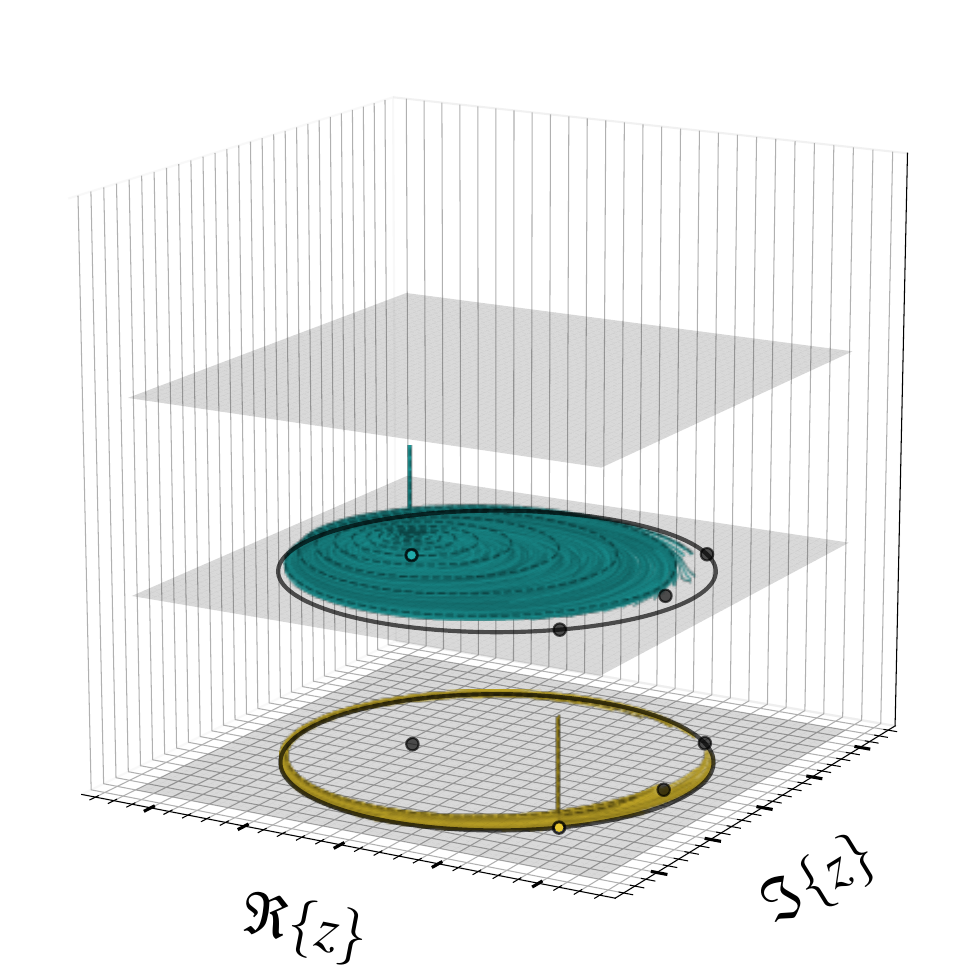}
\caption{$\tau = 0.1$}
\end{subfigure}
\begin{subfigure}{\localcolumnwidth}
\includegraphics[width=\linewidth]{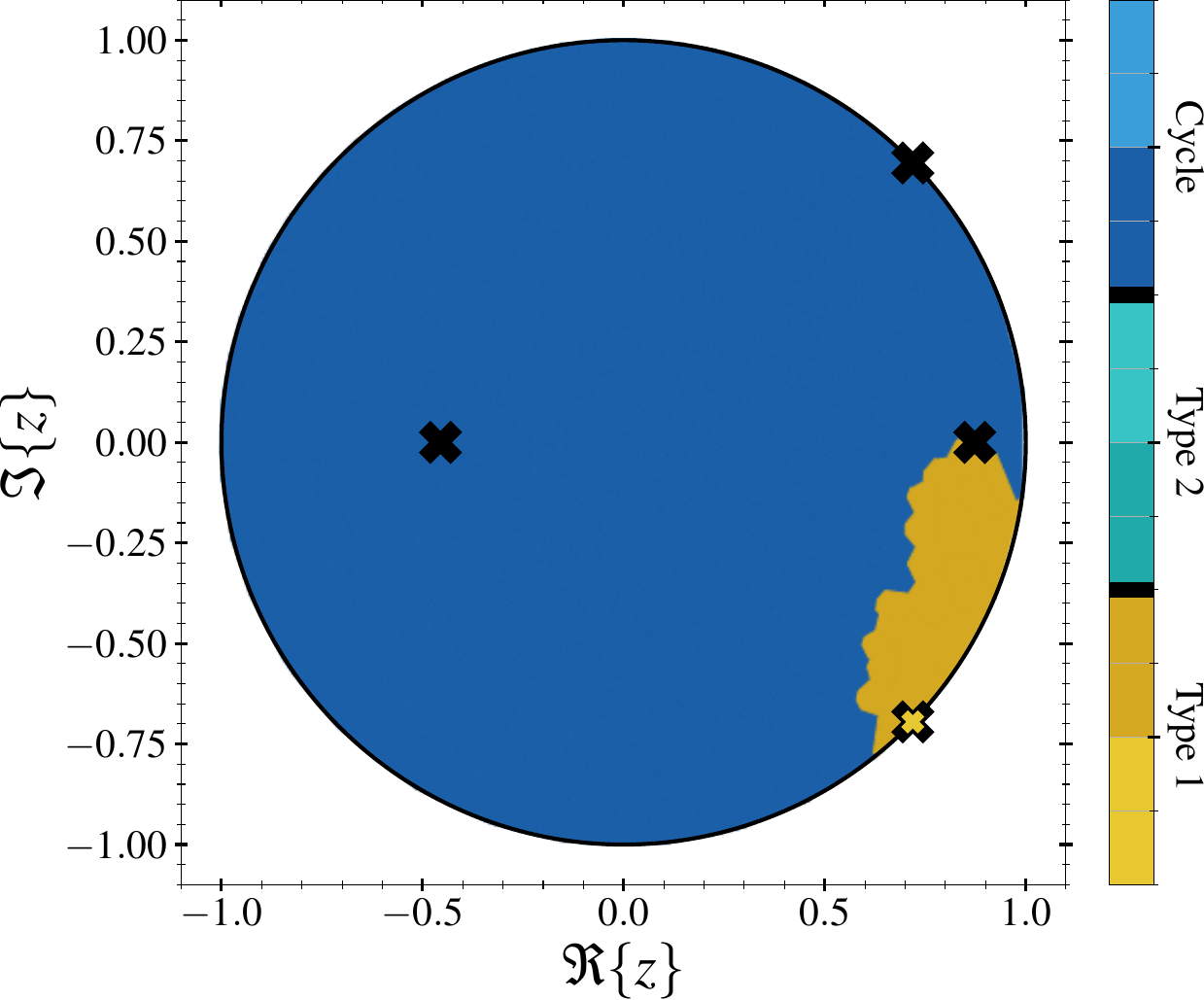}
\includegraphics[width=\linewidth]{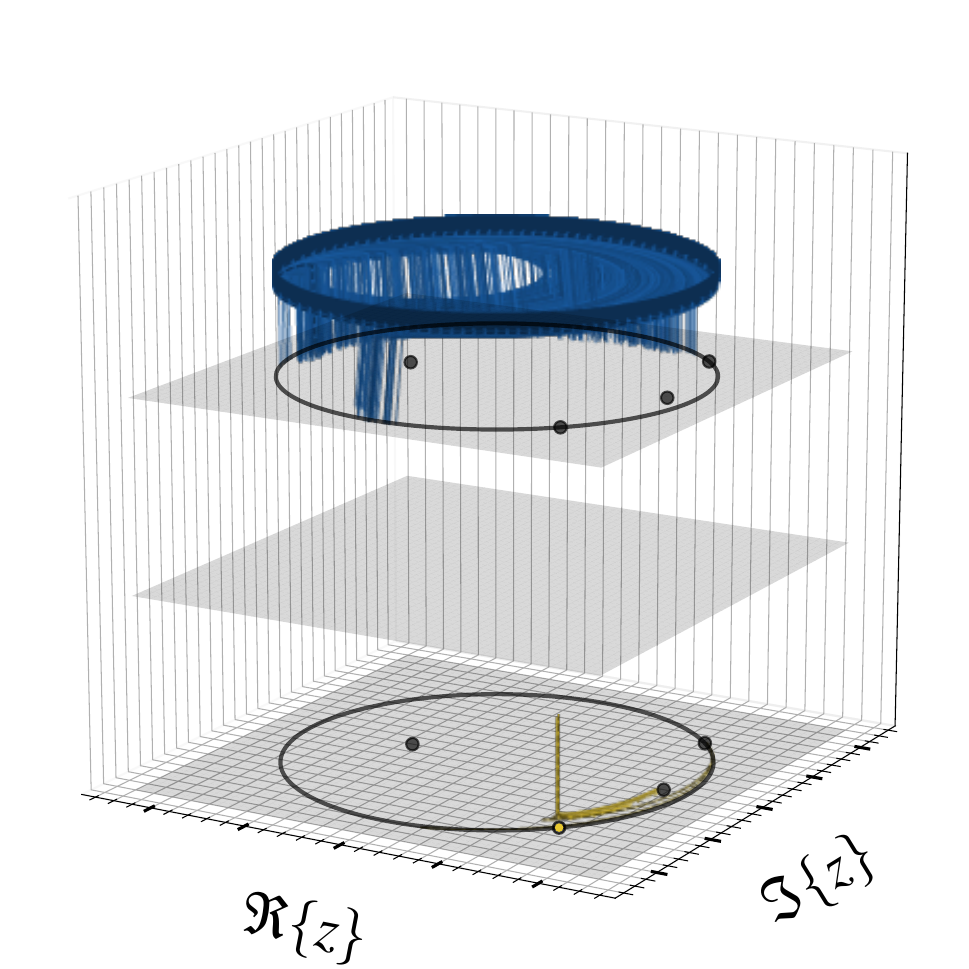}
\caption{$\tau = 0.6$}
\end{subfigure}
\begin{subfigure}{\localcolumnwidth}
\includegraphics[width=\linewidth]{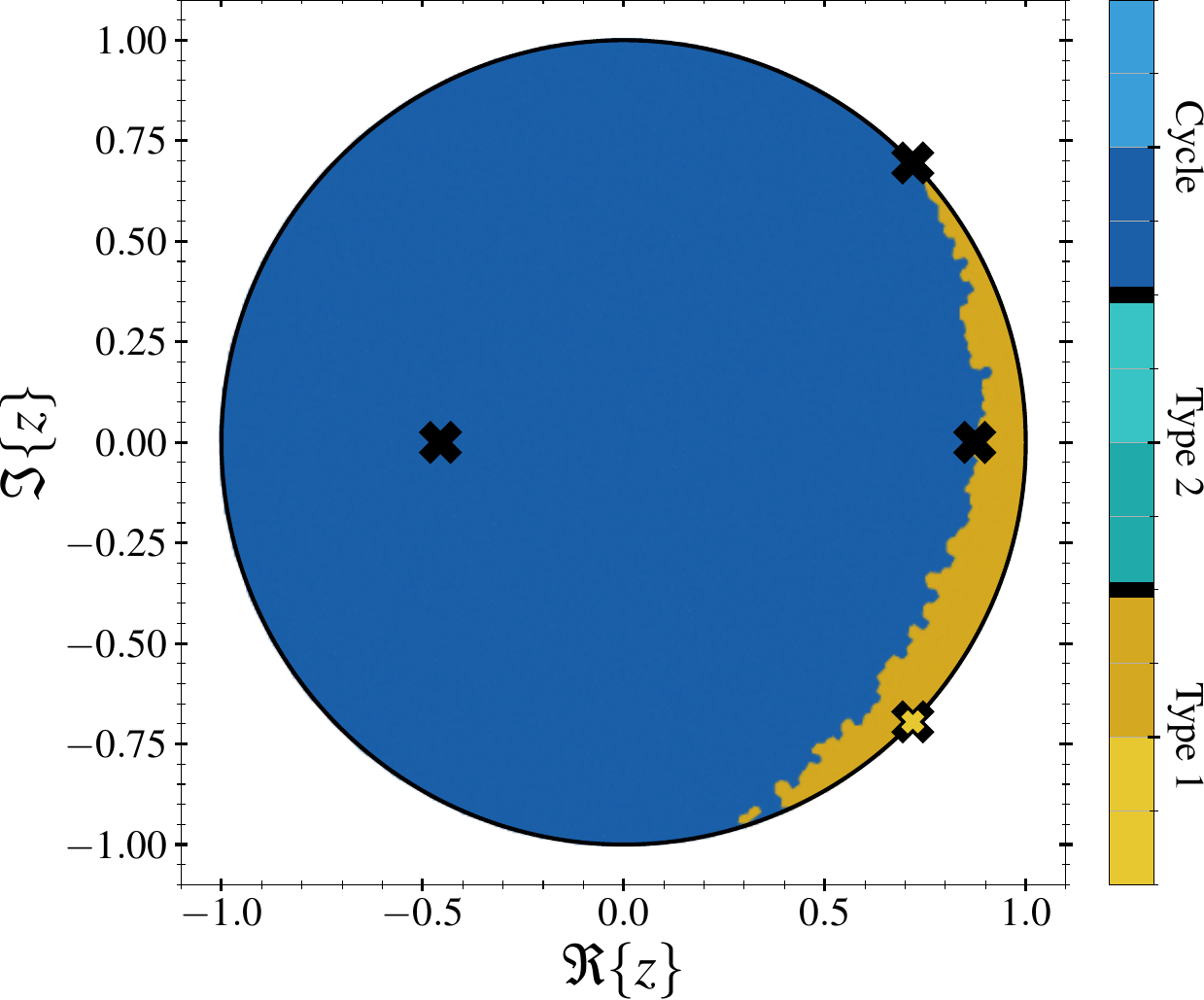}
\includegraphics[width=\linewidth]{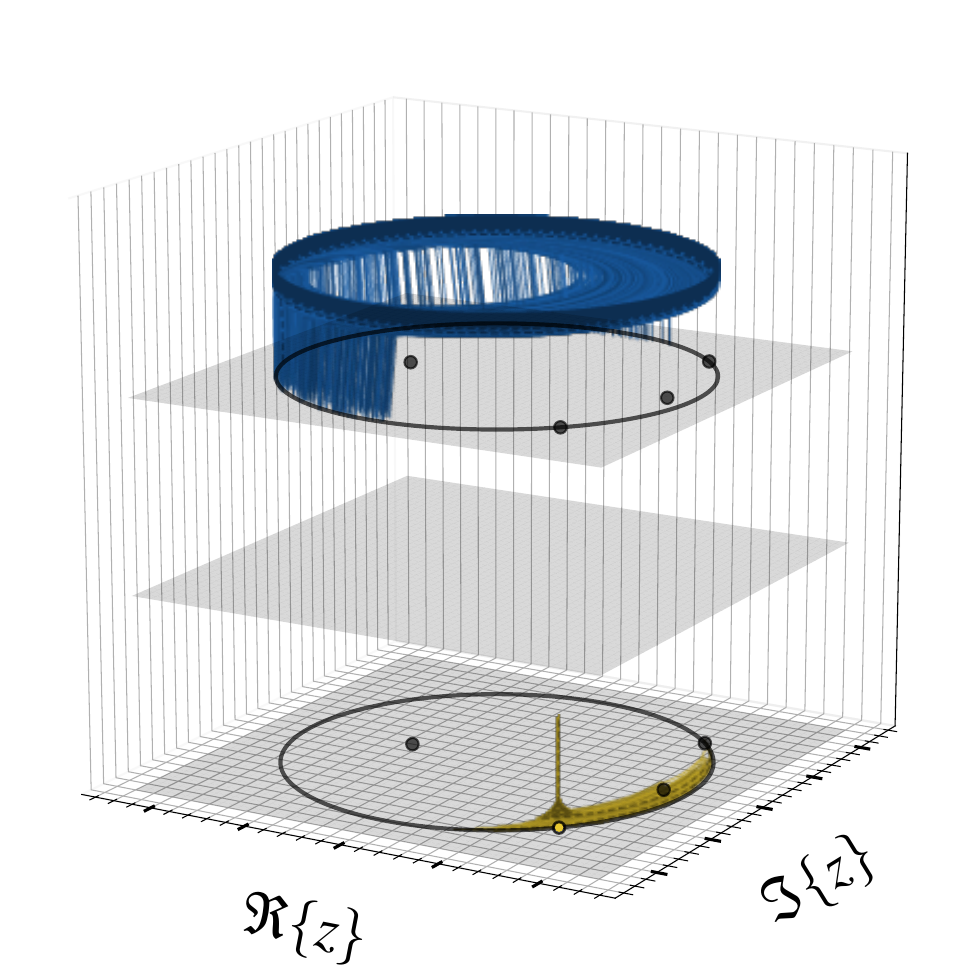}
\caption{$\tau = 0.9$}
\end{subfigure}%

\begin{subfigure}{\localcolumnwidth}
\includegraphics[width=\linewidth]{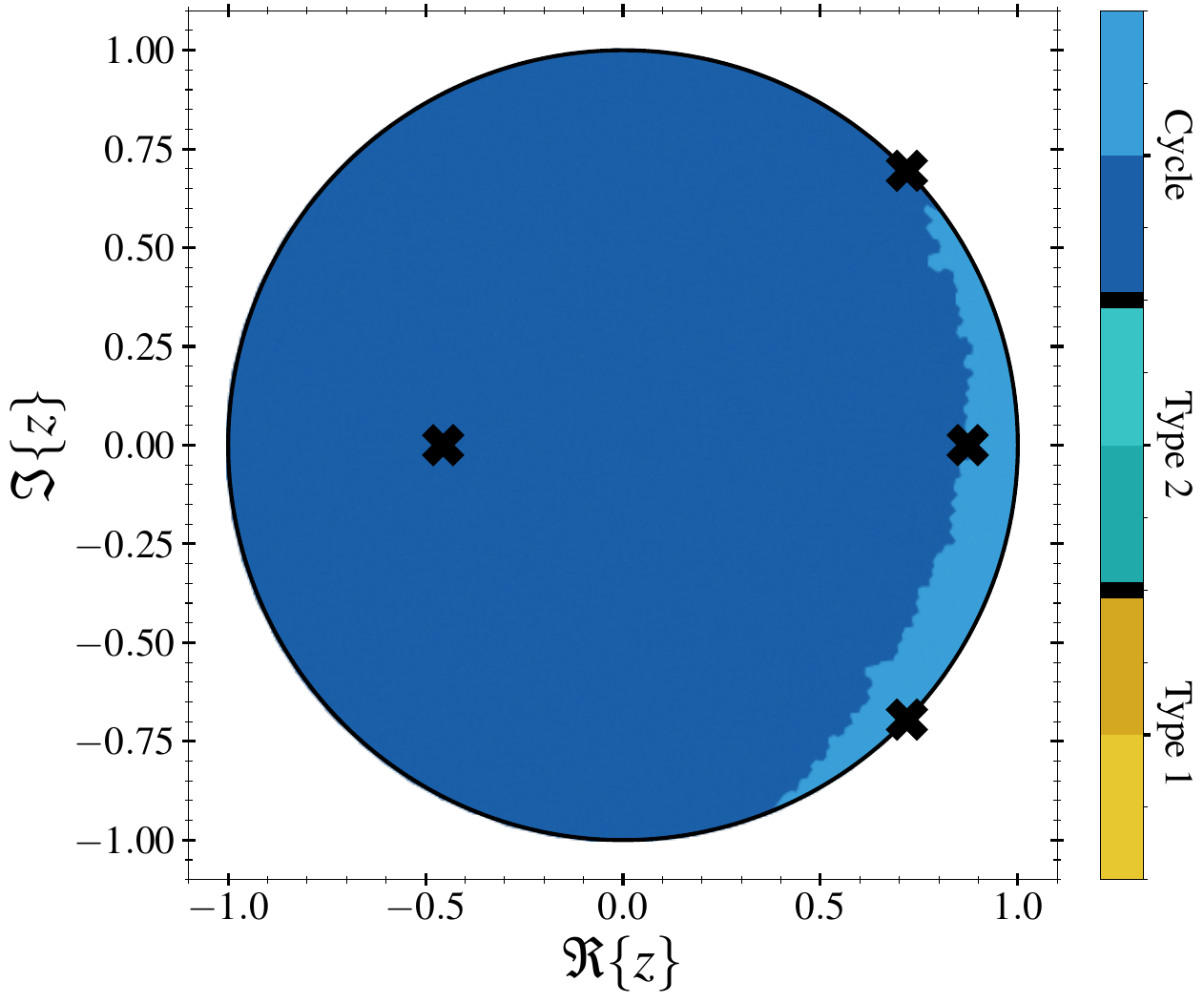}
\includegraphics[width=\linewidth]{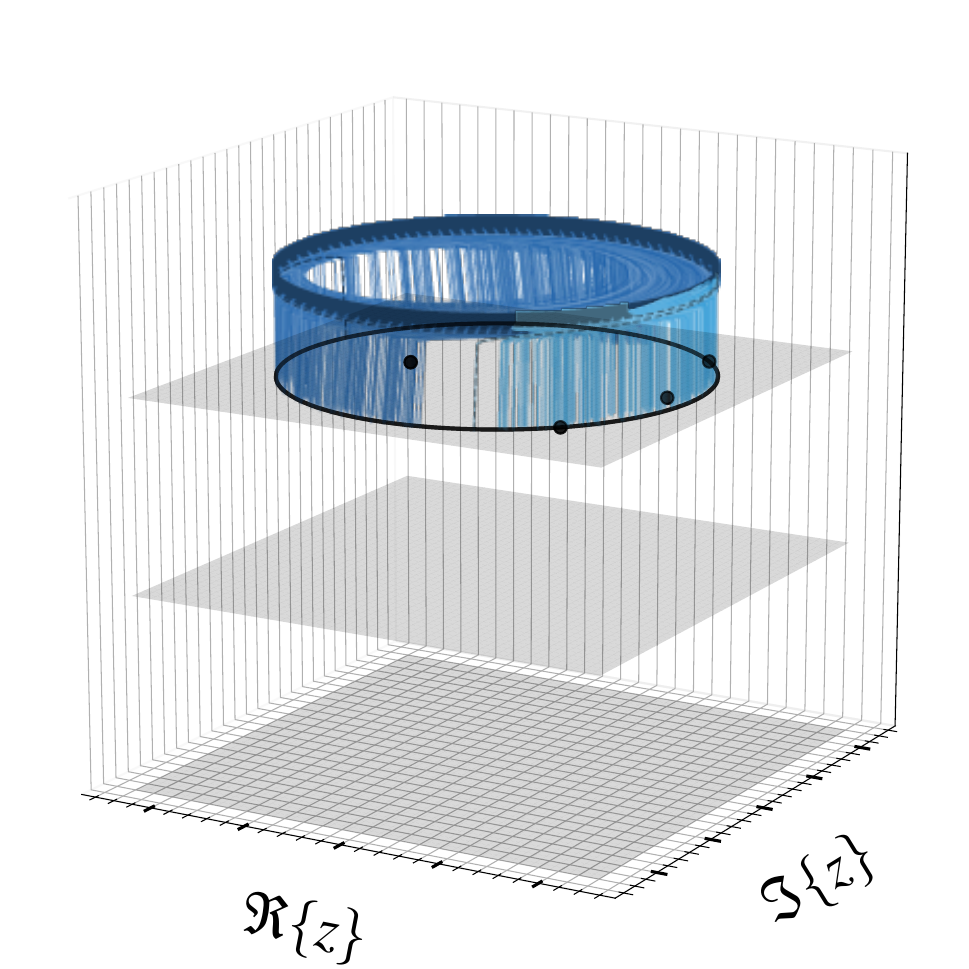}
\caption{$\tau = 1.25$}
\end{subfigure}
\begin{subfigure}{\localcolumnwidth}
\includegraphics[width=\linewidth]{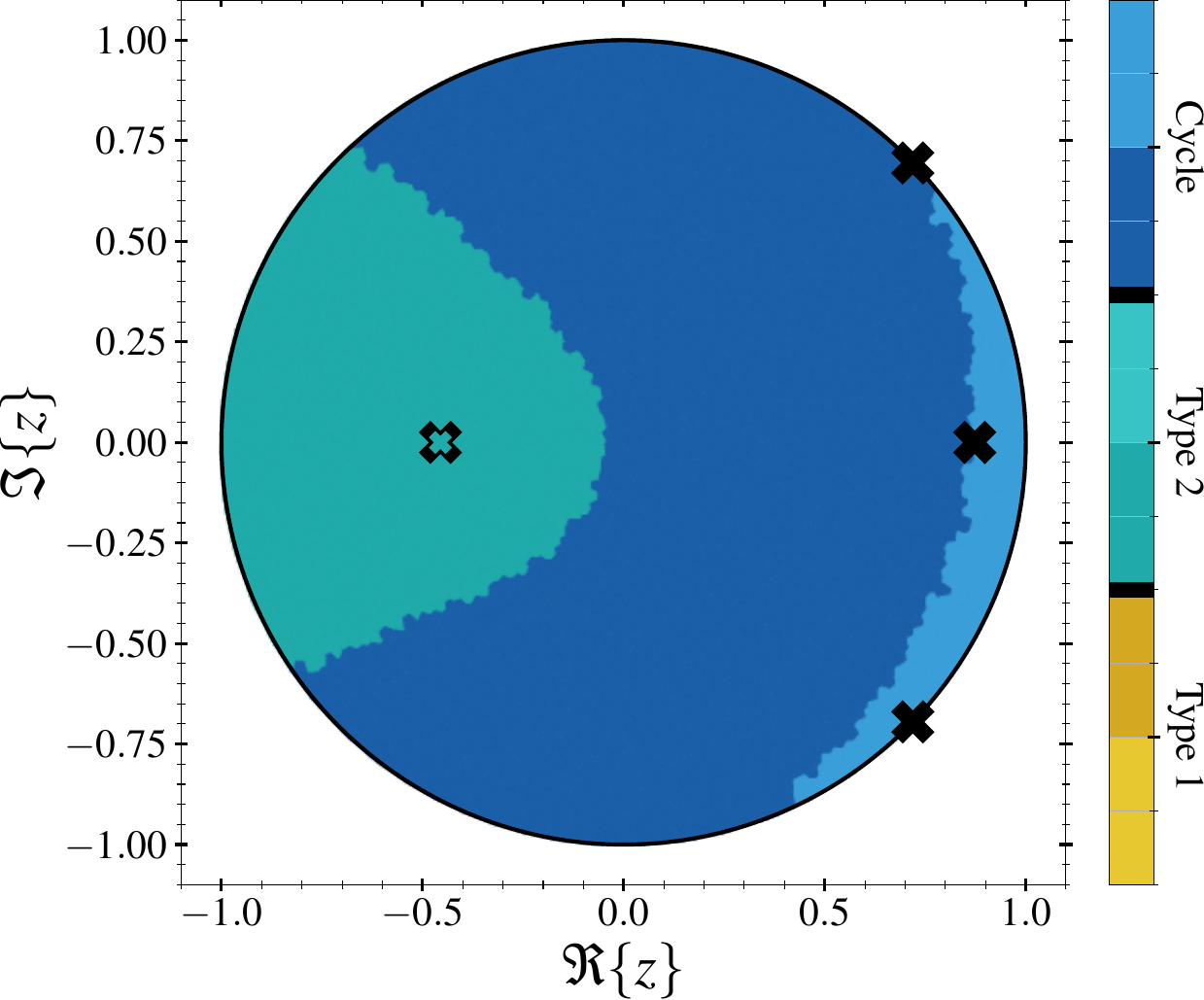}
\includegraphics[width=\linewidth]{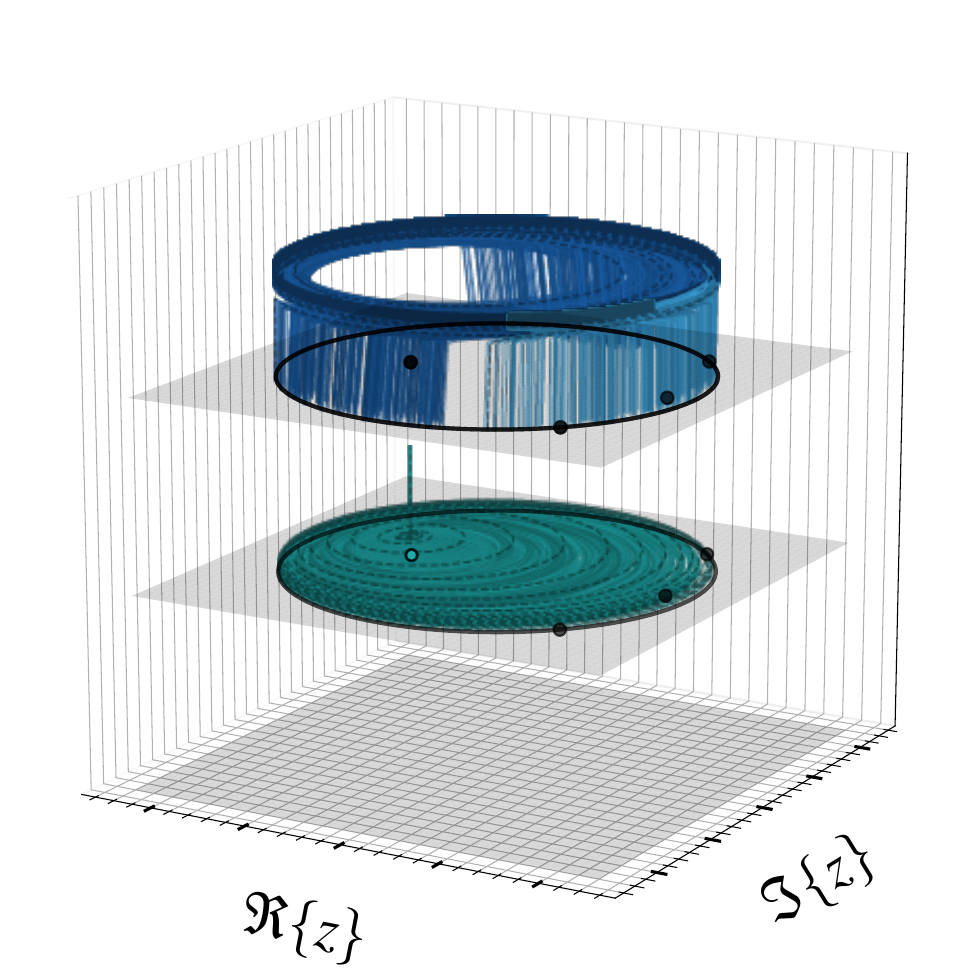}
\caption{$\tau = 1.4$}
\end{subfigure}
\begin{subfigure}{\localcolumnwidth}
\includegraphics[width=\linewidth]{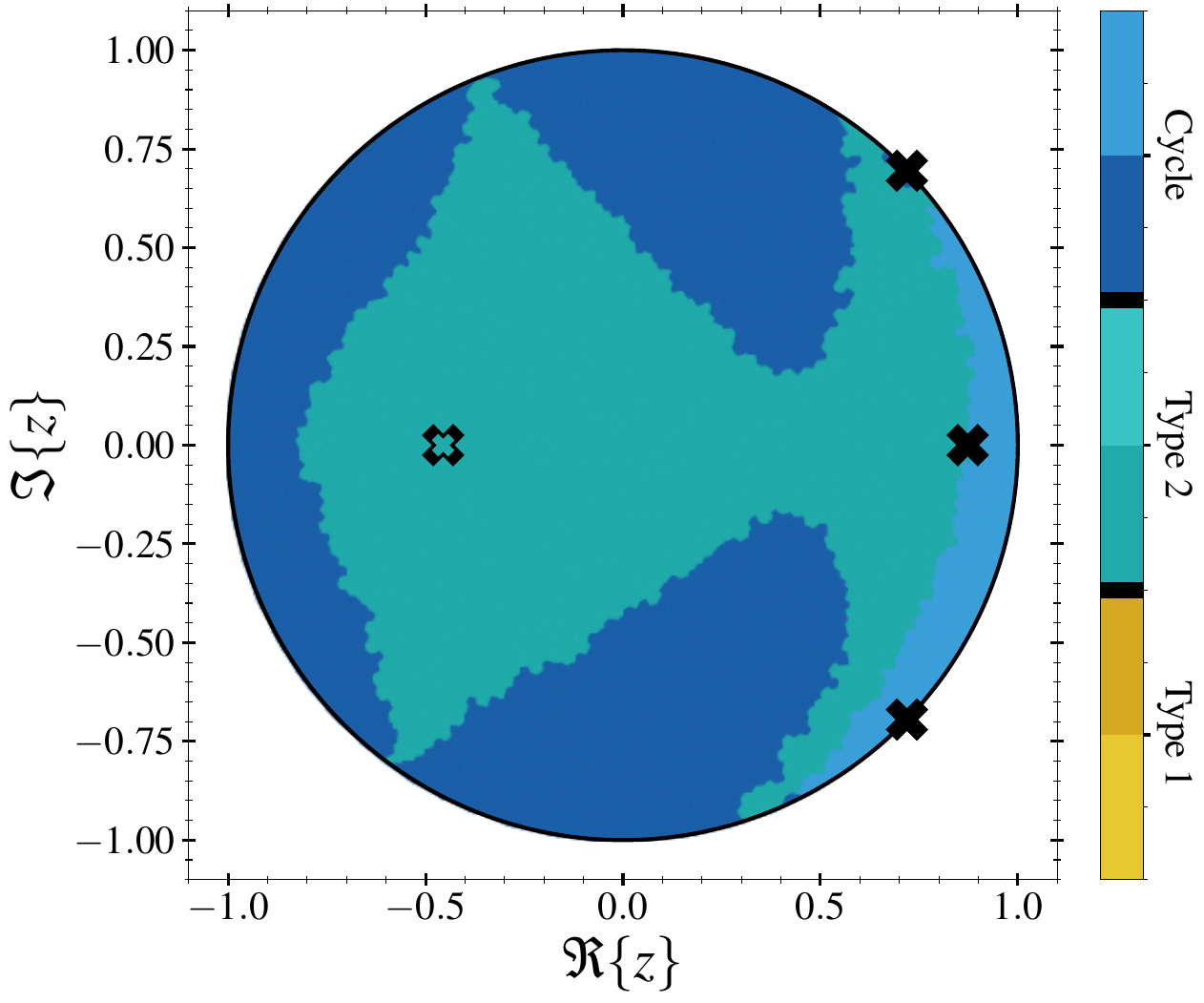}
\includegraphics[width=\linewidth]{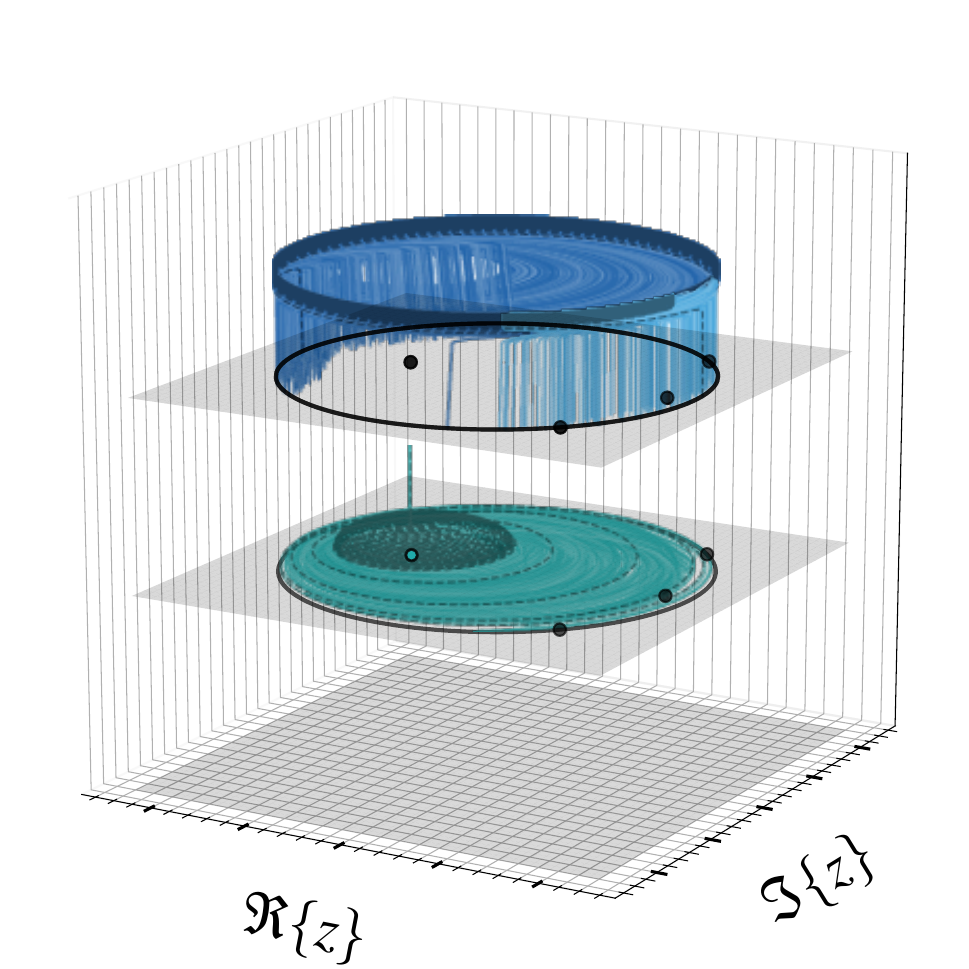}
\caption{$\tau = 1.55$}
\end{subfigure}
\begin{subfigure}{\localcolumnwidth}
\includegraphics[width=\linewidth]{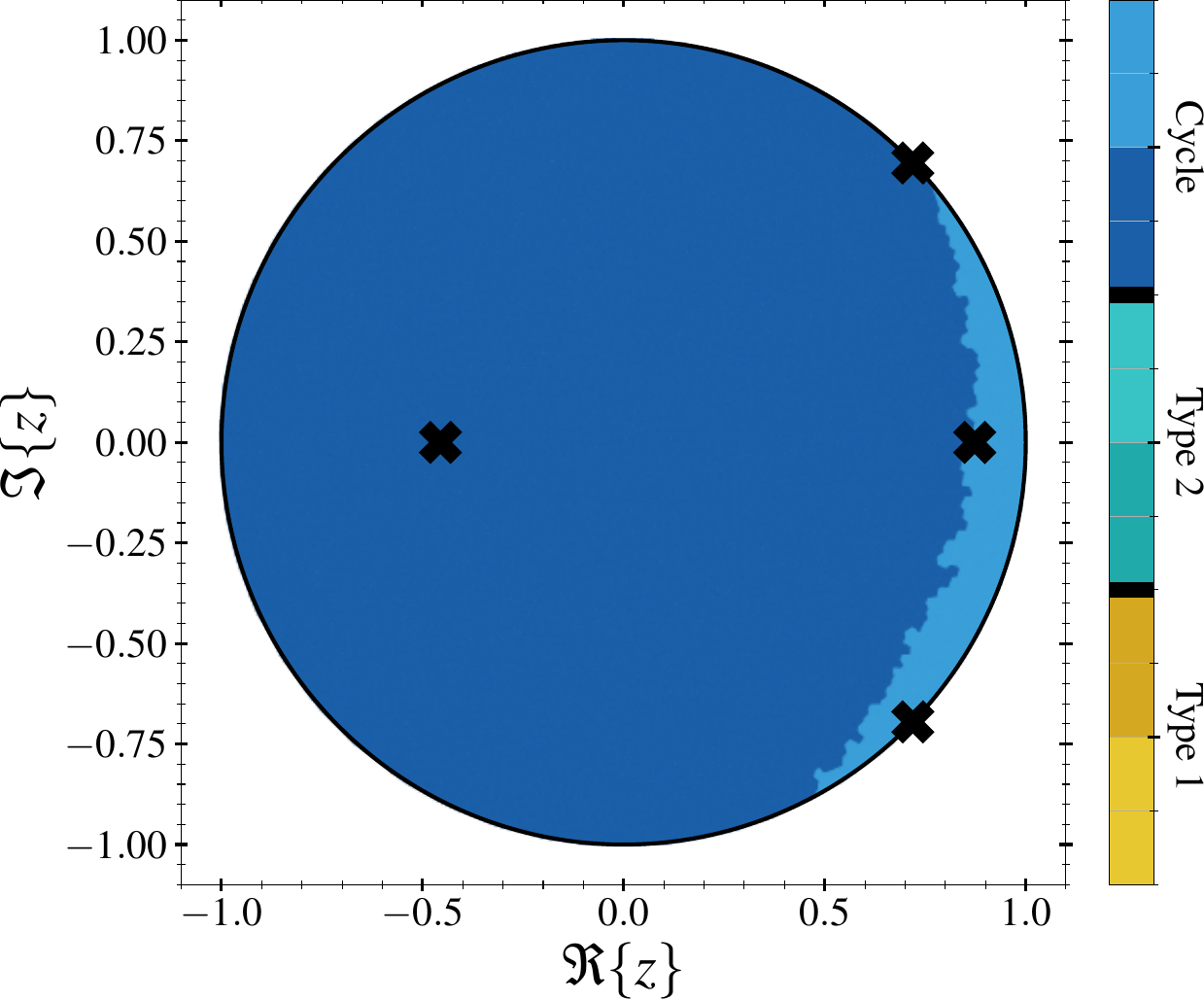}
\includegraphics[width=\linewidth]{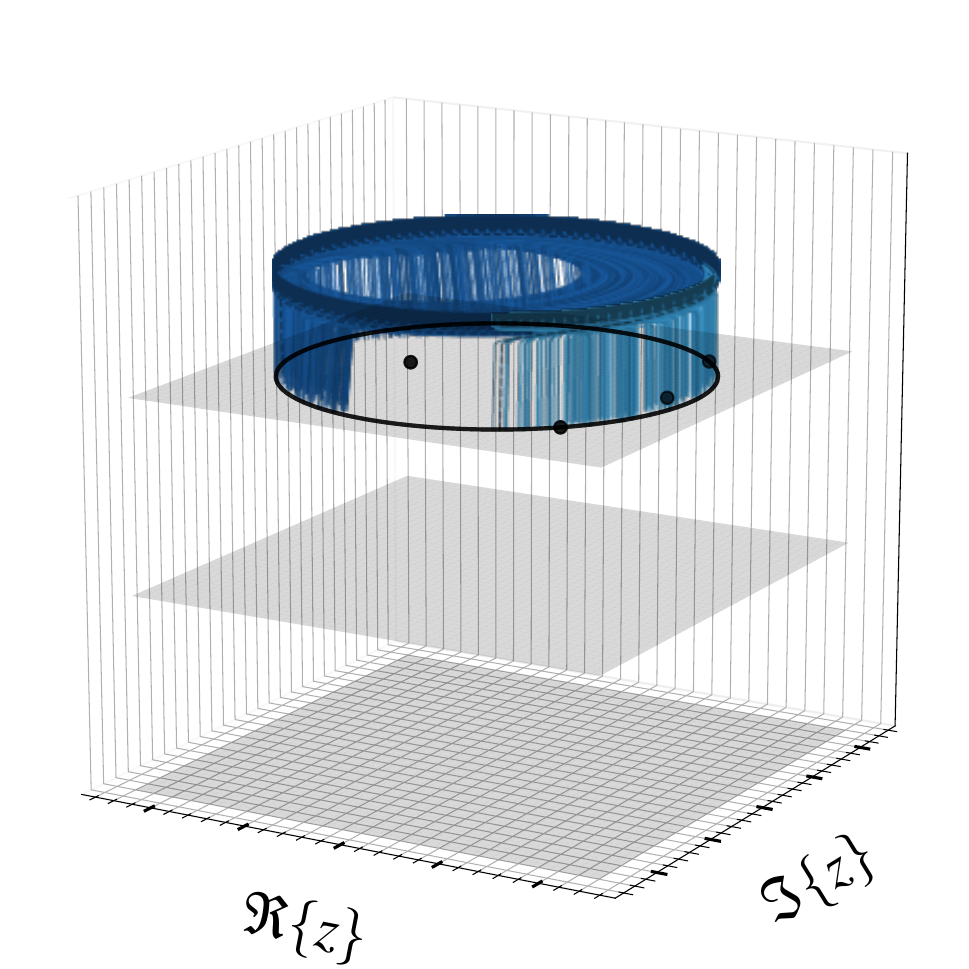}
\caption{$\tau = 2$}
\end{subfigure}
\caption{Set 4: (top) Regions of attraction and (bottom) trajectories for $(\kappa,
\eta) = (3, -0.4)$. The trajectories are shown for all types of behavior of interest:
(top) limit cycles; (middle) type 2 equilibria; and (bottom) type 1 equilibria. }
\label{fig:num.set4}
\end{figure}

\newpage
\section{Conclusions}
\label{sec:conclusions}

We analyzed the local dynamics of an infinite all-to-all coupled network of identical
theta neurons with delayed synaptic interaction. Using the Watanabe--Strogatz reduction,
together with the thermodynamic-limit assumption of uniformly distributed constants of
motion, the network dynamics were reduced to a complex delayed equation or, equivalently, to
a two-variable macroscopic system with memory. The equilibria of this reduced system are
made up of two geometrically distinct families: type 1 equilibria on the unit circle and
type 2 equilibria on the real axis.

For the type 1 family, we showed that equilibria on the upper semicircle are always
unstable and those on the lower semicircle have delay-independent stability when $\kappa
< 0$. Delay-induced destabilization of a type 1 equilibrium can occur only for $\kappa >
0$ and $\eta < 0$. In the discrete-delay case, this destabilization occurs through a
Hopf bifurcation at an explicitly computable critical delay. The numerical evaluation of
the first Lyapunov coefficient in the considered parameter region indicates that this Hopf
bifurcation is supercritical.

For the type 2 family, the delay-free saddle remains unstable for all admissible delay
kernels. The delay-free center is more sensitive to the delay: for small positive delays
it becomes asymptotically stable when $\kappa>0$ and unstable when $\kappa<0$. For the
weak and strong Gamma kernels considered here, the type 2 Hopf condition is not
satisfied in the chosen parameter range. In contrast, the Dirac kernel admits an
infinite sequence of Hopf thresholds with alternating crossing directions and,
therefore, allows stability switching of the type 2 branch.

The numerical simulations for the discrete-delay system support the local analysis. They
illustrate stability switching of a single type 2 equilibrium in the first
quadrant of the
$(\kappa,\eta)$ parameter plane, attracting periodic dynamics in a second-quadrant region where the local
equilibrium classification alone does not detect a stable branch, delay-independent stability in the third quadrant, and interaction between type 1 and type 2 Hopf mechanisms in the
fourth quadrant. These simulations also
show that the effect of the delay is strongly parameter-dependent: in some regimes the
phase portrait changes only slightly over the sampled delays, while in others the delay
produces transitions between stable equilibria, coexistence regimes, and stable periodic
orbits.

The stability results obtained for the reduced system should be interpreted as
macroscopic stability results for the original infinite network of theta neurons. Under the
Watanabe--Strogatz reduction, and after passing to the thermodynamic limit with
uniformly distributed constants of motion, the phase distribution of the
population is parametrized by the complex order parameter \(z(t)\). Hence an
equilibrium \(z^\star\) of the reduced delay differential equation represents a stationary
population state of the original network: type~1 equilibria on \(|z|=1\)
correspond to fully synchronized states, while type~2 equilibria
\(z^\star\in(-1,1)\) correspond to nonsynchronized stationary phase
distributions. Local stability, instability, and Hopf bifurcation of
\(z^\star\) therefore describe the local stability, instability, and emergence
of collective oscillations of the associated macroscopic population state.
These conclusions hold on the Watanabe--Strogatz/Ott--Antonsen invariant family selected by the
uniform distribution of constants of motion. They should not be interpreted as
global stability statements for arbitrary phase distributions or arbitrary
perturbations of the full infinite-dimensional network.

Natural directions for future work include extending the present analysis to
nonuniform constants of motion, heterogeneous theta-neuron populations, finite-size
effects, and broader classes of delay kernels.

\noindent \textbf{Data Availability}.
The software developed and used in this study is openly available in the Zenodo repository at
\href{https://doi.org/10.5281/zenodo.20718310}{https://doi.org/10.5281/zenodo.20718310}.
This archive (version 2026.1) includes all source code and documentation required to
reproduce the results, and is released under the MIT License.

\newpage
\appendix

\section{Proof of \Cref{prop:general}}
\label{ax:prop.general.proof}

\begin{proof}
\leavevmode
\begin{enumerate}
\item[i.]
From~\Cref{rem:laplace.admissible}, we have that $L(0) = 1$, so $\Delta_p(0, \tau) = a - b$.
Hence, $s = 0$ is a root if and only if $a = b$. Furthermore, we have that
\[
\pd{}{s} \Delta_p(s,\tau)=p s^{p-1}-b\tau L'(s\tau).
\]

From~\Cref{rem:laplace.admissible}, we also have that $L'(0) = -1$, so it follows that
\[
\pd{}{s} \Delta_p(0, \tau) =
\begin{cases}
1 + b \tau, & p=1, \\
b \tau, & p \ge 2.
\end{cases}
\]

This gives the simplicity conditions stated in the proposition for the root $s=0$.

\item[ii.]
Let \(f_\tau(x) \coloneq \Delta_p(x, \tau)\) restricted to $x \in [0, \infty)$. As
\(L\) is continuous on \([0,\infty)\), the function \(f_\tau\) is also continuous on
$[0, \infty)$ and, under the assumption $a < b$, satisfies \(f_\tau(0) = a - b < 0\).
Moreover, because \(0<L(x)\le 1\), it follows that \(f_\tau(x)\to+\infty\), as
$x\to\infty$. Hence, by the intermediate value theorem, \(f_\tau\) has at least one
positive root.

\item[iii.]
Consider the case of the Dirac kernel, where $L(s) = e^{-s}$. If $p = 1$, setting
$s = \ii \omega$, with $\omega > 0$ in~\eqref{eq:char.gen} gives
\begin{equation} \label{eq:prop.general.case3}
\Delta_1(\ii \omega, \tau) = \ii \omega + a - b e^{-\ii \omega \tau} = 0
\implies
\ii \omega + a=b e^{-\ii \omega\tau}.
\end{equation}

Taking the absolute value on both sides gives $\omega^2 = b^2 - a^2$. Hence, purely
imaginary roots exist if and only if $|a| < |b|$, in which case $\omega = \sqrt{b^2 - a^2}$.
Since $a > b$, this implies $b < 0$. Therefore, from~\eqref{eq:prop.general.case3}, we have that
\[
\cos(\omega\tau)=\frac{a}{b},
\qquad \text{and} \qquad
\sin(\omega\tau)=-\frac{\omega}{b} > 0,
\]
which gives the critical delays
\[
\tau_n = \frac{\arccos(a/b) + 2 n \pi}{\sqrt{b^2 - a^2}},
\qquad n \in \{0, 1, \dots\}.
\]

Then, if we consider $s(\tau)$ as a function of the delay, differentiating
$\Delta_1(s(\tau), \tau) = 0$ with respect to $\tau$ gives
\[
s'(\tau) = -\frac{b s e^{-s \tau}}{1+b\tau e^{-s\tau}}.
\]

Evaluating the condition above at $(\ii \omega, \tau_n)$, we obtain the transversality
condition
\[
\Re s'(\tau_n) = \frac{\omega^2}{(1+a\tau_n)^2+(\omega\tau_n)^2}>0.
\]

For the second part, we consider $p = 2$ and $s = \ii \omega$, with $\omega > 0$. Substituting
into~\eqref{eq:char.gen} gives $-\omega^2+a=be^{-\ii \omega\tau}$. Since $b \ne 0$ and
the left-hand side is real, we must have that
\[
\omega \tau = n \pi
\implies \omega_n^2 = a - (-1)^n b,
\]
for all $n \in \{1, 2, \dots\}$. Thus, purely imaginary roots exist exactly when \(a - (-1)^n b > 0\).
In that case, the critical delays are given by
\[
\tau_n = \frac{n\pi}{\sqrt{a - (-1)^n b}},
\qquad n\in \{1, 2, \dots\}.
\]

As before, we can differentiate $\Delta_2(s(\tau), \tau) = 0$ with respect to $\tau$
and evaluate at $(\ii \omega, \tau_n)$. This gives the transversality condition
\[
\Re s'(\tau_n) = -\frac{2b\,\omega_n^2(-1)^n}{(b\tau_n)^2+4\omega_n^2},
\]
which leads to \(\sign (\Re s'(\tau_n)) = (-1)^{n + 1} \sign (b)\), for \(n \in \{1, 2, \dots\}\).
\end{enumerate}
\end{proof}

\section{Dirac kernel: existence of a periodic orbit on the unit circle}
\label{ax:dirac.periodic.orbit}

\begin{proposition}
\label{prop:unit.circle.periodic.orbit}
Assume that $l_\tau(t) = \delta(t - \tau)$, with $\tau \ge 0$, is a Dirac delay kernel and
\begin{equation}\label{eq:uniform.positive.speed}
\kappa<0,\qquad\eta+4\kappa>0.
\end{equation}

Then, there exists at least one periodic solution on the invariant circle \(\rho=1\).
\end{proposition}

\begin{proof}
Since $\dot{\rho}|_{\rho = 1} = 0$ by~\eqref{eq:sys.2d}, the circle $\rho = 1$ is invariant.
On the invariant circle, the phase equation from~\eqref{eq:sys.2d} reduces to
\begin{equation}\label{eq:phase.unit.circle.dirac}
\dot\phi(t)=F(\phi(t), \phi(t - \tau)) \coloneq{}
[\eta + \kappa(1-\cos \phi(t - \tau))^2 - 1] (1 + \cos \phi(t)) + 2.
\end{equation}

We first establish uniform bounds on the right-hand side $F$. Since $\kappa <
0$ by~\eqref{eq:uniform.positive.speed} and $(1 - \cos \phi(t - \tau))^2 \in [0, 4]$, a
direct computation gives
\[
2 \min\{1, \eta + 4 \kappa\} \coloneq m \le
F(\phi(t), \phi(t - \tau)) \le
M \coloneq 2 \max\{1, \eta\},
\]
where $m> 0$ by~\eqref{eq:uniform.positive.speed}. Therefore,
every solution of~\eqref{eq:phase.unit.circle.dirac} satisfies, for all $t$,
\begin{equation}\label{eq:phase.speed.bounds}
m \le \dot\phi(t) \le M.
\end{equation}

Let \(X\subset C([-\tau,0],\mathbb{R})\) be the set of continuous functions \(\psi\),
such that \(\psi(0)=0\) and
\[
m(\theta_2-\theta_1)\le \psi(\theta_2)-\psi(\theta_1)\le M(\theta_2-\theta_1),
\]
for all $-\tau \le \theta_1 \le \theta_2 \le 0$. The set \(X\) is convex, closed,
bounded, and equicontinuous. Therefore, it is compact in \(C([-\tau,0], \mathbb{R})\) by
the Arzelà--Ascoli Theorem.

Let \(\psi\in X\), and let \(\phi(\cdot;\psi)\) denote the unique
solution~\cite{HaleVerduynLunel1993} of \eqref{eq:phase.unit.circle.dirac} with the
initial history \(\phi(\theta) = \psi\), \(\theta\in[-\tau,0]\). By
\eqref{eq:phase.speed.bounds}, \(\phi(\cdot;\psi)\) is strictly increasing. Therefore,
there exists a unique return time \(T(\psi)>0\), such that
\[
\phi(T(\psi); \psi) = 2 \pi.
\]

Moreover, again by \eqref{eq:phase.speed.bounds},
\[
\frac{2\pi}{M} \le T(\psi) \le \frac{2\pi}{m}.
\]

We define the Poincaré map \(P: X \to C([-\tau, 0],\mathbb{R})\) by
\[
(P \psi)(\theta):=\phi(T(\psi)+\theta;\psi)-2\pi,
\]
which satisfies $(P \psi)(0) = 0$, since \(\phi(T(\psi);\psi)=2\pi\). Moreover,
\[
(P\psi)(\theta_2)-(P\psi)(\theta_1)
=
\phi(T(\psi)+\theta_2;\psi)-\phi(T(\psi)+\theta_1;\psi),
\]
for all \(-\tau\le \theta_1\le \theta_2\le 0\). Finally, by integrating \eqref{eq:phase.speed.bounds},
we obtain
\[
m(\theta_2-\theta_1)\le (P\psi)(\theta_2)-(P\psi)(\theta_1)\le M(\theta_2-\theta_1).
\]

Hence \(P\psi\in X\), i.e., \(P: X \to X\). The map $P$ is continuous: both the solution
$\phi$ and the return time $T$ depend continuously on the initial history $\psi$. The
continuity of $\phi$ with respect to the initial condition is a standard
result~\cite{HaleVerduynLunel1993}. Continuity of $T$ is a consequence
of~\eqref{eq:phase.speed.bounds}, where $\dot{\phi} \ge m > 0$, and the implicit
function theorem. Since \(X\) is compact and convex, Schauder's fixed-point theorem
implies that there exists a fixed point \(\psi_\ast\in X\) such that
\[
P\psi_\ast = \psi_\ast.
\]

Let \(\phi_\ast\) denote the corresponding solution with return time \(T_\ast=T(\psi_\ast)\).
Then,
\[
\phi_\ast(T_\ast+\theta)=\phi_\ast(\theta)+2\pi
\qquad\text{for all }\theta\in[-\tau,0].
\]

Let $u(t) \coloneq \phi_\ast(t+T_\ast)-2\pi$. As the equation
\eqref{eq:phase.unit.circle.dirac} is autonomous and \(2\pi\)-periodic in each phase
argument, the function \(u\) satisfies the same delay equation as \(\phi_\ast\).
Moreover, \(u\) and \(\phi_\ast\) have the same history on \([-\tau,0]\). By uniqueness
of solutions,
\[
\phi_\ast(t+T_\ast)=\phi_\ast(t)+2\pi
\qquad\text{for all }t\ge -\tau.
\]

Finally, let \(z_\ast(t) \coloneq e^{\ii \phi_\ast(t)}\). By the above equation $z_\ast$ is
a \(T_\ast\)-periodic function that satisfies \(|z_\ast(t)| = 1\), for all \(t\). Hence,
$z_\ast$ solves~\eqref{eq:sys.2d} with $\rho = 1$, and is the desired solution on the
unit circle.
\end{proof}

\bibliographystyle{siamplain}
\bibliography{bibliography}

\end{document}